\documentclass[11pt]{article}

\usepackage[letterpaper, hmargin=1in, top=1in, bottom=1in, footskip=0.6in]{geometry}

\usepackage{titlesec}
\usepackage{tikz}
\usetikzlibrary{arrows, positioning, shapes}
\titleformat*{\section}{\Large\bfseries}
\titleformat{\subsection}[runin]{\large\bfseries}{\thesubsection.}{.5em}{}[.]\titlespacing{\subsection}{0pt}{2ex plus .1ex minus .2ex}{.8em}
\titleformat{\subsubsection}[runin]{\normalfont\bf}{\thesubsubsection.}{.3em}{}[.]\titlespacing{\subsubsection}{0pt}{1ex plus .1ex minus .2ex}{.5em}
\titleformat{\paragraph}[runin]{\normalfont\itshape}{\theparagraph.}{.3em}{}[.]\titlespacing{\paragraph}{0pt}{1ex plus .1ex minus .2ex}{.5em}

\usepackage{amsmath}
\usepackage{amssymb}
\usepackage{amsfonts}
\usepackage{latexsym}
\usepackage{amsthm}
\usepackage{amsxtra}
\usepackage{amscd}
\usepackage{bbm}
\usepackage{mathrsfs}
\usepackage{bm}
\usepackage{tensor}

\usepackage{graphicx, color}

\definecolor{darkred}{rgb}{0.9,0,0.3}
\definecolor{darkblue}{rgb}{0,0.3,0.9}

\definecolor{vdarkred}{rgb}{0.7,0,0.2}
\definecolor{vdarkblue}{rgb}{0,0.2,0.7}

\usepackage[pdftex, colorlinks, linkcolor=vdarkblue,urlcolor=black,citecolor=vdarkred]{hyperref}

\usepackage[nottoc,notlof,notlot]{tocbibind}
\usepackage{cite}

\numberwithin{equation}{section}
\numberwithin{figure}{section}

\theoremstyle{plain} 
\newtheorem{theorem}{Theorem}[section]
\newtheorem*{theorem*}{Theorem}
\newtheorem{lemma}[theorem]{Lemma}
\newtheorem*{lemma*}{Lemma}
\newtheorem{corollary}[theorem]{Corollary}
\newtheorem*{corollary*}{Corollary}
\newtheorem{proposition}[theorem]{Proposition}
\newtheorem*{proposition*}{Proposition}

\newtheorem*{conjecture*}{Conjecture}

\theoremstyle{definition} 

\newtheorem*{definition*}{Definition}

\newtheorem*{example*}{Example}
\newtheorem{remark}[theorem]{Remark}
\newtheorem*{remark*}{Remark}

\newtheorem*{assumption*}{Assumption}

\renewcommand{\b}[1]{\boldsymbol{\mathrm{#1}}} 
\newcommand{\bb}{\mathbb} 
\renewcommand{\cal}{\mathcal}

\newcommand{\ul}[1]{\underline{#1} \!\,} 
\newcommand{\ol}[1]{\overline{#1} \!\,} 

\newcommand{\E}{\mathbb{E}}
\newcommand{\R}{\mathbb{R}}
\newcommand{\C}{\mathbb{C}}
\newcommand{\N}{\mathbb{N}}

\newcommand{\e}{\mathrm{e}}

\newcommand{\G}{\widetilde{G}}

\renewcommand{\i}{\hat{\imath}}
\renewcommand{\j}{\hat{\jmath}}
\renewcommand{\k}{\hat{k}}
\renewcommand{\l}{\hat{l}}
\newcommand{\ii}{\mathrm{i}}
\newcommand{\dd}{\mathrm{d}}

\newcommand*{\deq}{\mathrel{\vcenter{\baselineskip0.65ex \lineskiplimit0pt \hbox{.}\hbox{.}}}=}
\newcommand*{\eqd}{=\mathrel{\vcenter{\baselineskip0.65ex \lineskiplimit0pt \hbox{.}\hbox{.}}}}

\renewcommand{\leq}{\leqslant}
\renewcommand{\le}{\leqslant}
\renewcommand{\geq}{\geqslant}

\renewcommand{\epsilon}{\varepsilon}
\newcommand{\psum}{\sideset{}{^*}\sum}

\DeclareMathOperator{\tr}{Tr}
\DeclareMathOperator{\var}{Var}
\DeclareMathOperator{\supp}{supp}

\DeclareMathOperator{\im}{Im}

\newcommand*{\rom}[1]{\expandafter\@slowromancap\romannumeral #1@}

\title{\bf \Large Edge universality of sparse Erd\H{o}s-R\'{e}nyi digraphs\vspace{0.5em}}

\author{Yukun He\footnote{Shanghai Center for Mathematical Sciences, Fudan University.
		 Email:	\href{mailto:heyukun@fudan.edu.cn}{heyukun@fudan.edu.cn}. Supported by National Key R\&D Program of China
		No.\,2023YFA1010400 and NSFC No.\,12322121.}
	\vspace{1em}
}

\begin{document}
\maketitle

\begin{abstract}
\vspace{0.2cm}
Let $\cal A$ be the adjacency matrix of the Erd\H{o}s-R\'{e}nyi directed graph $\mathscr G(N,p)$. We denote the eigenvalues of $\cal A$ by $\lambda_1^{\cal A},...,\lambda^{\cal A}_N$, with $|\lambda_1^{\cal A}|=\max_i|\lambda_i^{\cal A}|$. For $N^{-1+o(1)}\leq p\leq 1/2$, we show that
\[
\max_{i=2,3,...,N} \bigg|\frac{\lambda_i^{\cal A}}{\sqrt{Np(1-p)}}\bigg| =1+O(N^{-1/2+o(1)})
\]
with very high probability. In addition, we prove that near the unit circle, the local eigenvalue statistics of $\cal A/\sqrt{Np(1-p)}$  coincide with those of the real Ginibre ensemble. As a by-product, we also show that all non-trivial eigenvectors of $\cal A$ are completely delocalized.

For Hermitian random matrices, it is known that the edge statistics are sensitive to the sparsity: in the very sparse regime, one needs to remove many noise random variables (which affect both the mean and the fluctuation) to recover the Tracy-Widom distribution \cite{EKYY1,EKYY2,LS1,HLY,HK20,Lee21,HY22}. Our results imply that, compared to their analogues in the Hermitian case, the edge statistics of non-Hermitian sparse random matrices are more robust.

\end{abstract}

\section{Introduction}

In this paper, we consider the Erd\H{o}s-R\'enyi directed graph $\mathscr G(N,p)$, i.e.\,a directed graph on $N$ vertices, and each edge is included in the graph with probability $p$, independently from every other edge. We denote the adjacency matrix of $\mathscr G(N,p)$ by $\cal A$. It is easy to see that $\cal A\in \{0,1\}^{N\times N}$ is a random matrix with independent entries satisfying 
\[
\cal A_{ij}=\begin{cases}
	1 \quad \mbox{with probability $p$}\\
	0 \quad \mbox{with probability $1-p$} 
\end{cases}
\]
for all $i,j$. For numerical convenience, we introduce the normalized matrix
\begin{equation}  \label{1.1}
A=\sqrt{\frac{1}{Np(1-p)}}\,\cal A
\end{equation}
so that $\var({A}_{ij})=N^{-1}$. From the circular law\cite{Bai97,TV10,wood,RT19,BR19,GT10}, we know that when $pN\gg 1$, the limiting spectral density of ${A}$ converges to the uniform law on the unit disc of the complex plane.

One of the most important questions for the Erd\H{o}s-R\'enyi ensemble is to study its extreme
eigenvalue statistics. Since the entries of the adjacency matrix have positive expectations, its largest eigenvalue (in magnitude) is very large and far away from the rest of the spectrum. We are therefore interested in the probability distribution of the eigenvalues of A near the unit disc.

The matrix ${A}$ has typically $N^2p$ nonzero entries, and in the regime $p\asymp1$, ${A}$ is a \emph{dense matrix}, as a nontrivial portion of its entries are not zero. Under the four moment matching condition, it was proved in \cite{TV5} that the local statistics of a dense non-Hermitian random matrix coincide with those of the Ginibre ensemble. This is known as the \emph{universality of non-Hermitian random matrices}. Without the four moment matching condition, the local circular law was proved for matrices with uniform variance profile in\cite{BYY14,BYY142,Y1}, and for matrices with general variance profile in \cite{AEK18}. The spectral radius of non-Hermitian random matrices was determined in \cite{AEK5}. Near the spectral edge, the universality of non-Hermitian random matrices was established in \cite{CEK3}.

In the regime $p\ll 1$, which is more interesting in terms of graphs, the majority of the entries of ${A}$ are 0. In other words, ${A}$ is a \emph{sparse matrix}. For sparse non-Hermitian random matrices, there is by far no result on the local eigenvalue statistics. In this paper, we prove the edge universality for ${A}$ in the whole regime $N^{-1+o(1)}\leq p\leq 1/2$. 

For a square matrix $S\in \bb C^{N\times N}$ with eigenvalues $\lambda^S_1,...,\lambda^S_N$, we define its $k$-point correlation function $p_{k}^S$ through
\begin{equation} \label{1.2}
	\int_{\bb C^k} F(z_1,...,z_k) p_{k}^S(z_1,...,z_k) \dd z_1 \, \dots \dd z_k= {N \choose k}^{-1}\bb E \psum_{i_1,...,i_k=1}^{N} F(\lambda^S_{i_1},...,\lambda^S_{i_k})\,,
\end{equation}
for any smooth compactly supported $F: \bb C^k \to \bb C$, and $\sum^*$ is shorthand for distinct sum. For the real Ginibre ensemble $W$, and $\b w=(w_1,...,w_k)$, $\b z=(z_1,...,z_k)\in \bb C^k$, the microscopic scaling limit of $p_k^W$ exists such that
\[
\lim_{N\to \infty} p_k^W\Big(w_1+\frac{z_1}{N^{1/2}},...,w_k+\frac{z_k}{N^{1/2}}\Big)\eqd {\mathrm p}_{\b w}(\b z)\,.
\]
For the detailed formula of ${\mathrm p}_{\b w}(\b z)$, one may refer to \cite{BS2009}. We may now state our first main result.

\begin{theorem} \label{main theorem}
Fix $\tau>0$ and assume $p\in [N^{-1+\tau},1/2]$. Fix $k\in \bb N_+$. Let $\b w =(w_1,...,w_k)\in \bb C^k$ with $|w_1|=\cdots=|w_k|=1$. Then for any smooth compactly supported $F: \bb C^k \to \bb C$, we have
\[
\lim_{N \to \infty}   \int_{\bb C^k} F(\b z)\bigg[p_k^A\Big(\b w+\frac{\b z}{N^{1/2}}\Big)-\mathrm{p}_{\b w}(\b z)\bigg] \dd \b z=0\,.
\]
\end{theorem}

Theorem \ref{main theorem} shows that the edge universality of $\mathscr{G}(N,p)$ holds as long as the expected degree of each vertex has a polynomial growth in $N$. On the other hand, its Hermitian analogue is not true. Let $ A^H$ be the (rescaled) adjacency matrix of the undirected Erd\H{o}s-R\'enyi graph $\mathscr G^H(N,p)$ with normalization $\var(A^H_{ij})=N^{-1}$. In \cite{EKYY1,EKYY2,LS1}, it was proved that for $Np \geq N^{1/3+o(1)}$, the second largest eigenvalue of $A^H$ satisfies
\begin{equation}
	N^{2/3}(\lambda_2^H-\bb E \lambda_2^H)\overset{d}{\longrightarrow} \mathrm{TW}_1\,,
\end{equation}
where TW$_1$ is the Tracy-Widom distribution for GOE \cite{TW1,TW2}. For $N^{o(1)}\leq Np \leq N^{1/3-o(1)}$, it was shown in \cite{HLY,HK20} that 
\[
\sqrt{\frac{N^2p}{2}}(\lambda_2^H-\bb E \lambda_2^H)\overset{d}{\longrightarrow} \cal N(0,1)\,.
\]
In other words, there is a phase transition at $Np\asymp N^{1/3}$, and the edge universality fails if $Np \leq N^{1/3-o(1)}$. It was later observed in \cite{Lee21,HY22} that when $Np\asymp N^{\varepsilon}$, there are $\Omega(\e^{1/\varepsilon})$ number of noise random variables that outscale the Tracy-Widom distribution. These noise terms affect both the mean and fluctuation of $\lambda^H_2$. For small $\varepsilon$, there is by far no efficient way to calculate them explicitly, and even $\bb E \lambda_2^H$ is not precisely known.

In this paper, we say an event $\Omega$ holds with very high probability if for any $D>0$, $1-\bb P(\Omega)=O_D(N^{-D})$. Our second main result proves the optimal rigidity estimate of the spectral radius of $A$, as well as the complete delocalization of the eigenvectors.

\begin{theorem} \label{radius and delocalization}
Fix $\tau>0$ and assume $p\in [N^{-1+\tau},1/2]$. Let $\lambda_1,\lambda_2,...,\lambda_N$ be the eigenvalues of $A$ with $|\lambda_1|=\max_i |\lambda_i|$.
\begin{enumerate}
	\item For any fixed $\varepsilon>0$, we have
	\begin{equation} \label{rigidity}
	\max_{ 2\leq i \leq N} |\lambda_i|=1+O(N^{-1/2+\varepsilon})
	\end{equation}
with very high probability.	
\item Suppose $\b u\in \bb C^{N}$ satisfies $A \b u=\lambda\b  u$ for some $\lambda \in \bb C$ with $|\lambda|\leq 2$. Then for any fixed $\varepsilon>0$, we have $\|\b u\|_\infty =O (N^{-1/2+\varepsilon}\|\b u\|)$ with very high probability. 
\end{enumerate}

\end{theorem}

\begin{remark} \label{rmk3.1}
To simplify the presentation, we assume that all matrix elements of $A$ have identical variance
$1/N$. As in \cite{EYY1,EKYY1}, one may however easily generalize this condition and require that the diagonal elements of $A$
vanish. Thus one may for instance consider Erd\H{o}s-R\'enyi digraphs in which a vertex cannot link to itself.
\end{remark}

The main results imply that the edge statistics of Erd\H{o}s-R\'enyi digraphs are very robust: with the simple rescaling \eqref{1.1}, both the spectral radius and extreme eigenvalue fluctuations coincide with those of the real Ginibre ensemble. The phenomenon that non-Hermitian random matrices have more regular edge statistics than Hermitian matrices has also been observed in the literature. For instance, the convergence of spectral radius of non-Hermitian random matrices only requires the existence of the second moment \cite{BCG21,BCCT,BZ20}; in the Hermitian case, in order to have the convergence of the extreme eigenvalues, stronger conditions are needed both for sparse matrices \cite{ADK,BBK1,TY21} and for matrices with $\alpha$-stable entries \cite{LY}.

\subsection{Proof outline and new ideas}
\subsubsection*{The benefits of cusp singularity} 
Comparing to the dense case, the main obstacle in the sparse regime is the slow decay of the higher order terms. Our \textit{first key novelty} is turning another well-known obstacle in non-Hermitian matrices, cusp singularity, to our advantage. This completely eliminates the sparsity contribution near the edge.

More precisely, let $\widetilde{H}_w\in \bb C^{2N\times 2N}$ be the shifted Hermitization of $A$ defined in \eqref{Her} below. We study the spectrum $\{\lambda_1,...,\lambda_N\}$ of $A$ via Girko's Hermitization formula
\begin{equation} \label{1.5}
	\frac{1}{N}\sum_i f_{w_*}(\lambda_i)=\frac{\ii }{4\pi N}\int_{\bb C}  \int_0^{\infty} \nabla^2 f_{w_*}(w) \tr (\widetilde{H}_w-\ii \eta)^{-1} \dd \eta\, \dd^2 w\,,
\end{equation} 
where $f\in C_c^\infty(\bb C)$ is fixed, $f_{w_*}(w)\deq Nf(N^{1/2}(w-w_*))$ and $|w_*|=1$. The main step is thus to analyze the Stieltjes transform of $H_w$, namely
\[
\widetilde{g}\deq \frac{1}{2N} \tr \G\,, \quad \mbox{where} \quad\G\equiv \G_w(\ii \eta)\deq (\widetilde{H}_w-\ii\eta)^{-1}\,.
\]
The quantity $\widetilde{g}$ is expected to be close to a deterministic $m\equiv m(w,\ii \eta)$, where $m$ is the solution of 
\[
P(m)\deq m^3+2\ii \eta m^2+(1-\eta^2-|w|^2)m+\ii \eta=0
\] 
with $\im m>0$. It can be shown that
\[
\widetilde{g}-m=O\bigg(\frac{P(\widetilde{g})}{P'(m)}\bigg)\,,
\] 
and the key to understanding \eqref{1.5} is to get a good estimate of $P(\widetilde{g})$. Comparing to the bulk case, the local law near the unit circle is known to possess extra difficulties, due to the \textit{cusp singularity} \cite{EKS2020,AEK5}, i.e.
\begin{equation*} 
 P'(m)=3m^2+4\ii \eta m +(1-\eta^2-|w|^2)\asymp |1-|w||+\eta^{2/3}\,.
\end{equation*}
When $w$ is near the unit circle, the self-consistent equation is highly unstable, which requires a very precise bound of $P(\widetilde{g})$. The smallness of $P'(m)$ origins from that of $m$, i.e.\begin{equation} \label{smallm}
m=\ii \im m=O\big(|1-|w||^{1/2}+\eta^{1/3}\big)\,.
\end{equation}

Our observation here is that the cusp singularity, in particular \eqref{smallm}, can in fact help us on estimating higher order terms in the sparse regime. For instance, when we compute $\bb E P(\widetilde{g})$ via cumulant expansion (Lemma \ref{cumulant}), we get
\begin{equation} \label{cuexpand}
\bb E P(\widetilde{g})=O(N\cal C_4(A_{12})\bb E\widetilde{g}^3)+O(N\cal C_6(A_{12})\bb E\widetilde{g}^5)+\dots+\mbox{error}\,,
\end{equation}
where $\cal C_{k}$ denotes the $k$-th cumulant (here due to the algebraic structure of the Green function, the terms associated with odd cumulants are always small enough). Thanks to \eqref{smallm}, we have
\[
N\cal C_4(A_{12})\bb E\widetilde{g}^3\approx N\cal C_4(A_{12})m^3=O\bigg(\frac{|1-|w||^{3/2}+\eta}{Np}\bigg)=O\bigg(P'(m)\frac{|1-|w||^{1/2}+\eta^{1/3}}{Np}\bigg)\,.
\]
Thus this term will not affect the estimate of $\widetilde{g}-m$. In addition, the second term on RHS of \eqref{cuexpand} is even smaller, due to the increasing power of $\widetilde{g}$. The same type of smallness also occurs when we compute $\bb E |P(\widetilde{g})|^n$, which suggests that the fluctuation of $\widetilde{g}$ is also insensitive to the sparsity. As a result, it turns out that the cusp singularity is an \textit{advantage} rather than an \textit{obstacle} when studying sparse non-Hermitian matrices.

Let us make a comparison with the Hermitian case. We denote the Stieltjes transform of $A^H$ by $g^H\deq N^{-1} \tr (A^H-z)^{-1}$. It is known that $g^{H}$ can be approximated by the Stieltjes transform $m^H$ of the semicircle density, which satisfies
\[
P^H(m^H)\deq 1+zm^H+(m^H)^2=0 \,, \quad \im m^H>0\,.
\]
If we compute $\bb E P^H(g^H)$, we get
\[
\bb E P^H(g^H)=O(N\cal C_4(A^H_{12}))\bb E (g^H)^4)+O(N\cal C_6(A^H_{12})\bb E (g^H)^6)+\cdots+\mbox{error}\,.
\]
In the Hermitian case, the real part of the Stieltjes transform is no longer small near the spectral edge. Instead, we have $|m^H|\asymp 1$, and $N\cal C_4(A^H_{12}))\bb E (g^H)^4\asymp N^{-1}p^{-1}$, which is not negligible for small $p$. In fact, when $p$ is close to $N^{-1}$, we also observe nonomittable fluctuation terms, and they all needed to be added into $P^H$ to form a new self-consistent equation of $g^H$. We refer the readers to \cite{Lee21,HY22} for more details.

Another difficulty in the sparse regime is to utilize the contributions of (more than 2) off-diagonal entries of $\widetilde{G}$. To this end, we prove a generalized Ward identity (Lemma \ref{lemma4.4}), whose proof relies on the fact that the sum of Green functions preserves its form under differentiation (see \eqref{4q10}).

By exploring the cusp-singularity and performing careful estimates of $\bb E |P|^n$ (Proposition \ref{prop4.1}), it can be shown that near the unit circle, we have
\begin{equation} \label{1.9}
	|\widetilde{g}-m| \leq C N^{\varepsilon}\Big(\frac{1}{N\eta}+\frac{\eta^{1/3}}{Np}\Big)
\end{equation}
with very high probability. Thus for $\eta \leq N^{-3/4}$, we get the optimal estimate $|\widetilde{g}-m| =O (N^{-1+\varepsilon}\eta^{-1})$ with very high probability, regardless of the value of $p$.

\subsubsection*{Integration by-parts for the shift variable}
Our \textit{second key novelty} is the use of integration by-parts inside Girko's Hermitization, namely
\[
\int\int\nabla^2 f_{w_*}(w) \widetilde{g}\, \dd \eta \dd^2 w=-4\int\int \partial_{\bar{w}} f_{w_*}(w) \partial_w \widetilde{g} \,\dd \eta \dd^2w\,.
\]
This is a crucial step in our proof: \eqref{1.5} requires the understanding of $\widetilde{g}$ for all $\eta>0$, while the unimprovable bound \eqref{1.9} is  only sufficient for $\eta\leq N^{-3/4}$. In other words, estimates through original Hermitization will be too large for us. To this end, we write
\begin{equation} \label{1.10}
\begin{aligned}
&\,\int_{\bb C} \int_{N^{-3/4}}^{\infty} \nabla^2 f_{w_*}(w) \widetilde{g}\, \dd \eta \dd^2 w=-4\int_{\bb C} \int_{N^{-3/4}}^{\infty} \partial_{\bar{w}} f_{w_*}(w) \partial_w \widetilde{g} \,\dd \eta \dd^2w\\
=&\,\frac{2\ii }{N}\int_{\bb C}\int_{N^{-3/4}}^\infty \partial_{\bar{w}} f_{w_*} \sum_{i=1}^N \partial_\eta\G_{i+N,i} \,\dd \eta \dd^2 w=-\frac{2\ii}{N} \int_{\bb C}  \partial_{\bar{w}} f_{w_*} \sum_{i=1}^N \G_{i+N,i} (\ii N^{-3/4})\,\dd^2 w
\end{aligned}
\end{equation}
where in the second step we used $\partial_w \widetilde{g}=\frac{1}{2N}\sum_{i}(\G^2)_{i+N,i}=-\frac{\ii}{2N} \sum_{i}\partial_\eta \G_{i+N,i}$. We introduce \eqref{1.10} basing on two observations. Trivially, as $\|\nabla^2 f_{w_*}\|_1=O(N)$ and $\|\partial_{w}f_{w_*}\|_1=O(N^{1/2})$, the use of integration by parts improves the estimate by a factor of $N^{-1/2}$. In addition, we are able to prove that 
$$
\bigg|\frac{1}{N}\sum_{i=1}^N \G_{i+N,i} (\ii N^{-3/4})+\frac{1+m^2}{w}\bigg|\leq N^{-1/2+\varepsilon}\,,
$$
with very high probability (see Proposition \ref{prop5.1}). As a comparison, the trace only satisfies $ |\widetilde{g}(\ii N^{-3/4})-m|\leq N^{-1/4+\varepsilon}$.  In other words, after integration by parts, the Green function satisfies a stronger large deviation estimate. The combination these observations allows us to avoid the treatment of $\widetilde{g}-m$ at larger spectral scales. 

We remark that the idea of integration by parts for the shift variable inside Girko's Hermitization and estimating the derivatives of the Green function w.r.t.\,the shift is subsequently used by Cipolloni, Erd\H{o}s, and Xu in \cite{CEX23} (\texttt{arXiv:2312.08325}) to study the distribution of the spectral radius in the dense case. Interestingly, although the Gumbel proof in \cite{CEX23} has relied on only one integration by parts in the shift variable since the beginning, the introductions to v1--v3 present integrating by parts \textit{twice} as the ``first key new idea'' and a ``paradigm change'' in the analysis of Girko's formula \cite[arXiv \href{https://arxiv.org/pdf/2312.08325v1\#page=3}{\textcolor{vdarkblue}{v1}}, \href{https://arxiv.org/pdf/2312.08325v2\#page=3}{\textcolor{vdarkblue}{v2}}, \href{https://arxiv.org/pdf/2312.08325v3\#page=3}{\textcolor{vdarkblue}{v3}}, Section~1.3, p.~3]{CEX23}. In v1 of the non-Hermitian paper \cite{CEX23}, submitted eight months after the present paper, the authors cited three papers on sparse Hermitian matrices submitted in May and September 2023 (their [80], [13] and [76]) \cite[arXiv \href{https://arxiv.org/pdf/2312.08325v1\#page=2}{\textcolor{vdarkblue}{v1}}, p.~2]{CEX23}. The present paper on edge universality for sparse non-Hermitian matrices, submitted on 10~April 2023, was not mentioned.

Perhaps the incident most worth mentioning came in v2. The new footnote~1 in \cite[arXiv \href{https://arxiv.org/pdf/2312.08325v2\#page=3}{\textcolor{vdarkblue}{v2}}, Section~1.3, p.~3]{CEX23} cites \eqref{1.10} in the present paper for a single integration by parts in the shift variable and says that the resulting trace is ``still manageable without chiral symmetry''. However, the main text in the introduction still presented two integrations by parts in the shift variable as new, while the Gumbel proof used only one, leaving $\partial_{\bar z}f$ on the test function in (7.35) and (7.38) \cite[arXiv \href{https://arxiv.org/pdf/2312.08325v2\#page=39}{\textcolor{vdarkblue}{v2}}, pp.~\href{https://arxiv.org/pdf/2312.08325v2\#page=39}{\textcolor{vdarkblue}{39}}--\href{https://arxiv.org/pdf/2312.08325v2\#page=40}{\textcolor{vdarkblue}{40}}]{CEX23}. The two integrations by parts were worked out for a moment bound \cite[arXiv \href{https://arxiv.org/pdf/2312.08325v2\#page=44}{\textcolor{vdarkblue}{v2}}, (7.72)--(7.75)]{CEX23}, but were not used in the proof of the Gumbel law \cite[arXiv \href{https://arxiv.org/pdf/2312.08325v2\#page=15}{\textcolor{vdarkblue}{v2}}, p.~15]{CEX23}. The second integration by parts was evidently more essential to the novelty claim than to that proof. In v4, two years after v3, the introduction finally catches up with the proof, crediting this paper and stating that ``two integrations by parts is not used for our main result'' \cite[arXiv \href{https://arxiv.org/pdf/2312.08325v4\#page=4}{\textcolor{vdarkblue}{v4}}, Section~1.3, pp.~\href{https://arxiv.org/pdf/2312.08325v4\#page=3}{\textcolor{vdarkblue}{3}}--\href{https://arxiv.org/pdf/2312.08325v4\#page=4}{\textcolor{vdarkblue}{4}}]{CEX23}. However, v4 did not mention that the attributions in v1--v3 were inaccurate.

\begin{center}
\begin{minipage}{\linewidth}
\small
\noindent\textit{Timeline of arXiv submissions}\par\smallskip
\newcommand{\shiftversion}[3]{#1\newline\href{https://arxiv.org/pdf/#2\#page=#3}{\textcolor{vdarkblue}{\texttt{\footnotesize #2}}}}
\newcommand{\shiftcount}[2]{One.\newline#1 and #2.}
\newcommand{\tmark}{\leavevmode\smash{\llap{\makebox[1.5em][c]{\textcolor{vdarkblue}{\raisebox{.15ex}{$\scriptstyle\bullet$}}}}}}
\renewcommand{\arraystretch}{1}
\noindent\begin{tabular}{@{}p{.16\linewidth}@{\hspace{\dimexpr.035\linewidth-.75em-.275pt\relax}{\color{black!35}\vrule width .55pt}\hspace{\dimexpr.75em-.275pt\relax}}p{.515\linewidth}@{\hspace{.035\linewidth}}p{.245\linewidth}@{}}
\multicolumn{1}{@{}p{.16\linewidth}@{\hspace{.035\linewidth}}}{\raggedright\textbf{Submission}} &
\multicolumn{1}{@{}p{.515\linewidth}@{\hspace{.035\linewidth}}}{\raggedright\textbf{Integration by parts in Girko's formula}} &
\raggedright\textbf{Shift integrations by parts in the cited Gumbel estimates} \tabularnewline[5pt]
\multicolumn{1}{@{}p{.16\linewidth}@{\hspace{\dimexpr.035\linewidth-.75em-.275pt\relax}{\color{black!35}\vrule width .55pt height .45ex}\hspace{\dimexpr.75em-.275pt\relax}}}{\raggedright\shiftversion{10 Apr 2023}{2304.04723v1}{5}} &
\raggedright\tmark One integration by parts in the complex shift inside Girko's Hermitization, followed by explicit $\eta$-integration in (1.9) and (5.34) (now \eqref{5.28}). &
\raggedright --- \tabularnewline[5pt]
\raggedright\shiftversion{13 Dec 2023}{2312.08325v1}{3} &
\raggedright\tmark Two integrations by parts in the shift advertised as a ``major paradigm change''. The current paper uncited. &
\raggedright\shiftcount{(7.30)}{(7.33)} \tabularnewline[5pt]
\raggedright\shiftversion{25 Dec 2023}{2312.08325v2}{3} &
\raggedright\tmark Footnote~1 credits the single integration by parts in the shift. The text still claims two integrations by parts in the shift as the ``first key new idea''. The added reference on \href{https://arxiv.org/pdf/2312.08325v2\#page=4}{\textcolor{vdarkblue}{p.~4}} for the two integrations by parts leads to the expanded proof of a moment bound, which is not used in the proof of the main result. &
\raggedright\shiftcount{(7.35)}{(7.38)} \tabularnewline[5pt]
\raggedright\shiftversion{1 Feb 2024}{2312.08325v3}{3} &
\raggedright\tmark No changes with respect to these matters. &
\raggedright\shiftcount{(7.35)}{(7.38)} \tabularnewline[5pt]
\multicolumn{1}{@{}p{.16\linewidth}@{\hspace{\dimexpr.035\linewidth-.75em-.275pt\relax}{\color{black!35}\vrule width .55pt depth -.45ex}\hspace{\dimexpr.75em-.275pt\relax}}}{\raggedright\shiftversion{13 Feb 2026}{2312.08325v4}{3}} &
\raggedright\tmark The text credits this paper for the single integration by parts and places the second in the proof of moment convergence. The introduction finally catches up with the proof. However, it did not mention that the attributions in v1--v3 were inaccurate. &
\raggedright\shiftcount{(7.37)}{(7.44)}
\end{tabular}
\end{minipage}
\end{center}

\subsubsection*{The non-zero expectation}
Our \textit{third key novelty} is a self-similarity of the self-consistent equations of certain Green functions, which wipes out the effect of large expectations of the matrix entries.

Up to this point, we have not considered the fact that $A$ has positive expectations, and $\widetilde{H}_w$ is a rank-two perturbation of its centered version $H_w$ \eqref{Her}. As a result, there are in fact additional terms in the estimate of $P(\widetilde{g})$, e.g.
\begin{equation} \label{1.11}
\bb EA_{12}\cdot\sum_{i=1}^N \sum_{\alpha=N+1}^{2N}\G_{i\alpha}\,,
\end{equation}
(see also \eqref{5.23}). As $\bb EA$ is large, it is challenging to eliminate the effect of \eqref{1.11} for edge statistics. In fact, even for the Hermitian model (i.e.\,undirected Erd\H{o}s-R\'enyi graphs), the Tracy-Widom law was only obtained for the centered matrix when $p\leq N^{-2/3}$ \cite{HLY,HY22}.

To deal with this problem, we make use of the fact the rank-two perturbation is close to a projection, and this allows us to form new self-consistent equations for \eqref{1.11}, which yields the desired estimate for the expectation terms (see Proposition \ref{prop5.2}(ii) and Lemma \ref{lemma4.7}). The method presented here also applies to the Hermitian case.

\subsubsection*{Comparisons with Gaussian models} The above steps, together with the small ball probability estimate \cite{TV08}, allow one to prove Theorem \ref{radius and delocalization} as well as
\[
\frac{1}{N}\sum_i f_{w_*}(\lambda_i)-\frac{1}{\pi}\int_{|w|\leq 1} f_{w_*}(w) \dd^2w =O(N^{\varepsilon})
\]
with very high probability. To establish the edge universality, we need to study \eqref{1.5} near the critical regime $\eta \sim N^{-3/4}$ in more detail. To achieve this, we use the approach of Green function flow \cite{LS1,CEK3}. Here we again face the issue that $A$ is not centered. We solve it by using a two-step comparison. More precisely, let $W$ denote the real Ginibre ensemble. We first compare $\widetilde{H}_w$ and the Hermitization of  $W+\bb EA$ (Lemma \ref{lemma5.3}). We then compare the Hermitizations of $W$ and $W'=W+N\bb E A_{12} (1,0,...,0)^*(1,0,...,0)$, and conclude the proof with the fact that $W+\bb EA$ and $W'$ have the same distribution. In the comparison step we also make use of the isotropic estimate Lemma \ref{lemma4.7}.

\subsection*{Organization}

The paper is organized as follows. In Section \ref{sec4} we introduce the notations used in this paper. In Section \ref{sec3} we exploit the cusp fluctuation and prove strong local law for the centered model $H_w$ near the edge (Theorem \ref{thmHstrong}). In Section \ref{sec4.4}, for the non-centered matrix $\widetilde{H}_w$, we obtain entrywise local law in the whole spectrum (Theorem \ref{theoremA}) and strong local law outside the spectral domain (Proposition \ref{prop4.5}). These results also establish the upper bound in Theorem \ref{radius and delocalization} (i) as well as Theorem \ref{radius and delocalization} (ii). In Section \ref{sec5} we prove Theorem \ref{main theorem}. In addition, we prove a local law for $A$ near the edge (Theorem \ref{thmcircularlaw}), which yields the lower bound of Theorem \ref{radius and delocalization} (i).

\subsection*{Conventions} Unless stated otherwise, all quantities depend on the fundamental large parameter $N$, and we omit this dependence from our notation. We use the usual big $O$ notation $O(\cdot)$, and if the implicit constant depends on a parameter $\alpha$ we indicate it by writing $O_\alpha(\cdot)$. For random variables $X$ (possibly complex) and $Y \geq0$, we write $X \prec Y$, or equivalently $X=O_{\prec}(Y)$, if for any fixed $\varepsilon,D>0$,
\[
\bb P(|X|\geq Y N^{\varepsilon}) =O_{\varepsilon,D} (N^{-D})\,.
\]
We write $X\asymp Y$ if $X =O(Y)$ and $Y=O(X)$. We write $X \ll Y$ to mean $X=O_\varepsilon(YN^{-\varepsilon})$ for some fixed $\varepsilon>0$.

\subsection*{Acknowledgment}
The author would like to thank FromSoftware and Miyazaki Hidetaka for the hospitality in Yharnam.

\section{Notations and preliminaries} \label{sec4}
Let $A$ be defined as in \eqref{1.1}. For the rest of this paper, we use the parameters
\[
q\deq \sqrt{Np(1-p)} \quad \mbox{and} \quad \xi \deq \log_N 2q\,.
\]
We always assume $\xi\in [\tau/2,1/2]$ with $\tau$ defined as in Theorem \ref{main theorem}. We denote the centered adjacency matrix $B$ by $B\deq A-\bb EA$. It is easy to see that
$$
A = B + f \b e \b e^*\,,
$$
where $\b e \deq N^{-1/2}(1,1,\dots,1)^*$, and $f\asymp q$. We have $\var(B_{ij})=N^{-1}$ and 
\begin{equation} \label{2.1}
	\bb E |B_{ij}|^k=O_k(N^{-1}q^{-k+2})
\end{equation}
uniformly for all $i,j$ and $k\geq 3$. 

For $w \in \bb C$, we define the shifted Hermitizations of $B$ and $A$ by
\begin{equation} \label{Her}
\quad H_w\deq
\begin{pmatrix}
0&  B-w\\
B^*-\bar{w}&0
\end{pmatrix}\,, \quad \mbox{and} \quad \widetilde{H}_w\deq
\begin{pmatrix}
0&  A-w\\
A^*-\bar{w}&0
\end{pmatrix} \,.
\end{equation}
In addition, we abbreviate
\begin{equation} \label{H}
\quad H\deq H_0 \quad \mbox{and} \quad \kappa\deq ||w|-1|\,.
\end{equation}
 For $z=E+\ii \eta$ and $\eta>0$, we define the Green functions by
\begin{equation} \label{2.4}
G\equiv G_w(z)\deq  (H_w-z)^{-1}  \quad \mbox{and} \quad \widetilde{G}\equiv \widetilde{G}_w(z)\deq  (\widetilde{H}_w-z)^{-1}\,.
\end{equation}
We have the resolvent identities
\[
I+G(z-H_w)=I+(z-H_w)G=0 \quad \mbox{and} \quad I+\G(z-\widetilde{H}_w)=I+(z-\widetilde{H}_w)\G=0\,.
\]

In the sequel, we use the convention that the indices satisfy
\[
i,j,k,...\in \{1,2,...N\}\,, \quad\alpha,
\beta,\gamma...\in \{N+1,...,2N\}\quad \mbox{and} \quad\i,\j,\k, ,... \in \{1,2,...,2N\}\,.
\]
As a result,
\[
\sum_i\equiv\sum_{i=1}^N\,,\quad \sum_{\alpha}\equiv\sum_{\alpha=N+1}^{2N}\quad \mbox{and} \quad \sum_{\i}\equiv \sum_{\i=1}^{2N}\,.
\]
Furthermore, we set $\i'=(\i+N-1$  ({mod} $2N$))+1. In particular, $i'\deq i+N$ and  $\alpha'\deq \alpha-N$. We have the abuse of notations
\[
\sum_{i} f(i,i')\equiv \sum_{i=1}^N f(i,i+N) \quad \mbox{and}\quad \sum_{\alpha} f(\alpha',\alpha)\equiv \sum_{\alpha=N+1}^{2N}f(\alpha-N,\alpha)\,.
\]
We abbreviate
\begin{equation*}
\partial_{\i\j} F\deq \frac{\partial F}{\partial H_{\i\j}}
\end{equation*}
for differentiable functions $F$ of $H_w$ and $\widetilde{H}_w$. It is easy to see that
\begin{equation} \label{diff}
	\partial_{\i\j} G_{\k\l}=-G_{\k\i}G_{\j\l}-G_{\k\j}G_{\i\l} \quad \mbox{and}\quad \partial_{\i\j} \widetilde{G}_{\k\l}=-\widetilde{G}_{\k\i}\widetilde{G}_{\j\l}-\widetilde{G}_{\k\j}\widetilde{G}_{\i\l}
\end{equation}
whenever $H_{\i\j}\not\equiv0$. For a square matrix $Q \in \bb R^{n\times n}$ of any size, we denote its normalized trace by 
\begin{equation} \label{trace}
	\ul{Q}\deq n^{-1}\tr Q\,.
\end{equation}
 Let us denote the limiting density law of $H_w$ by $\varrho_w$. We denote the Stieltjes transform of $\varrho_w$ at $z$ by $m\equiv m(w,z)$. It satisfies
\begin{equation}\label{mmm}
-\frac{1}{m}=z+m-\frac{|w|^2}{z+m}\,,\quad \im m>0\,.
\end{equation}
We also define
\[
\frak m\equiv \frak m(w,z)\deq -\frac{m}{z+m}\,.
\]
The deterministic approximation of $G$ is defined by $M\equiv M(w,z)\in \bb C^{2N\times 2N}$, which satisfies
\begin{equation} M\deq
	\begin{pmatrix}
		mI_{N}& w\frak mI_{N}\\
		\bar{w}\frak mI_{N}&mI_{N}
	\end{pmatrix}\,,
\end{equation}
where $I_N$ is the identity matrix in $\bb R^{N\times N}$. The next lemma collects some elementary facts whose proofs we omit. 

\begin{lemma}  \label{lemma4.11} 
	\begin{enumerate}  
		\item We have the Ward identity
		\begin{equation} \label{ward}
			\sum_{\i}|G_{\i\j}|^2=\frac{\im G_{\j\j}}{\eta}\leq \frac{|G_{\j\j}- m|+\im m}{\eta}\,.
		\end{equation}
	
	\item We have
	\begin{equation} \label{sumialpha}
		\sum_i G_{ii}=\sum_\alpha G_{\alpha\alpha}\,,
	\end{equation}
	and for $z=\ii \eta$,
	\[
	G_{i\alpha}=\overline{G_{\alpha i}}\,, \quad G_{ij}=-\overline{ G_{ji}}\,, \quad \mbox{and} \quad G_{\alpha\beta}=-\ol{G_{\beta\alpha}}\,.
	\]
	
	\item Parts (i) and (ii) remain valid when we replace $G$ by $\widetilde{G}$.
	
	\item For $z=\ii \eta$ and $0< \eta\leq 1$, the quantity $m\equiv m(w, \ii\eta)$ satisfy
	\begin{equation} \label{um}
		m=\ii \im m\asymp \begin{cases}
			\kappa^{1/2}+\eta^{1/3} \quad &\mbox{if} \quad |w|\leq 1\\
			\tfrac{\eta}{\kappa+\eta^{2/3}} \quad &\mbox{if} \quad |w|>1\,.
		\end{cases} 	
	\end{equation}
	\end{enumerate}
\end{lemma} 

\paragraph{Cumulant expansion} Recall that for a real random variable $h$, all of whose moments are finite, the $k$-cumulant of $h$ is
\begin{equation*}
	\cal C_k(h)\deq(-\mathrm{i})^k\bigg(\frac{\dd^{k}}{\dd t^k}\log\mathbb{E}[e^{\mathrm{i}th}]\bigg)\Bigg|_{t=0}.
\end{equation*}
We shall use a standard cumulant expansion from \cite{KKP, Kho2, HK}. The proof was given in e.g.\,\cite[Appendix A]{HKR}. %
\begin{lemma} \label{cumulant}
	Let $f:\R\to\C$ be a smooth function, and denote by $f^{(k)}$ its $k$th derivative. Then, for every fixed $\ell \in\N$, we have 
	\begin{equation}\label{eq:cumulant_expansion}
		\mathbb{E}\big[h\cdot f(h)\big]=\sum_{k=0}^{\ell}\frac{1}{k!}\mathcal{C}_{k+1}(h)\mathbb{E}[f^{(k)}(h)]+\cal R_{\ell+1},
	\end{equation}	
	assuming that all expectations in \eqref{eq:cumulant_expansion} exist, where $\cal R_{\ell+1}$ is a remainder term (depending on $f$ and $h$), such that for any $t>0$,
	\begin{equation} \label{R_l+1}
		\cal R_{\ell+1} = O(1) \cdot \bigg(\E\sup_{|x| \le |h|} \big|f^{(\ell+1)}(x)\big|^2 \cdot \E \,\big| h^{2\ell+4} \mathbf{1}_{|h|>t} \big| \bigg)^{1/2} +O(1) \cdot \bb E |h|^{\ell+2} \cdot  \sup_{|x| \le t}\big|f^{(\ell+1)}(x)\big|\,.
	\end{equation}
\end{lemma}

The following result gives bounds on the cumulants of the entries of $B$, whose proof follows from \eqref{2.1} and the homogeneity of the cumulants.
\begin{lemma} \label{Tlemh}
	For every fixed $k \geq3$ we have
	\begin{equation*}
		\cal C_{k}(B_{ij})=O_{k}(1/(Nq^{k-2}))
	\end{equation*}
	uniformly for all $i,j$.
\end{lemma}

\section{Local law for $H_w$.} \label{sec3}

In this section\footnote{Figure \ref{fig:proof3} describes the relations between different results in Section \ref{sec3}. The key steps are marked in blue. We shall make similar introductions to Sections \ref{sec4.4} and \ref{sec5} in Figures \ref{fig:proof4.2}, \ref{fig:proof4.3} and \ref{fig:proofthm1.1}.}, we focus on the centered model $H_w$ . We shall first prove a weak local law on the whole spectrum (Theorem \ref{theorem 4.1}), and then establish a strong local law near the spectral edge (Theorem \ref{thmHstrong}). The proof of Theorem \ref{theorem 4.1} is rather standard, and the main technical steps in showing Theorem \ref{thmHstrong} are Lemmas \ref{lemma4.4} and \ref{lemma4.3}.

Lemma \ref{lemma4.4} is a generalized War-identity, which aims at utilizing the contribution of 4 off-diagonal entries of the Green function. The main idea behind the proof is that the sum $\sum_{\i} |G_{\i \j}|^4$ preserves its structure under differentiations w.r.t.\,$H$. 

Lemma \ref{lemma4.3}, in particular \eqref{4.32}, estimates the higher order terms in our computation. When a term has many off-diagonal entries of $G$, it can be estimated using Lemma \ref{lemma4.4}; when a term mainly consists of diagonal entries of $G$, we make use of the fact that $G_{\i \i}$ is small near the edge. The source of this smallness is the cusp-singularity.

Combining the above results yields Proposition \ref{prop4.1}, an estimate of the self-consistent equation $P(\ul{G})$. Theorem \ref{thmHstrong} is then deduced from Proposition \ref{prop4.1} by a standard stability analysis argument.

\begin{figure}[ht]
	\begin{tikzpicture}[>=stealth,every node/.style={shape=rectangle,draw,rounded corners, minimum width=2.9cm,},]
		\node[very thick, ] (l2) {
			Theorem~\ref{thmHstrong}};

		\node[very thick, minimum width=3.3cm] (l5)[right=of l2]{Proposition~\ref{prop4.1} };

		\node[very thick] (l7) [right=of l5]{Theorem~\ref{theorem 4.1}};
		\node[very thick] (l8) [below=of l7, , fill=blue!30]{Lemma~\ref{lemma4.3}};
		\node[very thick] (l6) [right=of l8, fill=blue!30]{Lemma~\ref{lemma4.4}};
		
			\node[very thick] (l9) [below=of l6]{Lemma~\ref{lempartialP}};
		
		\node[very thick] (l10) [right=of l7]{Proposition~\ref{prop4.2}};

		
		\draw[->,  line width=.6mm] (l5) to[out=180,in=0] (l2);
		
		\draw[->,  line width=.6mm] (l6) to[out=180,in=0] (l8);
		\draw[->,  line width=.6mm] (l7) to[out=180,in=0] (l5);
		\draw[->,  line width=.6mm] (l8) to[out=180,in=-4] (l5);

		\draw[->,  line width=.6mm] (l10) to[out=180,in=0] (l7);
		
				\draw[->,  line width=.6mm] (l9) to[out=180,in=-3mm] (l8);
		
	\end{tikzpicture}
	\caption{The structure of Section \ref{sec3}}
	\label{fig:proof3}
\end{figure}
\subsection{Weak local law for $H_w$}
For fixed $\delta>0$, we define the domain
\begin{equation} \label{Ddelta}
\b D_\delta \deq \{ (w,z)\in \bb C^2: |w|\leq \delta^{-1},\, z=E+\ii \eta,\, |E|\leq \delta^{-2},  N^{-1+\delta}\leq \eta \leq \delta^{-1}\}\,.
\end{equation}
We shall show that the random matrix $H_w$ satisfies the local density law, which says that its eigenvalue distribution is close to the deterministic $\varrho_w$ in \eqref{mmm}, down to spectral scales containing slightly more than one eigenvalue. This local density law is formulated using the Green function, whose individual entries are controlled by large deviation bounds. The detailed statements are as follows.

\begin{theorem}\label{theorem 4.1}
Fix $\delta\in (0,\xi/100)$. We have
\begin{equation} \label{4.8}
\max_{\i\j}\big|	G_{\i\j}-M_{\i\j} \big|\prec \frac{1}{(N\eta)^{1/6}}+\frac{1}{q^{1/3}}  \quad \mbox{and} \quad \max_{\i \j}|G_{\i\j}|\prec 1
\end{equation}
uniformly for $(w,z)\in \b D_\delta$.
\end{theorem}

The next result is the probabilistic step in showing Theorem \ref{theorem 4.1}. The proof is a standard process using Lemma \ref{cumulant}. The proof of a similar result can be found in \cite[Theorem 1.5(i)]{HKR}; we omit the details.

\begin{proposition} \label{prop4.2}
Fix $\delta\in (0,\xi/100)$ and $\nu \in (0,\delta/100)$. Let $(w,z)\in \b D_\delta$. Suppose that $\max_{\i\j}|G_{\i\j}-M_{\i\j}|\prec \phi$ for some deterministic $\phi\in [N^{-1},N^{\nu}]$ at $(w,z)$. Then at $(w,z)$ we have
	\[
	\max_{\hat\imath \hat\jmath}\big|(HG)_{\hat\imath\hat\jmath}+\ul{G}G_{\hat{\imath}\hat\jmath} \big|\prec (1+\phi)^4\cdot \bigg(\sqrt{\frac{1}{N\eta}}+\frac{1}{q}\bigg)\eqd 
	{\cal E}\,.
	\]
	Here $H=H_0$, as defined in \eqref{H}, and the normalized trace $\ul{G}$ was defined in \eqref{trace}.
\end{proposition}

Having Proposition \ref{prop4.2} at hand, Theorem \ref{theorem 4.1} then follows from a straightforward stability analysis argument. 

\begin{proof}[Proof of Theorem \ref{theorem 4.1}]
	Let us denote $	g\deq \ul{G}$, and $\frak g\deq N^{-1} \sum_{i} G_{i' i}$. Suppose that $\max_{\i\j}|G_{\i\j}-M_{\i\j}|\prec \phi$ at $(w,z)\in \b D_\delta$. Triangle inequality and \eqref{um} imply $\max_{\i \j}|G_{\i\j}|\prec 1+\phi$. Proposition \ref{prop4.2} and the resolvent identity imply that
	\begin{equation} \label{4.133}
		1+z g+g^2+w\frak g\prec \cal E \quad \mbox{and} \quad z \frak g +g\frak g+\bar{w}g\prec \cal E\,.
	\end{equation}
	Multiplying $z+g$ and $w$ on the first and second relation of \eqref{4.133} respectively, together with $|z|=O(1)$ and $\max_{\i \j}|G_{\i\j}|\prec 1+\phi$, we have
	\[
	g^3+2z g^2+(1+z^2-|w|^2)g+z \prec (1+\phi)\,\cal E\,.
	\]
As $m$ satisfies $m^3+2z m^2+(1+z^2-|w|^2)m+z=0$, by Taylor expansion, we get a cubic equation for $g-m$, namely
	\begin{equation} \label{4.199}
		(g-m)^3+(3m+2z)(g-m)^2+(3m^2+4z m+1+z^2-|w|^2)(g-m)\prec  (1+\phi)\,\cal E\,.
	\end{equation}
Note that $\im (3m+2z)>\im 2z=2\eta>0$. A simple analysis of the cubic equation \eqref{4.199} implies
	\begin{equation} \label{4.25}
		g-m\prec  ((1+\phi)\cal E)^{1/3}\,.
	\end{equation}
To estimate the entries of $G$, we can use Proposition \ref{prop4.2}, \eqref{4.25}, and the resolvent identity  to obtain
\begin{equation} \label{888}
	\delta_{i \j}+(z +m) G_{i\j}+wG_{i'\j} \prec (1+\phi)^2\cal E^{1/3}\,, \quad \delta_{i' \j}+\bar{w}G_{i\j} +(z +m) G_{i'\j}\prec (1+\phi)^2\cal E^{1/3}\,.
\end{equation}
 We can view \eqref{888} as a system of linear equations, with unknowns $G_{i\j}$ and $G_{i'\j}$. Note that its determinant satisfies
\[
|(z+m)^2-|w^2||^{-1}=\Big|\frac{m}{z+m}\Big|=|\frak m|=\Big|\frac{1+zm+m^2}{w^2}\Big|=O(1)\,,
\]
where we use the second and fourth term to get the estimate for the cases $|w|\leq 1/2$ and $|w|\geq 1/2$ respectively. Solving \eqref{888} for $G_{i\j}$ and $G_{i'\j}$, we get
\begin{equation} \label{4/26}
	\max_{\i\j}|G_{\i\j}-M_{\i\j}| \prec (1+\phi)^2\cal E^{1/3}
\end{equation}
at $(w,z)$, provided that $\max_{\i\j}|G_{\i\j}-M_{\i\j}|\prec \phi$. In addition, \eqref{4/26} also implies $\max_{\i}|G_{\i\i}|\prec 1$ at $(w,z)$. By a deterministic monotonicity result (see e.g.\,\cite[Lemma 10.2]{BK16}), one can show that
\[
\max_{\i \j}|G_{\i\j}|\prec N^{\nu}
\]
for all $(w,\tilde{z})$, where $\tilde{z}=E+\ii \tilde{\eta}$, $\widetilde{\eta} \in [N^{-\nu}\eta,\eta)$.  Then the proof can be concluded through a standard bootstrap argument (see e.g.\cite[Section 4.2]{HKR}).
\end{proof}

A straight-forward consequence (see e.g.\cite[Corollary 2.4]{AEK5}) of Theorem \ref{theorem 4.1} is the following complete eigenvector delocalization of $H_w$.

\begin{corollary}\label{lem3.3}
Let $w$ satisfy $|w|\leq \delta^{-1}$ for some fixed $\delta>0$, and let $\b u_1,...,\b u_{2N}\in \bb C^{2N}$ be the $L^2$-normalized eigenvectors of $H_w$. Then $\max_i\|\b u_i\|_\infty \prec N^{-1/2}$. 
\end{corollary}

\subsection{Strong local law for $H_w$ near the spectral edge} \label{sec3.2}
For fixed $\delta>0$, we define the spectral domains near the edge
\begin{equation} \label{3.11}
\b S^{(1)}_{\delta}\deq \{(w,\ii\eta) \in \bb C\times \ii \bb R: ||w|-1|\leq N^{-1/2+\delta}, N^{-1+\delta}\leq \eta \leq N^{-3/4+\delta}\}
\end{equation}
and
\begin{equation} \label{3.111}
	\b S^{(2)}_\delta \deq \{(w,\ii\eta)\in \bb C\times \ii\bb R: 1+N^{-1/2+\delta}\leq|w|\leq \delta^{-1}, N^{-1+\delta}\leq \eta \leq N^{-3/4+\delta}\}\,.
\end{equation}
We also set $\b S_\delta \deq \b S_\delta^{(1)} \cup \b S_\delta^{(2)}$. In other words, we are now only considering the Stieltjes transform of $H_w$ at $z=\ii \eta$. Applying Corollary \ref{lem3.3}, we can improve \eqref{ward} to
\begin{equation} \label{wardd}
	\sum_{\i}|G_{\i\j}|^2=\frac{\im G_{\j\j}}{\eta}\prec \frac{\im \ul{G}}{\eta} \prec \frac{|\ul{G}- m|+\im m}{\eta}\,,
\end{equation}
and as $m(w,\ii\eta)$ is purely imaginary, we also have
\begin{equation} \label{4q2}
	G_{\i\i}=\ii \im G_{\i\i} \prec |\ul{G}|\leq |\ul{G}-m|+\im m\,.
\end{equation} 
In sections \ref{sec3.2} -- \ref{sec4.3} we shall prove the following result.

\begin{theorem} \label{thmHstrong}
	Fix $\delta\in (0,\xi/100)$, $\nu \in (0,\delta/100)$. We have the following results.
	
	(i) For $(w,\ii\eta)\in \b S^{(1)}_\delta$, we have
	\[
	\ul{G}-m \prec \frac{1}{N\eta}\,.
	\]
	
	(ii) When $(w,\ii\eta)\in \b S^{(2)}_\delta$, we have the stronger estimate
	\begin{equation}\label{luguanyan}
		\ul{G}-m\prec \frac{1}{N^{1+\nu}\eta}\,.
	\end{equation}

\end{theorem}

An immediate consequence of Theorem \ref{thmHstrong} is the optimal upper bound on the spectral radius of $B$.

\begin{corollary}
	Let $\lambda_1^B,...,\lambda_N^B$ be eigenvalues of $B$. Then for any fix $\delta>0$, we have
	\[
	\big(\max_{i} |\lambda_i^B|- 1\big)_+= O(N^{-1/2+\delta})
	\] 
	with very high probability.
\end{corollary}
\begin{proof}
Fix $\nu \in (0,\delta/100)$. By Theorem \ref{thmHstrong} (ii) and Lemma \ref{lemma4.11} (iv), together with the fact that $\ul{G}$ is $N^3$-H\"{o}lder continuous in the variables $w$ and $\eta$, we get
	\[
	\im \ul{G} \prec \frac{1}{N^{1+\nu}\eta}
	\] 
	simultaneously for all $(w,\ii\eta)\in \b S^{(2)}_\delta$. This means for $1+N^{-1/2+\delta}\leq|w|\leq \delta^{-1} $, with very high probability, $H_w$ has no zero eigenvalue, and $B-w$ has no zero singular value. Thus with very high probability, no eigenvalue of $B$ lies in the ring $\{w:1+N^{-1/2+\delta}\leq|w|\leq \delta^{-1}\}$.
	
	On the other hand, using the moment method, it is not hard to see that $\|H\|= O(1)$ with very high probability (see e.g.\,\cite[Lemma 4.3]{EKYY1} for the proof of a similar result), and thus the spectral radius of $B$ is also bounded. This finishes the proof.
\end{proof}
For $(w,\ii\eta) \in \b S_\delta$, we denote 
\begin{equation} \label{3.12}
P(x)\equiv P_{w,\eta}(x)\deq x^3+2\ii \eta x^2+(1-\eta^2-|w|^2)x+\ii \eta\,.
\end{equation}
Clearly from \eqref{mmm}, $m$ satisfies $P(m)=0$. The main step in showing Theorem \ref{thmHstrong} is the following strong self-consistent equation of $H_w$, where $w$ is near the spectral edge.

\begin{proposition} \label{prop4.1}
Fix $\delta\in (0,\xi/100)$, and let $(w,\ii\eta)\in \b S_\delta$. Denote $g\deq \ul{G}$. Suppose that $|g-m|\prec \Lambda$ for some deterministic $\Lambda \in [N^{-1},N^{-\delta}]$ at $(w,\ii\eta)$. Then at $(w,
\ii\eta)$ we have
\[
P(g) \prec \frac{(\Lambda+\im {m})^2}{N\eta}+\frac{(\Lambda+\im m)^{1/2}}{N^{5/2}\eta^{5/2}}+\frac{(\Lambda+\im m)^{1/2}\kappa^{3/4}}{N\eta}+\frac{\Lambda^3+(\im {m})^3+\eta+\eta^{1/3}\kappa}{q^2}+\frac{1}{N\eta^{1/3}} \eqd \cal E_1\,.
\]
\end{proposition}

The next estimate will be useful in the subsequent steps. The proof is a straight-forward application of \eqref{diff}, \eqref{wardd} and \eqref{4q2}, and we omit the details here.

\begin{lemma} \label{lempartialP}
	Let us adopt the assumptions of Proposition \ref{prop4.1}. Then
	\[
	P'(g) \prec (\im m+\Lambda+\eta+\kappa^{1/2})^2 \quad \mbox{ and } \quad P''(g) \prec \im m+\Lambda+\eta\,.
	\]
	For any fixed integer $r\geq 1$, we have
	\[
	\partial_{\i\j}^r P(g) \prec  \frac{(\im m+\Lambda+\eta+\kappa^{1/2})^2 (\im m+\Lambda)}{N\eta} \prec  (\im m+\Lambda+\kappa^{1/2}+\eta) \cal E_1\,.
	\]
\end{lemma}

In the sequel, we first prove a prior estimate, Lemma \ref{lemma4.4}, in Section \ref{sec3.3}. We then prove Proposition \ref{prop4.1} in Section \ref{sec3.4}. Finally, we deduce Theorem \ref{thmHstrong} from Proposition \ref{prop4.1} in Section \ref{sec4.3}.

\subsection{A generalized Ward identity} \label{sec3.3}
We have the following generalization of \eqref{wardd}.
\begin{lemma} \label{lemma4.4}
	Let us adopt the assumptions of Proposition \ref{prop4.1}. Then
	\begin{equation} \label{4.4.1}
\max_{\j}	\frac{1}{N}\sum_{\i} |G_{\i \j}|^4 \prec  \Big(\frac{\Lambda+\im m}{N\eta}\Big)^2+\frac{1}{N}\eqd \cal E_2\,.
	\end{equation}
\end{lemma}

\begin{proof}
Fix an index $\j$. Let us denote $\cal G\deq N^{-1}\sum_{i} |G_{i \j}|^4$ and $\cal G_*\deq\max_{\j} N^{-1}\sum_{\i} |G_{\i\j}|^4$. Fix a positive integer $n$. We shall prove the lemma by showing that
	\begin{equation} \label{4q6}
		\bb E \cal G^n \prec \sum_{a=1}^n (\cal E_2+N^{-\delta}\Phi)^a\bb E \cal G^{n-a}+\Phi \cal E_2^{1/4} \bb E \cal G^{n-5/4}
	\end{equation}
provided $\cal G_* \prec \Phi$ for some deterministic $\Phi \in [N^{-1},N^4]$. More precisely, \eqref{4q6} and H\"{o}lder's inequality imply
\[
\bb E \cal G^{n} \prec \sum_{a=1}^n (\cal E_2+N^{-\delta}\Phi+\Phi^{4/5}\cal E_2^{1/5})^a(\bb E \cal G^{n})^{(n-a)/n}\,.
\]
Since $n$ is arbitrary, we get $\cal G \prec \cal E_2+N^{-\delta}\Phi+\Phi^{4/5}\cal E_2^{1/5}$. Similarly, we can also show that
\[
\frac{1}{N}\sum_{\alpha} |G_{\alpha \j}|^4 \prec \cal E_2+N^{-\delta}\Phi+\Phi^{4/5}\cal E_2^{1/5}\,.
\] 
Take the maximum over $\j$, we have
\begin{equation} \label{4q7}
	\cal G_* \prec \cal E_2+N^{-\delta}\Phi+\Phi^{4/5}\cal E_2^{1/5}
\end{equation}
provided $\cal G_* \prec \Phi$. Iterating \eqref{4q7} yields the desired result.

Let us turn to the proof of \eqref{4q6}. By the resolvent identity, we have $\bar{w}G_{i\j}=-\delta_{i'\j}-\ii\eta G_{i'\j}+(HG)_{i'\j}$. As $|w|\geq 1/2$, we have
\begin{align}
	\bb E \cal G^n=&\,\frac{1}{N\bar{w}} \sum_i \bb E G_{i\j}G^{*2}_{\j i} (-\delta_{i'\j}-\ii\eta G_{i'\j}+(HG)_{i'\j})\cal G^{n-1}\nonumber\\	
	=&\,\frac{1}{N\bar{w}}\sum_{i} \bb E G_{i\j}G^{*2}_{\j i}  (HG)_{i'\j}\cal G^{n-1}+O_{\prec}(N^{-1})\bb E \cal G^{n-1} \label{4.7}\\	
	=&\,\frac{1}{N\bar{w}}\sum_{r=1}^{\ell}\sum_{ik} \frac{1}{r!}\cal C_{r+1}(H_{i'k}) \bb E\partial_{i'k}^r (G_{i\j}G^{*2}_{\j i} G_{kj}\cal G^{n-1})+\cal R_{\ell+1}+O_{\prec}(N^{-1})\bb E \cal G^{n-1} \nonumber\,,
\end{align}
where in the first step we used \eqref{wardd}, and in the second step we used Lemma \ref{cumulant}. 

The second term on RHS of \eqref{4.7}  is the remainder term. Following a standard argument (e.g. \cite[Section 4.3]{HK}), one can show that for any fixed $D>0$, there is a fixed $\ell>0$ such that $\cal R_{\ell+1} =O(N^{-D})$. For the rest of the paper, we shall always assume the remainder term is negligible for large enough $\ell$. As a result,
	\begin{align}
	&\bb E \cal G^n=\frac{1}{N\bar{w}}\sum_{r=1}^{\ell}\sum_{ik} \frac{1}{r!}\cal C_{r+1}(H_{i'k}) \bb E\partial_{i'k}^r (G_{i\j}G^{*2}_{\j i} G_{k\j}\cal G^{n-1})+O(N^{-n})+O_{\prec}(N^{-1})\bb E \cal G^{n-1}\nonumber\\
	&= \frac{1}{N\bar{w}}\sum_{r=1}^{\ell}\sum_{r_1=0}^r\sum_{ik}\frac{1}{r!}\cal C_{r+1}(H_{i'k})  {r\choose r_1}\bb E\partial_{i'k}^{r_1} (G_{i\j}G^{*2}_{\j i} G_{k\j})\partial_{i'k}^{r-r_1}(\cal G^{n-1})+O(N^{-n})+N^{-1}\bb E \cal G^{n-1}\nonumber\\
	&\eqd \sum_{r=1}^{\ell}\sum_{r_1=0}^rW_{r,r_1}+O(N^{-n})+O_{\prec}(N^{-1})\bb E \cal G^{n-1}\label{4.99}\,,
	\end{align}
where in the second step we used Lemma \ref{Tlemh}. Note that by \eqref{diff}, we have the estimate
\begin{equation} \label{4q10}
\partial^l_{i'k}\cal G \prec \cal G_* \prec \Phi \quad \mbox{and} \quad N^{-2}\sum_{ik}\bb E\partial_{i'k}^{l} (G_{i\j}G^{*2}_{\j i} G_{k\j}) \prec \cal G_*\prec \Phi
\end{equation}
for all fixed $l\geq 0$. Now for fixed $(r,r_1)$, let us estimate $W_{r,r_1}$. We split into three cases.

\textit{Case 1.} When $r\geq 2$, by Lemma \ref{Tlemh} and \eqref{4q10}
\begin{equation} \label{4w11}
	\begin{aligned}
		W_{r,r_1} &\prec \frac{1}{N^2q^{r-1}} \sum_{ik} \bb E |\partial_{i'k}^{r_1} (G_{i\j}G^{*2}_{\j i} G_{k\j})\partial_{i'k}^{r-r_1}(\cal G^{n-1})|\\
		&\prec  \frac{1}{q^{r-1}}\cdot \Phi\cdot \sum_{a=1}^{n\wedge (r+1)}\Phi^{a-1} \bb E \cal G^{n-a} \prec  \sum_{a=1}^{n}q^{-a/3}\Phi^a \bb E \cal G^{n-a}\prec \sum_{a=1}^{n}(N^{-\delta}\Phi)^a \bb E \cal G^{n-a}\,.
	\end{aligned}
\end{equation}
Here in the last two steps we used $3(r-1)\geq r+1\geq a$ and $q^{1/3}=N^{\xi/3}> N^{10\delta}$ respectively.

\textit{Case 2.} When $(r,r_1)=(1,0)$, by \eqref{wardd} and \eqref{4q10} we have
\begin{equation} \label{4q11}
	\begin{aligned}
	W_{1,0} &\prec \frac{1}{N^2} \sum_{ik} \bb E\bigg[ |G_{i\j}G^{*2}_{\j i} G_{k\j}| \cdot \Phi \cal G^{n-2}\bigg] \prec \frac{\Phi}{N} \bigg(\frac{\im m+\phi}{N\eta}\bigg)^{1/2} \sum_i\bb E |G_{i\j}G^{*2}_{\j i}\cal G^{n-2}|\\
	&\prec \Phi \bigg(\frac{\im m+\phi}{N\eta}\bigg)^{1/2} \bb E \cal G^{n-5/4}\leq \Phi \cal E_2^{1/4} \bb E \cal G^{n-5/4}\,.
	\end{aligned}	
\end{equation}
Here in the third step we used H\"{o}lder's inequality.

\textit{Case 3.} When $(r,r_1)=(1,1)$, by \eqref{diff} and \eqref{wardd} we get
\begin{equation} \label{4q13}
	\begin{aligned}
	W_{1,0}&=-\frac{1}{N^2\bar{w}}\sum_{ik} \bb E G_{i\j}G^{*2}_{\j i}G_{i'\j}G_{kk} \cal G^{n-1}+O_{\prec}(\cal E_2)\bb E\cal G^{n-1}\\
	&=-\frac{m}{N\bar{w}}\sum_{i} \bb E G_{i\j}G^{*2}_{\j i}G_{i'\j} \cal G^{n-1}+O_{\prec}(\cal E_2+N^{-\delta}\Phi)\bb E\cal G^{n-1}=O_{\prec}(\cal E_2+N^{-\delta}\Phi)\bb E\cal G^{n-1}\,,
	\end{aligned}
\end{equation}
where in the second step we used Theorem \ref{theorem 4.1}, and in the third step we used Lemma \ref{lemma4.11} (iv) and H\"{o}lder's inequality. Inserting \eqref{4w11} -- \eqref{4q13} into \eqref{4.99}, we get \eqref{4q6} as desired. This finishes the proof.
\end{proof}

\subsection{Proof of Proposition \ref{prop4.1}} \label{sec3.4}
Fix an even integer $n\geq 2$, and abbreviate $P\equiv P(g)$.  Proposition \ref{prop4.1} follows directly from
\begin{equation} \label{3.3}
	\bb E |P|^{n} \prec \sum_{a=1}^{n}\cal E_1^a\bb E |P|^{n-a}
\end{equation}
and H\"{o}lder's inequality. In the sequel, we shall ignore the absolute value on LHS of \eqref{3.3}, which plays no role in the estimates. More precisely, we will show that
\begin{equation} \label{3.4}
	\bb E P^{n} \prec \sum_{a=1}^{n}\cal E_1^a\bb E |P|^{n-a}\,.
\end{equation}
By the resolvent identity and Lemma \ref{lemma4.11} (ii), we can split
\[
P=(\ii \eta+g)\Big(N^{-1}\sum_i(HG)_{ii}+g^2\Big)-w\Big( N^{-1}\sum_i(HG)_{i'i}+gN^{-1}\sum_i G_{i'i}\Big)\eqd (\ii \eta+g)P_1+P_2\,.
\]
The estimate \eqref{3.4} follows from
\begin{equation} \label{noncusp}
	\bb E (\ii \eta+g)P_1P^{n-1} \prec \sum_{a=1}^{n}\cal E_1^a\bb E |P|^{n-a}
\end{equation}
and
\begin{equation} \label{cusp}
	\bb E P_2P^{n-1} \prec \sum_{a=1}^{n}\cal E_1^a\bb E |P|^{n-a}\,.
\end{equation}
Comparing to $P_2$, the term $(\ii \eta+g)P_1$ contains the additional factor $\ii \eta+g\prec \im m+\Lambda$, which contributes extra smallness to the estimate. As a result, the proof of \eqref{cusp} is more involved than that of \eqref{noncusp}. One key idea of showing \eqref{cusp} is the cusp fluctuation averaging introduced in  \cite{EKS2020,AEK5}. We shall only give detailed steps in showing \eqref{cusp}, and omit the proof of the simpler case \eqref{noncusp}.

By Lemma \ref{cumulant}, we have
	\begin{align}
		\bb E P_2P^{n-1}& =-\frac{w}{N}\sum_{i j} \bb E H_{i'j}G_{ji} P^{n-1}-\frac{w}{N} \sum_{i} \bb EgG_{i'i} P^{n-1}\nonumber\\
		&=-\frac{w}{N}\sum_{r=1}^{\ell}\sum_{ij} \frac{1}{r!} \cal C_{r+1}(H_{i'j}) \bb E  \partial_{i'j}^r (G_{ji}P^{n-1})-\frac{w}{N} \sum_{i} \bb EgG_{i'i} P^{n-1}+O_{\prec}(\cal E_1^n)\label{4.5}\\
		&\eqd \sum_{r=1}^\ell X_r-\frac{w}{N} \sum_{i} \bb EgG_{i'i} P^{n-1}+O_{\prec}(\cal E_1^n)\nonumber\,. 
	\end{align}
The proof of \eqref{cusp} then follows from the next lemma.

\begin{lemma} \label{lemma4.3}
	We have
	\begin{equation} \label{4.31}
		X_1-\frac{w}{N} \sum_{i} \bb EgG_{i'i} P^{n-1} \prec  \sum_{a=1}^{n}\cal E_1^a\bb E |P|^{n-a}
	\end{equation}
and
\begin{equation}\label{4.32}
		X_r  \prec  \sum_{a=1}^{n}\cal E_1^a\bb E |P|^{n-a} 
\end{equation}
for $r=2,3,...,\ell$.	
\end{lemma}

In the remaining part of Section \ref{sec3.4} we prove \eqref{4.31} and \eqref{4.32} for $r=2,3$. For $r \geq 4$, the estimates of $X_r$ is easier due to Lemma \ref{lempartialP}, and we omit the details.

\subsubsection{Proof of \eqref{4.31}}  \label{section4.1}
By $\cal C_2(H_{i'j})=N^{-1}$ and \eqref{diff}, we split
\begin{equation}\label{4.9}
	\begin{aligned}
	X_1&=-\frac{w}{N^2}\sum_{ij} \bb E(\partial_{i'j} G_{ji}) P^{n-1}-\frac{w(n-1)}{N^2}\sum_{ij} \bb E G_{ji}(\partial_{i'j} P) P^{n-2}\\
	&=\frac{w}{N}\sum_{i} \bb E gG_{i'i}P^{n-1}+\frac{w}{N^2}\sum_{ij} \bb E G_{ji'}G_{ji}P^{n-1}-\frac{w(n-1)}{N^2}\sum_{ij} \bb E G_{ji}(\partial_{i'j} P) P^{n-2}\\
	&\eqd \frac{w}{N}\sum_{i} \bb E gG_{i'i}P^{n-1}+X_{1,1}+X_{1,2}\,.
		\end{aligned}
\end{equation}
Note that the first term on RHS of \eqref{4.9} gives us the cancellation in \eqref{4.31}. It remains to estimate $X_{1,1}$ and $X_{1,2}$. 

\textit{Step 1. The estimate of $X_{1,1}$.} In this estimate we make use of the cusp fluctuation averaging, which is contained in the factor $G_{ji'}$. The resolvent identity gives
$
wG_{ji'}=-\delta_{ji}-\ii \eta G_{ji}+ (GH)_{ji}$. Thus
	\begin{align}
		X_{1,1}&=-\frac{1}{N^2}\sum_{i} \bb E G_{ii}P^{n-1}-\frac{\ii \eta }{N^2}\sum_{ij} \bb E G_{ji}G_{ij}P^{n-1}+\frac{1}{N^2}\sum_{ij} \bb E (GH)_{ji}G_{ij}P^{n-1}\nonumber\\
		&=\frac{1}{N^2}\sum_{ij\alpha} \bb E G_{j\alpha}H_{\alpha i}G_{ij}P^{n-1}+O_\prec(\cal E_1) \bb E |P|^{n-1} \label{4.10}\\
		&=\frac{1}{N^2}\sum_{r=1}^{\ell}\sum_{ij\alpha} \frac{1}{r!}\cal C_{r+1}(H_{\alpha i})\bb E \partial^r_{\alpha i}(G_{j\alpha}G_{ij}P^{n-1})+O_\prec(\cal E_1) \bb E |P|^{n-1}\eqd \sum_{r=1}^{\ell} X_{1,1,r}+O_\prec(\cal E_1) \bb E |P|^{n-1}\nonumber\,,
	\end{align}
where in the second and third steps we used \eqref{wardd}, \eqref{4q2} and Lemma \ref{cumulant} respectively. We start to estimate the RHS of \eqref{4.10} for the case $r=1$. By \eqref{diff} and \eqref{wardd}, we can split $X_{1,1,1}$ into
\begin{equation} \label{beer}
	\frac{1}{N^3} \sum_{ij\alpha} \bb E\big[ (-G_{ji}G_{\alpha\alpha}G_{ij}-G_{j\alpha}G_{ii}G_{\alpha j})P^{n-1}+(n-1)G_{j\alpha}G_{ij}P^{n-2}(\partial_{\alpha i} P)\big]
+O_\prec(\cal E_1)\bb E|P|^{n-1}\,.
\end{equation}
By \eqref{wardd} and \eqref{4q2}, we get
\begin{equation} \label{4.11}
\frac{1}{N^3} \sum_{ij\alpha} \bb E (-G_{ji}G_{\alpha\alpha}G_{ij}-G_{j\alpha}G_{ii}G_{\alpha j})P^{n-1} \prec \cal E_1 \bb E |P|^{n-1}\,.
\end{equation}
In addition, \eqref{wardd} and Lemma \ref{lempartialP} implies
\begin{equation} \label{4.13}
\hspace{-0.2cm}\frac{1}{N^3} \sum_{ij\alpha} \bb E(n-1)G_{j\alpha}G_{ij}P^{n-2}(\partial_{\alpha i} P) \prec \frac{(\im m+\Lambda+\kappa^{1/2}	)\cal E_1}{N^3}\sum_{ij\alpha} \bb E |G_{ja}G_{ij}P^{n-2}| \prec \cal E_1^2 \bb E|P|^{n-2}\,.
\end{equation}
Combining \eqref{beer}, \eqref{4.11} and \eqref{4.13} results
\begin{equation} \label{4.14}
	X_{1,1,1} \prec \cal E_1 \bb E |P|^{n-1}+\cal E_1^2 \bb E |P|^{n-2}\,.
\end{equation}
For $r\geq 2$, the estimates are similar to those of $r=1$. More precisely, 
\begin{equation} \label{41}
X_{1,1,r} \prec \frac{1}{N^3q^{r-1}}\sum_{r_1=0}^r \sum_{ij\alpha}\bb E \partial^{r_1}_{\alpha i} (G_{j\alpha}G_{ij}) \partial^{r-r_1}_{\alpha j} (P^{n-1})\eqd \sum_{r_1=0}^{r}X_{1,1,r,r_1}\,.
\end{equation}
When $r_1\leq r-1$, by Lemma \ref{lempartialP} we have 
$\partial^{r-r_1}_{\alpha j} (P^{n-1})\prec (\im m+\Lambda+\kappa^{1/2}	)\sum_{a=2}^{n}\cal E_1^{a-1} |P|^{n-a}$, 
and together with \eqref{wardd} we get 
\begin{equation} \label{4.16}
X_{1,1,r,r_1}\prec \frac{\im m+\Lambda}{N\eta}\cdot (\im m+\Lambda+\kappa^{1/2}	)\sum_{a=2}^{n}\cal E_1^{a-1} \bb E |P|^{n-a} \prec \sum_{a=2}^{n}\cal E_1^{a} \bb E |P|^{n-a}\,.
\end{equation}
When $r_1=r$, by \eqref{diff}, we see that in each term of $\partial^{r_1}_{\alpha i} (G_{j\alpha}G_{ij})$, there are either three entries of $G$ that we can apply \eqref{wardd}, or there are only two entries of $G$ that we can apply \eqref{wardd} but there is at least one diagonal entry of $G$ that we can estimate by \eqref{4q2}. As a result, $\sum_{ij\alpha} \partial^{r_1}_{\alpha i} (G_{j\alpha}G_{ij}) \prec N^3 \cal E_1$, and
\begin{equation} \label{4.17}
	X_{1,1,r,r}\prec \cal E_1 \bb E|P|^{n-1}\,.
\end{equation}
Inserting \eqref{4.16} and \eqref{4.17} into \eqref{41} we get $X_{1,1,r} \prec \sum_{a=1}^{n}\cal E_1^{a} \bb E |P|^{n-a}$ for all $r \geq 2$. Together with \eqref{4.10} and \eqref{4.13}, we get
\begin{equation} \label{4.22}
	X_{1,1} \prec \sum_{a=1}^{n}\cal E_1^{a} \bb E |P|^{n-a}\,.
\end{equation}

\textit{Step 2. The estimate of $X_{1,2}$.} By \eqref{diff}, we have
\begin{equation} \label{4.23}
	X_{1,2}=\frac{w(n-1)}{2N^3} \sum_{ij}\bb EG_{ji}\big((G^2)_{ji'}+(G^2)_{i'j}\big) P' P^{2n-2}\,,
\end{equation}
where $P'=P'(\ul{G})$ and the derivative is on the variable $\ul{G}$. By the resolvent identity we have
$
\bar{w}(G^2)_{ji'}=-G_{ji}-\ii \eta (G^2)_{ji}+ (G^2H)_{ji}\,.
$
Together with \eqref{wardd} and Lemma \ref{lempartialP}, we get
\begin{equation*}
	\frac{w}{N^3} \sum_{ij}\bb EG_{ji}(G^2)_{ji'} P' P^{2n-2}
	=\,\frac{w}{N^3\bar{w}}\sum_{ij}\bb EG_{ji}(-\ii \eta(G^2)_{ij}+(G^2H)_{ji}) P' P^{2n-2}+O_\prec(\cal E_1^2) \bb E|P|^{n-2}\,.
\end{equation*}
To estimate the first term on RHS of the above, we apply
\begin{equation} \label{4.24}
	\sum_{j}|G_{ji}(G^2)_{ij}| \leq 	\sum_{\j}|G_{\j i}(G^2)_{i\j}| \leq ((GG^*)_{ii}(G^2G^{*2})_{ii})^{1/2} \prec \frac{\Lambda+\im m}{\eta^2}\,,
\end{equation}
where in the second step we used the resolvent identity and \eqref{4q2}. Thus
	\begin{align}
	&\frac{w}{N^3} \sum_{ij}\bb EG_{ji}(G^2)_{ji'} P' P^{2n-2}
	=\,\frac{w}{N^3\bar{w}}\sum_{ij}\bb EG_{ji}(G^2H)_{ji} P' P^{2n-2}+O_\prec(\cal E_1^2) \bb E|P|^{n-2}\nonumber\\
	=&\,\frac{w}{N^3\bar{w}} \sum_{ij\alpha}\bb EG_{ji}(G^2)_{j\alpha}H_{\alpha i} P' P^{2n-2}+O_\prec(\cal E_1^2) \bb E|P|^{n-2}\label{4.21}\\
	=&\,\frac{w}{N^3\bar{w}}\sum_{r=1}^{\ell} \sum_{ij\alpha}\frac{1}{r!} \cal C_{r+1}(H_{\alpha i}) \partial_{\alpha i}^r(G_{ji}(G^2)_{j\alpha} P'P^{2n-2})+O_\prec(\cal E_1^2) \bb E|P|^{n-2}\eqd \sum_{r=1}^\ell X_{1,2,r}+O_\prec(\cal E_1^2) \bb E|P|^{n-2}\nonumber\,.
	\end{align}
Here in the third step we used Lemma \ref{cumulant}. For fixed $r\geq 1$, we have
\begin{equation} \label{4.48}
X_{1,2,r} \prec \frac{1}{N^4q^{r-1}}\sum_{r_1=0}^r\partial_{\alpha i}^{r_1}(G_{ji}(G^2)_{j\alpha})\partial_{\alpha i}^{r-r_1} (P'P^{2n-2}) \eqd \sum_{r_1=0}^r X_{1,2,r,r_1}\,.
\end{equation}

The remaining steps are similar to the estimates of $X_{1,1,r,r_1}$ in \eqref{41}. More precisely, when $r_1\leq r-1$,  by \eqref{diff}, \eqref{wardd} and Lemma \ref{lempartialP}, we have 
\[
\partial_{\alpha i}^{r-r_1} (P'P^{2n-2}) \prec((\im m)^2+\Lambda^2+\kappa+\eta^2) \sum_{a=3}^n((\im m+\Lambda+\kappa^{1/2}	)\cal E_1)^{a-2} |P|^{n-a}+(\im m+\Lambda+\kappa^{1/2}	)\cal E_1|P|^{n-2}
\]
 Note that each term of $\partial_{\alpha i}^{r_1}(G_{ji}(G^2)_{j\alpha})$ contains the factor $G_{jx}(G^2)_{jy}$ or $G_{jx}G_{jy}$ for some $x,y \in \{i,\alpha\}$, and we can estimate it using \eqref{4.24} or \eqref{wardd} respectively. We get
\begin{equation} \label{4.46}
	X_{1,2,r,r_1}\prec \sum_{a=2}^{n} \cal E_1^a\bb E|P|^{n-a}\,.
\end{equation}
When $r_1=r$, we can split the terms of $\partial_{\alpha i}^{r_1}(G_{ji}(G^2)_{j\alpha})$ into two cases. The first case is when the result contains the factor $G_{jx}(G^2)_{jy}$ for some $x,y \in \{i,\alpha\}$. For this factor, we can estimate it using \eqref{4.24}. In addition, there is at least one off-diagonal entry of $G$ that we can estimate by \eqref{wardd}, or one diagonal entry of $G$ that we can estimate by \eqref{4q2}. The second case is when the result contain $G_{jx}G_{jy}(G^2)_{zw}$ for some $x,y,z,w\in \{i,a\}$, and we can estimate $\sum_j G_{jx}G_{jy}$ and $(G^2)_{zw}$ both by \eqref{wardd}. Together with Lemma \ref{lempartialP}, we have
\begin{equation} \label{4.47}
		X_{1,2,r,r} \prec \cal E_1^2 \bb E|P|^{n-2}\,.
\end{equation}
Inserting \eqref{4.48} -- \eqref{4.47} into  \eqref{4.21}, we obtain
\[
\frac{w}{N^3} \sum_{ij}\bb EG_{ji}(G^2)_{ji'} P' P^{2n-2}\prec  \sum_{a=2}^{n} \cal E_1^a\bb E|P|^{n-a}\,,
\]
which finishes the estimate of the first term on RHS of \eqref{4.23}. The estimate of the other term follow in the same fashion. Thus
\begin{equation} \label{4.49}
	X_{1,2}\prec \sum_{a=2}^{n} \cal E_1^a\bb E|P|^{n-a}\,.
\end{equation}
Inserting \eqref{4.22} and \eqref{4.49} into \eqref{4.9} concludes the proof of \eqref{4.31}.

\subsubsection{The estimate of $X_2$}
Note that $\cal C_3(B_{12})=O(N^{-1}q^{-1})$, and we split
\begin{equation}\label{4.410}
	\hspace{-0.2cm}X_2=-\frac{w}{N}\sum_{ij} \frac{1}{2} \cal C_{3}(H_{i'j}) \bb E  \partial_{i'j}^2 (G_{ji}P^{n-1})
	=O\Big(\frac{1}{N^2q}\Big)\sum_{ij}  \bb E  \partial_{i'j}^{2-r} (G_{ji})\partial_{i'j}^{r}(P^{n-1})\eqd \sum_{r=0}^2 X_{2,r}\,.
\end{equation}
We first consider $X_{2,0}$. By \eqref{diff}, we get
\begin{equation*}
	\begin{aligned}
		X_{2,0} =&\,O\Big(\frac{1}{N^2q}\Big) \sum_{ij}\bb E G_{ji}G_{i'j}^2 P^{n-1}+O\Big(\frac{1}{N^2q}\Big) \sum_{ij}\bb E G_{jj}G_{i'i'}G_{ji} P^{n-1}\\
		&+O\Big(\frac{1}{N^2q}\Big) \sum_{ij}\bb E G_{jj}G_{i'i}G_{i'j} P^{n-1} \eqd X_{2,0,1}+X_{2,0,2}+X_{2,0,3}\,.
	\end{aligned}
\end{equation*}
By H\"{o}lder's inequality and Lemma \ref{lemma4.4}, we have $N^{-1}\sum_{i}|G_{ij}G^2_{i'j}| \prec \cal E_2^{3/4}$, and thus
\begin{equation*}
	X_{2,0,1} \prec q^{-1} \cal E_2^{3/4} \prec \cal E_1 \bb E |P|^{n-1}\,.
\end{equation*}
By \eqref{wardd}, \eqref{4q2} and Cauchy-Schwarz, we get 
\begin{equation*}
	X_{2,0,2}\prec \frac{1}{N^2q} N^2 (\im m+\Lambda)^2 \Big(\frac{\im m+\Lambda}{N\eta}\Big)^{1/2} \bb E |P|^{n-1} \prec \cal E_1 \bb E |P|^{n-1}\,.
\end{equation*}
To estimate $X_{2,0,3}$, note that we have the resolvent identity $wG_{i'j}=-\delta_{ij}-\ii \eta G_{ij}+(HG)_{ij}$. Together with $|w|^{-1}\leq 2$, we get
\begin{equation} \label{123}
	\begin{aligned}
	X_{2,0,3}&=O\Big(\frac{1}{N^2q}\Big) \sum_{ij}\bb E G_{jj}G_{i'i}(HG)_{ij} P^{n-1} +O_{\prec}(\cal E_1)\bb E |P|^{n-1}\\
	&=O\Big(\frac{1}{N^2q}\Big) \sum_{ij\alpha}\bb E G_{jj}G_{i'i}H_{i\alpha}G_{\alpha j} P^{n-1} +O_{\prec}(\cal E_1)\bb E |P|^{n-1}\,.
	\end{aligned}
\end{equation}
Now we apply Lemma \ref{cumulant} to the first term on RHS of \eqref{123} with $h=H_{i\alpha}$, and estimate the results by \eqref{wardd}, \eqref{4q2} and Lemma \ref{lempartialP}. This leads to the estimate $X_{2,0,3} \prec \sum_{a=1}^n \cal E_1^a \bb E|P|^{n-a}$ as desired. As a result, we have
\begin{equation} \label{3.47}
	X_{2,0} \prec \sum_{a=1}^n \cal E_1^a \bb E|P|^{n-a}\,.
\end{equation}
Now let us consider $X_{2,1}$. By \eqref{diff}, \eqref{wardd}, \eqref{4q2} and Lemma \ref{lempartialP}, we see that
\begin{equation} \label{3.48}
	\begin{aligned}
		X_{2,1}&=O\Big(\frac{1}{N^2q}\Big) \sum_{ij}\bb E G_{i'i}G_{jj}(\partial_{i'j}\ul{G})P' P^{n-2}+O\Big(\frac{1}{N^2q}\Big) \sum_{ij}\bb E G_{ji}G_{ji}(\partial_{i'j}\ul{G})P' P^{n-2}\\
		&=O\Big(\frac{1}{N^3q}\Big) \sum_{ij}\bb E G_{i'i}G_{jj}((G^2)_{i'j}+(G^2)_{ji'})P' P^{n-2}+O_{\prec}(\cal E_1^2) \bb E |P|^{n-2}\,.
	\end{aligned}
\end{equation}
To estimate the first term on RHS of \eqref{3.48}, we again use the resolvent identity $w(G^2)_{i'j}=-G_{ij}-\ii \eta (G^2)_{ij}+(HG^2)_{ij}$ and $|w|^{-1}\leq 2$, which lead to 
\begin{equation}  \label{3.49}
	\begin{aligned}
		X_{2,1}
		=O\Big(\frac{1}{N^3q}\Big) \sum_{ij\alpha}\bb E G_{i'i}G_{jj}H_{i\alpha}(G^2)_{\alpha j}P' P^{n-2}+O_{\prec}(\cal E_1^2) \bb E |P|^{n-2}\prec \sum_{a=1}^n \cal E_1^a \bb E|P|^{n-a}\,.
	\end{aligned}
\end{equation}
Here in the second step we used Lemma \ref{cumulant} with $h=H_{i\alpha}$, and estimate the results by \eqref{wardd}, \eqref{4q2} and Lemma \ref{lempartialP}. The second term on RHS of \eqref{3.48} can be estimated in the same fashion. Then let us consider the term $X_{2,2}$, which cam be split into
\[
X_{2,2}=O\Big(\frac{1}{N^2q}\Big)\sum_{ij} \bb E G_{ji} (\partial^2_{i'j} P) P^{n-2}+O\Big(\frac{1}{N^2q}\Big)\sum_{ij} \bb E G_{ji} (\partial_{i'j} P)^2 P^{n-3}\eqd X_{2,2,1}+X_{2,2,2}\,.
\]
By \eqref{diff}, the most dangerous term in $X_{2,2,1}$ is 
\begin{align*}
	&\,O\Big(\frac{1}{N^3q}\Big)\sum_{ij}\bb E G_{ji}G_{i'i'}(G^2)_{jj}P'P^{n-2}\\
	\prec&\, \frac{1}{N^3q} N^2  \Big(\frac{\im m+\Lambda}{N\eta}\Big)^{1/2}\frac{(\im m+\Lambda)^2}{\eta}(\im m+\Lambda+\eta+\kappa^{1/2})^{2} \bb E |P|^{n-2} \prec \cal E_1^2 \bb E |P|^{n-2}\,,
\end{align*}
where in the first step we used \eqref{wardd}, \eqref{4q2} and Lemma \ref{lempartialP}. By estimating other terms in $X_{2,2,1}$ in a similar fashion, we get $X_{2,2,1} \prec \cal E_1^2 \bb E |P|^{n-2}$. In addition,  by Lemma \ref{lempartialP} and \eqref{wardd}, one easily sees that $X_{2,2,2} \prec \cal E_1^3 \bb E |P|^{n-3}$. Thus $X_{2,2} \prec \cal E_1^2 \bb E |P|^{n-2}+\cal E_1^3 \bb E |P|^{n-3}$. Together with \eqref{4.410}, \eqref{3.47} and \eqref{3.49}, we get $X_{2} \prec \sum_{a=1}^n \cal E_1^a \bb E|P|^{n-a}$ as desired.

\subsubsection{The estimate of $X_3$}
Note that $\cal C_4(B_{12})=O(N^{-1}q^{-2})$, and thus
\begin{align}
X_3&=-\frac{w}{N}\sum_{ij} \frac{1}{6} \cal C_{4}(H_{i'j}) \bb E  \partial_{i'j}^3 (G_{ji}P^{n-1})\nonumber\\
&=\sum_{r=0}^3\sum_{ij} O\Big(\frac{1}{N^2q^2}\Big) \bb E  \partial_{i'j}^{3-r} (G_{ji})\partial_{i'j}^{r}(P^{n-1})\eqd \sum_{r=0}^3 X_{3,r}\label{4.40}\,.
\end{align}
Let us first estimate $X_{3,0}$. When we apply three $\partial_{i'j}$ on $G_{ji}$, there will be three types of terms: terms containing only diagonal entries of $G$, terms contain two off-diagonal entries of $G$, and terms containing four diagonal entries of $G$. All these terms can be easily estimated by Lemma \ref{lemma4.4} and \eqref{4q2}. More precisely, we have
	\begin{align}
		X_{3,0}\prec &\,\frac{1}{N^2q^2} \sum_{ij} \bb E| G_{jj}^2G_{i'i}G_{i'i'}P^{n-1}|+\frac{1}{N^2q^2} \sum_{ij} \bb E |G_{i'i}G_{jj}G^2_{i'j} P^{n-1}|\label{4q41}\\
		&+\frac{1}{N^2q^2} \sum_{ij} \bb E |G_{i'i'}G_{jj}G_{i'j}G_{ij} P^{n-1}|+\frac{1}{N^2q^2} \sum_{ij} \bb E |G_{i'j}^3G_{ij} P^{n-1}|\prec \cal E_1 \bb E |P|^{n-1}\,.\nonumber
	\end{align}
Next we estimate of $X_{3,1}$. By \eqref{diff}, we see that
\begin{equation*}
	\begin{aligned}
	X_{3,1}&\prec \frac{1}{N^3q^2}\sum_{ij} \bb E |\partial_{i'j}^2(G_{ji}) P' (G^2)_{i'j}P^{n-2}|\\
	& \prec \frac{1}{N^3q^2}\sum_{ij} \bb E |G_{jj}G_{ii'}G_{i'j} P' (G^2)_{i'j}P^{n-2}|+ \frac{1}{N^3q^2}\sum_{ij} \bb E |G_{jj}G_{i'i'}G_{ij} P' (G^2)_{i'j}P^{n-2}|\\
	&\quad + \frac{1}{N^3q^2}\sum_{ij} \bb E |G_{ij}G_{i'j}^2 P' (G^2)_{i'j}P^{n-2}|\eqd X_{3,1,1}+X_{3,1,2}+X_{3,1,3}\,.
	\end{aligned}
\end{equation*}
By \eqref{4q2} and Lemma \ref{lempartialP}, we have
\begin{align*}
X_{3,1,1}&\prec \frac{(\Lambda+\im m)((\im m)^2+\Lambda^2+\kappa	+\eta^2)}{N^3q^2} \sum_{ij}\bb E |G_{i'j}(G^2)_{i'j}P^{n-2}|\\
&\prec   \frac{(\Lambda+\im m)((\im m)^2+\Lambda^2+\kappa	+\eta^2)}{N^3q^2} \frac{N (\im m+\Lambda)}{\eta^2} \bb E |P|^{n-2}\prec \cal E_1^2 \bb E |P|^{n-2}\,,
\end{align*}
and the same estimate works for $X_{3,1,2}$. In addition, \eqref{wardd} and Lemma \ref{lempartialP} imply
\[
X_{3,1,3} \prec  \frac{(\Lambda+\im m)((\im m)^2+\Lambda^2+\kappa	+\eta^2)}{N^3\eta q^2}\sum_{ij}\bb E |G^2_{i'j}P^{n-2}| \prec \cal E_1^2 \bb E |P|^{n-2}\,.
\]
As a result,
\begin{equation} \label{wq1}
	X_{3,1} \prec \cal E_1^2 \bb E |P|^{n-2}\,.
\end{equation}
Now we estimate $X_{3,2}$. By \eqref{diff}, we have
\begin{equation*}
	\begin{aligned}
		X_{3,2} \prec &\, \frac{1}{N^2q^2}\sum_{ij} \bb E| (\partial_{i'j}G_{ji})(\partial_{i'j} P)^2 P^{n-3}|+\frac{1}{N^2q^2}\sum_{ij} \bb E| (\partial_{i'j}G_{ji})(\partial^2_{i'j} P) P^{n-2}|\\
	\prec &\, \frac{1}{N^4q^2}\sum_{ij} \bb E| G_{i'i}G_{jj}(P'(G^2)_{i'j})^2P^{n-3}|+\frac{1}{N^4q^2}\sum_{ij} \bb E |G_{ji'}G_{ji}(P'(G^2)_{i'j})^2P^{n-3}|\\
		&\,+\frac{1}{N^3q^2}\sum_{ij} \bb E |G_{i'i}G_{jj}\partial_{i'j} (P'(G^2)_{i'j}) P^{n-2}|+\frac{1}{N^3q^2}\sum_{ij} \bb E |G_{ji}G_{ji'}\partial_{i'j} (P'(G^2)_{i'j}) P^{n-2}|\\
		\eqd& X_{3,2,1}+\cdots+X_{3,2,4}\,.
	\end{aligned}
\end{equation*}
By \eqref{wardd}, \eqref{4q2} and Lemma \ref{lempartialP}, we get
\begin{align*}
	X_{3,2,1} &\prec \frac{(\Lambda+\im m)(\Lambda^2+(\im m)^2+\kappa+\eta^2)^2}{N^4q^2} \sum_{ij} \bb E |(G^2)_{i'j}^2P^{n-3}| \\
	&\prec \frac{(\Lambda+\im m)(\Lambda^2+(\im m)^2+\kappa+\eta^2)^2}{N^4q^2} \frac{N(\Lambda+\im m)}{\eta^3}\bb E |P^{n-3}| \prec \cal E_1^3 \bb E |P|^{n-3}
\end{align*}
and similarly we get $X_{3,2,2} \prec \cal E_1^3 \bb E |P|^{n-3}$. In addition, by \eqref{diff}, \eqref{wardd}, \eqref{4q2} and Lemma \ref{lempartialP} we get
\begin{align*} 
	X_{3,2,3}&\prec\frac{1}{N^3q^2}\sum_{ij} \bb E G_{ii'}G_{jj} (G_{jj}(G^2)_{i'i'} +G_{i'i'}(G^2)_{jj})P'P^{n-2}+\cal E_3^2\bb E|P|^{n-2}\\
	&\prec \frac{1}{N^3q^2}\frac{N^2(\im m+\Lambda)^3}{\eta}(\im m+\Lambda+\eta+\kappa^{1/2})^2 \bb E |P|^{n-2}+\cal E_3^2\bb E|P|^{n-2} \prec \cal E_3^2\bb E|P|^{n-2}\,,
\end{align*}
and similarly $X_{3,2,4} \prec \cal E_1^2 \bb E |P|^{n-2}$. As a result, we get 
\begin{equation} \label{4q49}
	\begin{aligned}
X_{3,2}\prec\cal E_3^2\bb E|P|^{n-2}+\cal E_3^3\bb E|P|^{n-3}\,.
	\end{aligned}
\end{equation}
Finally we estimate of $X_{3,3}$. By \eqref{wardd}, \eqref{4.24} and Lemma \ref{lempartialP}, it is easy to see that
\begin{equation} \label{4.41}
	\begin{aligned}
	X_{3,3} &\prec\frac{1}{N^3q^2}\sum_{ij}  \bb E   |G_{ji}(\partial_{i'j}^{3} P) P^{n-2}|+\sum_{a=3}^4\cal E_1^a \bb E|P|^{n-a}\\
	&\prec \frac{1}{N^3q^2}\sum_{ij}  \bb E |  G_{ji}\partial_{i'j}^{2}( P' (G^2)_{i'j}) P^{n-2}|+\sum_{a=3}^4\cal E_1^a\bb E|P|^{n-a}\\
	&\prec \frac{1}{N^3q^2}\sum_{ij}  \bb E  | G_{ji}(\partial_{i'j}^{2} (G^2)_{i'j}) P'P^{n-2}|+\sum_{a=3}^4\cal E_1^a \bb E|P|^{n-a}\,.
	\end{aligned}
\end{equation}
By \eqref{diff}, \eqref{wardd} and \eqref{4q2}, the first term on RHS of \eqref{4.41} is stochastically dominated by 
	\begin{align*}
	\frac{1}{N^3q^2} \sum_{i,j}\bb E \big|G_{ji}\big[(G^2)_{i'i'}G_{jj}G_{i'j}+(G^2)_{jj}G_{i'i'}G_{i'j}+(G^2)_{i'j}G^2_{i'j}+(G^2)_{i'j}G_{i'i'}G_{jj}\big]P'P^{n-2}\big|\prec \cal E_1^2 \bb E |P|^{n-2}
	\end{align*}
Inserting the above into \eqref{4.41} we get 
\begin{equation} \label{4.42}
	X_{3,3} \prec \sum_{a=2}^4\cal E_1^a \bb E|P|^{n-a}\,.
\end{equation}
Combining \eqref{4q41}, \eqref{wq1}, \eqref{4q49} and \eqref{4.42}, we get $X_3 \prec\sum_{a=2}^{4}\cal E_1^a\bb E |P|^{n-a}$ as desired. This finishes the proof of Lemma \ref{lemma4.3}, and thus we conclude the proof of Proposition \ref{prop4.1}.

\subsection{Stability analysis: Proof of Theorem \ref{thmHstrong}}\label{sec4.3} By Proposition \ref{prop4.1}, Taylor expansion and $P(m)=0$, we have
\begin{equation*}
		(g-m)^3+(3m+2\ii \eta)(g-m)^2+(3m^2+4\ii m\eta+1-\eta^2-|w|^2)(g-m)\prec \cal E_1\,.
\end{equation*}
From Lemma \ref{lemma4.11}, we see that
\[
3m^2+4\ii m\eta+1-\eta^2-|w|^2=2m^2+2\ii m\eta-\frac{\ii \eta}{m}\asymp (\im m)^2+\frac{\eta}{\im m}\,.
\]
By a standard stability analysis of cubic equations (e.g. \cite[Lemma 3.10]{EKS2020}) we get
\begin{equation} \label{4.26}
	g-m \prec \min \bigg\{ \cal E_1^{1/3}, \frac{\cal E_1^{1/2}}{(\im m+\eta)^{1/2}}, \frac{\cal E_1}{(\im m)^2+\eta/\im m}\bigg\}\,.
\end{equation}

(i) When $(w,\ii\eta)\in \b S^{(1)}_\delta$, Lemma \ref{lemma4.11} and \eqref{4.26} imply
\begin{equation} \label{4.29}
	g-m\prec \min\bigg\{\cal E_1^{1/3},\frac{\cal E_1}{\kappa+\eta^{2/3}}\bigg\}=\min_{0\leq a\leq 1}\cal E_1^{a/3}\bigg(\frac{\cal E_1}{\kappa+\eta^{2/3}}\bigg)^{1-a}\,.
\end{equation}
In addition, 
\begin{align*} 
	\cal E_1 \prec  \frac{\Lambda^2+(\im m)^2}{N\eta}+\frac{\Lambda^{1/2}+(\im m)^{1/2}}{(N\eta)^{5/2}}+\frac{\Lambda^{1/2}+(\im m)^{1/2}}{N\eta}\kappa^{3/4}+\frac{\Lambda^3+(\im m)^3+\eta+\eta^{1/3}\kappa}{q^2}+\frac{1}{N\eta^{1/3}}\,,
\end{align*}
where $\im m \prec \kappa^{1/2}+\eta^{1/3}$. Combining \eqref{4.29} and the above, we get
\[
g-m \prec \bigg(\frac{\Lambda^2}{N\eta}\bigg)^{1/3}+\frac{1}{N\eta}+ \Big(\frac{\Lambda}{(N\eta)^5}\Big)^{1/6}+\frac{1}{N\eta}+\frac{\Lambda^{1/3}}{(N\eta)^{2/3}}+\frac{\Lambda}{q^{2/3}}+\frac{\kappa^{1/2}+\eta^{1/3}}{q^2}+\frac{1}{N\eta}
\]
provided that $g-m\prec \Lambda$. Iterating the above gives 
\[
g-m \prec \frac{1}{N\eta}+\frac{\kappa^{1/2}+\eta^{1/3}}{q^2} \prec \frac{1}{N\eta}\,,
\]
which is the desired result.

(ii) When $(w,\ii\eta)\in \b S^{(2)}_\delta$, Lemma \ref{lemma4.11} and \eqref{4.26} imply
\begin{equation} \label{4.27}
	g-m \prec \min \bigg\{ \cal E_1^{1/3}, \frac{\cal E_1}{\kappa}\bigg\}= \min \bigg\{ \cal E_1^{1/3},\frac{\cal E_1^{1/2}}{\kappa^{1/4}}, \frac{\cal E_1}{\kappa}\bigg\}\,.
\end{equation}
In addition, Lemma \ref{lemma4.11} shows $\im m=O(\eta/\kappa)$, and thus
\begin{equation*} 
	\cal E_1 \prec \frac{\Lambda^2+(\eta/\kappa)^2}{N\eta}+\frac{\Lambda^{1/2}+(\eta/\kappa)^{1/2}}{(N\eta)^{5/2}}+\frac{\Lambda^{1/2}+(\eta/\kappa)^{1/2}}{N\eta}\kappa^{3/4}+\frac{\Lambda^3}{q^2}+\frac{\eta^3}{q^2\kappa^3}+\frac{\eta}{q^2}+\frac{\eta^{1/3}\kappa}{q^2}+\frac{1}{N\eta^{1/3}}\,.
\end{equation*}
Combining \eqref{4.27} with the above we get
\[
g-m\prec \frac{1}{N^{1+\nu}\eta}+N^{-\nu} \Lambda
\]
provided that $g-m\prec \Lambda$. By iterating the above we get the desired result.

\section{Local law of $\widetilde{H}_w$, delocalization and spectral radius of $A$.} \label{sec4.4}

This section generalizes the results in Section \ref{sec3} to the non-centered model $\widetilde{H}_w$. Consequently, for the adjacency matrix $A$, we establish the complete delocalization of the eigenvectors and the sharp upper bound of the spectral radius. Comparing to $H_w$, our main task in this section is to treat the large expectation of $\widetilde{H}_w$.

Section \ref{sec4.1} proves the averaged local law for $\widetilde{H}_w$. As $\widetilde{H}_w$ is a rank-two perturbation of $H_w$, the result is a simple consequence of Theorems \ref{theorem 4.1} and \ref{thmHstrong}.

Section \ref{sec4.2} proves Theorem \ref{radius and delocalization} (ii) by establishing Theorem \ref{theoremA}, an entriwise local law for $\widetilde{H}_w$. The main technical step is Proposition \ref{prop5.2} (ii), a weak estimate of the Green functions arising from the large expectations of $\widetilde{H}_w$. The proof idea is similar to that of Lemma \ref{lemma4.4}: the Green function we are interested in preserves its structure under differentiations w.r.t.\,$H$.

Section \ref{secradius} proves the upper bound in Theorem \ref{radius and delocalization} (i), by showing a stronger estimate of the Green function outside the spectrum (Proposition \ref{prop4.5}). The main technical result, Lemma \ref{lemma4.7}, is a sharp estimate of the Green functions arising from the large expectation of $\widetilde{H}_w$. To improve from Proposition \ref{prop5.2} (ii) to Lemma \ref{lemma4.7}, we make use of the fact that $\widetilde{H}_w-H_w$ is of rank two. Together with the resolvent identity, we manage to turn the large expectation of $\widetilde{H}_w$ to our advantage, by forming new self-consistent equations for $\sum_{ij} \widetilde{G}_{ij}$ and $\sum_{i\alpha} \widetilde{G}_{i\alpha}$. Armed with Lemma \ref{lemma4.7}, the rest of the proof can essentially be imported from the centered case, Proposition \ref{prop4.1} and Lemma \ref{lemma4.4}.

\begin{figure}
	\begin{tikzpicture}[>=stealth,every node/.style={shape=rectangle,draw,rounded corners, minimum width=2.9cm,},]
		\node[very thick,] (l2) {
			Theorem~\ref{radius and delocalization}\,(ii)};

		\node[very thick, minimum width=2.7cm] (l5)[right=of l2]{Lemma~\ref{lemmadelocalization} };
		
		\node[very thick] (l8) [below=of l5,minimum width=2.7cm]{Lemma~\ref{lemma 5.1}};

		\node[very thick] (l7) [right=of l5]{Theorem~\ref{theoremA}};

		\node[very thick] (l10) [right=of l7, fill=blue!30]{Proposition~\ref{prop5.2}};

		
		\draw[->,  line width=.6mm] (l5) to[out=180,in=0] (l2);
		
		\draw[->,  line width=.6mm] (l7) to[out=180,in=0] (l5);
		\draw[->,  line width=.6mm] (l8) to[out=180,in=-4] (l2);

		\draw[->,  line width=.6mm] (l10) to[out=180,in=0] (l7);
		
	\end{tikzpicture}
	\caption{The structure of Section \ref{sec4.2}}
	\label{fig:proof4.2}
\end{figure}

\begin{figure}
	\begin{tikzpicture}[>=stealth,every node/.style={shape=rectangle,draw,rounded corners, minimum width=2.9cm,},]
		\node[very thick, ] (l2) { \begin{tabular}{c}Corollary~\ref{cor4.7}\\(Upper bound in \\
			Theorem~\ref{radius and delocalization}\,(i))\end{tabular}}; 

		\node[very thick] (l7) [right=of l2]{Proposition~\ref{prop4.5}}; 

		\node[very thick] (l10) [right=of l7]{Lemma~\ref{lemma4.8}}; 
		
			\node[very thick] (l11) [right=of l10, fill=blue!30]{Lemma~\ref{lemma4.7}}; 
			
				\node[very thick] (l12) [below=of l11]{\begin{tabular}{c} Proposition~\ref{prop4.1}\\
			and	Lemma~\ref{lemma4.4}\end{tabular}}; 


		\draw[->,  line width=.6mm] (l7) to[out=180,in=0] (l2);

		\draw[->,  line width=.6mm] (l10) to[out=180,in=0] (l7);
		
				\draw[->,  line width=.6mm] (l11) to[out=180,in=0] (l10);
				
					\draw[->,  line width=.6mm] (l12) to[out=180,in=-4] (l10);
		
	\end{tikzpicture}
	\caption{The structure of Section \ref{secradius}}
	\label{fig:proof4.3}
\end{figure}

\subsection{The average law} \label{sec4.1}
Recall the definition of $\G$ in \eqref{2.4}. We have the following result.

\begin{corollary} \label{cor4.2}
Fix $\delta\in (0,\xi/100)$. Let $\b D_\delta$ and $\b S_\delta$ be defined as in \eqref{Ddelta} and just below \eqref{3.111} respectively. The following estimates hold.

	(i) We have
	\[
	\ul{\G}-m \prec \frac{1}{(N\eta)^{1/6}}+\frac{1}{q^{1/3}}
	\]
	uniformly for $(w,z)\in \b D_{\delta}$.
	
	(ii) We have 
	\[
	\ul{\G}-m \prec \frac{1}{N\eta}
	\]
	uniformly for $(w,\ii\eta)\in \b S_{\delta}$.
\end{corollary}

\begin{proof}
	Let $w \in \bb C$, $\b 0\deq (0,0,...,0)^*\in \bb R^{N}$ and $\b e\deq N^{-1/2}(1,1,...,1)^*\in \bb R^N$. Then
	\[
	\widetilde{H}_{w}=H_w+f\begin{pmatrix}
		\b e\\ \b e
	\end{pmatrix} \begin{pmatrix}
	\b e^* \ \b e^*
\end{pmatrix}-f\begin{pmatrix}
\b e\\ \b 0
\end{pmatrix} \begin{pmatrix}
\b e^* \ \b 0^*
\end{pmatrix}-f\begin{pmatrix}
\b 0\\ \b e
\end{pmatrix} \begin{pmatrix}
\b 0^* \ \b e^*
\end{pmatrix}\,.
	\]
In other words, $\widetilde{H}_w$ and $H_w$ differs by three rank-one perturbations. Let $\rho_w$ and $\widetilde{\rho}_w$ denote the empirical eigenvalue densities of $H_w$ and $\widetilde{H}_w$ respectively, then Cauchy interlacing theorem implies
\[
|\rho_w(I)-\widetilde{\rho}_w(I)| \leq \frac{3}{2N}\
\]
for any $I \subset \bb R$. Thus using integration by parts we have
\[
\ul{\G}-\ul{G}=\int \frac{\widetilde{\rho}_w(x)-\rho_w(x)}{x-\ii \eta} \dd x=-\int \frac{\widetilde{\rho}_w((-\infty, x])-\rho_w((-\infty,x])}{(x-\ii \eta)^2} \dd x=O(N^{-1})\int \frac{1}{x^2+\eta^2} \dd x \,,
\]
which implies $|\ul{\G}-\ul{G}|\leq C/(N\eta)$. The result then follows from Theorems \ref{theorem 4.1} and \ref{thmHstrong}.
\end{proof}

\subsection{Entrywise law and delocalization} \label{sec4.2}
In this section prove the following entrywise density law for $\widetilde{H}_w$.
\begin{theorem} \label{theoremA}
	Fix $\delta \in (0,\xi/100)$, and let $\b D_\delta$ be defined as in \eqref{Ddelta}. We have
	\begin{equation*} 
		\max_{\i\j}\big|	\widetilde{G}_{\i\j}-M_{\i\j} \big|\prec \frac{1}{(N\eta)^{1/6}}+\frac{1}{q^{1/3}}  \quad \mbox{and} \quad \max_{\i \j}|\widetilde{G}_{\i\j}|\prec 1
	\end{equation*}
	uniformly for $(w,z)\in \b D_\delta$.
\end{theorem}

To show Theorem \ref{theoremA}, we need the following probabilistic estimates.
\begin{proposition} \label{prop5.2}
Recall the definition of $H$ in \eqref{H}. Fix $\delta\in (0,\xi/100)$ and $\nu \in (0,\delta/100)$. Let $(w,z)\in \b D_\delta$. Suppose that $\max_{\i\j}|\widetilde{G}_{\i\j}-M_{\i\j}|\prec \phi$ for some deterministic $\phi\in [N^{-1},N^{\nu}]$ at $(w,z)$.
	
	(i)  At $(w,z)$ we have
	\[
	\max_{\hat\imath \hat\jmath}\big|(H\widetilde{G})_{\hat\imath\hat\jmath}+\ul{\widetilde{G}}\widetilde{G}_{\hat{\imath}\hat\jmath} \big|\prec (1+\phi)^4\cdot \bigg(\sqrt{\frac{1}{N\eta}}+\frac{1}{q}\bigg)\eqd 
	{\cal E}\,.
	\]
	Here $H=H_0$, as defined in \eqref{H}.
	
	(ii) Suppose in addition that
	\[
	\max_{\j}\Big|\sum_{i} \widetilde{G}_{i\j}\Big|+\max_{\j}\Big|\sum_{\alpha} \widetilde{G}_{\alpha\j}\Big| \prec \psi
	\]
	for some deterministic $\psi\in [N^{-1},N^{1+\nu}]$ at $(w,z)$. Then at $(w,z)$ we have
	\begin{equation*}
		\begin{aligned}
\cal Q_*&\deq 		\max_{\j}\Big|\sum_{i} \big((H\widetilde{G})_{i\j}+\ul{\widetilde{G}}\widetilde{G}_{i\j}\big)\Big|+\max_{\j}\Big|\sum_{\alpha}\big( (H\widetilde{G})_{\alpha\j}+\ul{\widetilde{G}}\widetilde{G}_{\alpha\j}\big)\Big|\\
&\prec (\im m+\phi) \eta^{-1} +(1+\phi)^2(|m|+\phi)^2Nq^{-2}+\psi\eqd \widetilde{\cal E}\,.
		\end{aligned}
	\end{equation*}
\end{proposition}

\begin{proof}
	Part (i) is essentially identical to Proposition \ref{prop4.2}, whose proof is a standard argument using Lemma \ref{cumulant}, and we shall omit it and only give the details of part (ii). Fix an even integer $n \geq 2$ and an index $\j$.  Let us denote $\cal Q\deq\sum_{i} ((H\widetilde{G})_{i\j}+\ul{\widetilde{G}}\widetilde{G}_{i\j})$ and suppose $\cal Q_* \prec \Psi$ for some deterministic $\Psi\in [1,N^{2}]$. We shall prove our statement by showing that
	\begin{equation} \label{5.3}
		\bb E |\cal Q|^{n} \prec \sum_{a=1}^n (\widetilde{\cal E}+\widetilde{\cal E}^{1/2}\Psi^{1/2}+N^{-\delta}
		\Psi)^a\bb E \cal |\cal Q|^{n-a}\,.
	\end{equation}
Indeed, \eqref{5.3} implies $\sum_i((H\widetilde{G})_{i\j}+\ul{\widetilde{G}}\widetilde{G}_{i\j}) \prec \widetilde{\cal E}+\widetilde{\cal E}^{1/2}\Psi^{1/2}+N^{-\delta}
\Psi$. Similarly, we also have $\sum_\alpha((H\widetilde{G})_{\alpha\j}+\ul{\widetilde{G}}\widetilde{G}_{\alpha\j}) \prec \widetilde{\cal E}+\widetilde{\cal E}^{1/2}\Psi^{1/2}+N^{-\delta}
\Psi$. Since the estimates are uniform in $\j$, we get
\begin{equation} \label{5.4}
\cal Q_* \prec \widetilde{\cal E} +\widetilde{\cal E}^{1/2}\Psi^{1/2}+N^{-\delta} \Psi
\end{equation}
provided that $\cal Q_* \prec \Psi$. Iterating \eqref{5.4} we get the desired result. As complex conjugates play no roles in our argument, we shall ignore it on LHS of \eqref{5.3} and prove
\begin{equation} \label{5.5}
	\bb E \cal Q^{n} \prec \sum_{a=1}^n (\widetilde{\cal E}+\widetilde{\cal E}^{1/2}\Psi^{1/2}+N^{-\delta}
	\Psi)^a\bb E \cal |\cal Q|^{n-a}
\end{equation}
instead. By Lemma \ref{cumulant} we get
\begin{equation} \label{4.412}
	\begin{aligned}
		\bb E \cal Q^{n}&=\sum_{i\alpha} \bb E H_{i\alpha }\widetilde{G}_{\alpha j} \cal Q^{n-1}+\sum_i \bb E\ul{\widetilde{G}}\widetilde{G}_{i\j}\cal Q^{n-1}\\
		&=\sum_{r=1}^{\ell}\sum_{i\alpha}\cal C_{r+1}(H_{i\alpha})\bb E\partial_{i\alpha}^r(\widetilde{G}_{\alpha\j}\cal Q^{n-1})+\sum_i \bb E\ul{\widetilde{G}}\widetilde{G}_{i\j}\cal Q^{n-1}+(N^{-10n})\\
		&\eqd \sum_{r=1}^{\ell} Y_r+\sum_i\bb E \ul{\widetilde{G}}\widetilde{G}_{i\j}\cal Q^{n-1}+(N^{-10n})\,.
	\end{aligned}
\end{equation}
By \eqref{diff} and Lemma \ref{lemma4.11} (iii), we have
	\begin{align}
	Y_1=&\,-\sum_i \bb E\ul{\widetilde{G}}\widetilde{G}_{i\j}\cal Q^{n-1}-\frac{1}{N} \sum_{i\alpha}\bb E \widetilde{G}_{\alpha i} \widetilde{G}_{\alpha \j}\cal Q^{n-1}+\frac{n-1}{N}\sum_{i\alpha}\bb E \widetilde{G}^2_{\alpha \j}\cal Q^{n-2}\nonumber\\
	&\,-\frac{n-1}{N}\sum_{i\alpha}\bb E \widetilde{G}_{\alpha\j}\G_{\alpha j} \sum_{k}((H\widetilde{G})_{ki}+\ul{\widetilde{G}}\widetilde{G}_{ki})\cal Q^{n-2}-\frac{n-1}{N}\sum_{i\alpha}\bb E \widetilde{G}_{\alpha\j}\G_{i j} \sum_{k}((H\widetilde{G})_{k\alpha}+\ul{\widetilde{G}}\widetilde{G}_{k\alpha})\cal Q^{n-2}\nonumber\\
	&-\frac{(n-1)}{2N^2}\sum_{i\alpha}\bb E \widetilde{G}_{\alpha j} ((\G^2)_{i\alpha}+(\G^2)_{\alpha i})\sum_{k}\G_{k\j}\cal Q^{n-2}\nonumber\\
	=&\,-\sum_i \bb E\ul{\widetilde{G}}\widetilde{G}_{i\j}\cal Q^{n-1}+O_{\prec}((\im m+\phi)\eta^{-1}) \bb E |\cal Q|^{n-1}+O_{\prec}((\im m+\phi)\eta^{-1}) \bb E |\cal Q|^{n-2}\nonumber\\
	&\,+O_{\prec}((\im m+\phi)\eta^{-1}\Psi) \bb E |\cal Q|^{n-2}+O_\prec((\im m+\phi)N^{-1}\eta^{-2}\psi)\bb E |\cal Q|^{n-2}\nonumber\,,
	\end{align}
and thus 
\begin{equation} \label{5.6}
Y_1	+\sum_i \bb E\ul{\widetilde{G}}\widetilde{G}_{i\j}\cal Q^{n-1}=O_\prec(\widetilde{\cal E}) \bb E |\cal Q|^{n-1}+O_\prec(\widetilde{\cal E}^2+\widetilde{\cal E}\Psi) \bb E |\cal Q|^{n-2}\,.
\end{equation}
Note that for $s\geq 1$
\begin{equation} \label{4.4}
\partial_{i\alpha}^s \cal Q \prec (1+\phi)^{s-1}\Big(\Psi(|\G_{i\j}|+|\G_{\alpha \j}|)+\frac{(\im m+\phi)\psi}{N\eta}\Big)\,,
\end{equation}
and $ \partial_{i\alpha}^s \G_{\alpha \j}\prec (1+\phi)^{s-1} (|\G_{i\j}|+|\G_{\alpha \j}|)(|\G_{\alpha i}|+|\G_{ii}|+|\G_{\alpha\alpha}|)$. Thus Lemma \ref{lemma4.11} (iii) and \eqref{ward} implies
\begin{equation} \label{4.66}
\quad \frac{1}{N}\sum_{i\alpha}\partial_{i\alpha}^s \G_{\alpha \j}\prec (1+\phi)^{s-1} \frac{\im m+\phi}{\eta}+(1+\phi)^{s-1} \frac{(\im m+\phi)^{1/2}(|m|+\phi)N^{1/2}}{\eta^{1/2}}\,. 
\end{equation}
For $r\geq 2$, we split
\begin{align}
	Y_r\prec \frac{1}{Nq^{r-1}}\sum_{r_1=0}^r\sum_{i\alpha}\bb E (\partial_{i\alpha}^{r-r_1} \G_{\alpha j})(\partial^{r_1}_{i\alpha} \cal Q^{n-1})\eqd \sum_{r_1=0}^{r}Y_{r,r_1}\,.
\end{align}
When $r_1\geq 1$, by \eqref{4.4},  $ \partial_{i\alpha}^{r-r_1} \G_{\alpha \j}\prec (1+\phi)^{r-r_1} (|\G_{i\j}|+|\G_{\alpha \j}|)$ and \eqref{ward}, we get
\begin{align*}
	Y_{r,r_1}\prec \frac{(1+\phi)^{r-1}(\im m+\phi)}{\eta q^{r-1}}\sum_{a=1}^{r\wedge (n-1)} (\Psi +\psi)^a\bb E |\cal Q|^{n-1-a}\prec\sum_{a=1}^{n-1} (\im m+\phi)\eta^{-1}N^{-\delta a}(\Psi +\psi)^a\bb E |\cal Q|^{n-1-a}
\end{align*}
For $r_1=0$, we get from \eqref{4.66} that
\begin{align}
Y_{r,0}	& \prec\bigg( \frac{(1+\phi)^{r-1}(\im m+\phi)}{\eta q^{r-1}}+\frac{(1+\phi)^{r-1}(\im m+\phi)^{1/2}(|m|+\phi)N^{1/2}}{\eta^{1/2}q^{r-1}}\bigg)\bb E |\cal Q|^{n-1}\nonumber\\
	&\prec \bigg( \frac{(\im m+\phi)}{\eta }+\frac{(1+\phi)(\im m+\phi)^{1/2}(|m|+\phi)N^{1/2}}{\eta^{1/2}q}\bigg)\bb E |\cal Q|^{n-1}\prec \widetilde{\cal E}\bb E |\cal Q|^{n-1}\nonumber\,.
\end{align}
where in the second step we used $r \geq 2$, and in the third step we used 
$$
(1+\phi)(\im m+\phi)^{1/2}(|m|+\phi)N^{1/2}\eta^{-1/2}q^{-1}\prec (\im m+\phi)\eta^{-1}+(1+\phi)^2(|m|+\phi)^2Nq^{-2}\,.
$$ Thus
\begin{equation} \label{1234}
	Y_{r}\prec  \sum_{a=1}^n (\widetilde{\cal E}+\widetilde{\cal E}^{1/2}\Psi^{1/2}+N^{-\delta}
	\Psi)^a\bb E \cal |\cal Q|^{n-a}
\end{equation}
for all $r \geq 2$. Inserting \eqref{5.6}, \eqref{1234} into \eqref{4.412}, we get \eqref{5.5} as desired. This finishes the proof.
\end{proof}

\begin{proof}[Proof of Theorem \ref{theoremA}]
Fix $\nu\in (0,\delta/100)$. Suppose that $\max_{\i\j}|G_{\i\j}-M_{\i\j}|\prec \phi$ for some deterministic $\phi\in [N^{-1},N^{\nu}]$ at $(w,z)\in \b D_\delta$. By triangle inequality we get $\max_{\i\j}|G_{\i\j}|\prec 1+\phi$. Thus
	\begin{equation} \label{qwq}
		\max_{\j}\Big|\sum_{i} \widetilde{G}_{i\j}\Big|+\max_{\j}\Big|\sum_{\alpha} \widetilde{G}_{\alpha\j}\Big| \prec \psi
		\end{equation}
	for some deterministic $\psi\in [N^{-1},N^{1+\nu}]$. By Proposition \ref{prop5.2} (ii) and the resolvent identity, we get
	\[
\sum_{i} \big((H\widetilde{G})_{i\j}+\ul{\widetilde{G}}\widetilde{G}_{i\j}\big)=\delta_{\j\in \{1,2,...,N\}}+z\sum_{i}\G_{i\j}+\ul{\G}\sum_i\G_{i\j}+w\sum_{\alpha}\G_{\alpha\j}-f\sum_\alpha \G_{\alpha \j} \prec \widetilde{\cal E}\,,
	\]
and thus 
	\[
	\sum_{\alpha}\G_{\alpha \j} \prec f^{-1}(\widetilde{\cal E}+\psi)\prec (\im m+\phi) \eta^{-1} f^{-1}+(1+\phi)^2(|m|+\phi)^2Nq^{-2}f^{-1}+N^{-\delta}\psi\,.
	\]
Here in the last step we used $f\asymp N^{\xi}$. Similarly, we get the same estimate for $\sum_{i}\G_{i\j}$. As the last estimates hold uniformly in $\j$, we get
\begin{equation} \label{qqq}
\max_{\j}\Big|\sum_{i} \widetilde{G}_{i\j}\Big|+\max_{\j}\Big|\sum_{\alpha} \widetilde{G}_{\alpha\j}\Big| \prec (\im m+\phi) \eta^{-1}f^{-1} +(1+\phi)^2(|m|+\phi)^2Nq^{-2}f^{-1}+N^{-\delta}\psi
\end{equation}
provided \eqref{qwq} holds. Iterating \eqref{qqq} we get
\begin{equation} \label{5.9}
	\max_{\j}\Big|\sum_{i} \widetilde{G}_{i\j}\Big|+\max_{\j}\Big|\sum_{\alpha} \widetilde{G}_{\alpha\j}\Big| \prec (\im m+\phi) \eta^{-1}f^{-1} +(1+\phi)^2(|m|+\phi)^2Nq^{-2}f^{-1}
\end{equation}
at $(w,z)$. By Proposition \ref{prop5.2} (i) and resolvent identity, we get
\[
\max_{i\j} |\delta_{i\j}+z\G_{i\j}+\ul{\G}\G_{i\j}+wG_{i'\j}|\prec \cal E + \max_{\j}\Big|\sum_{\alpha} \frac{f}{N}\widetilde{G}_{\alpha\j}\Big| \prec \cal E
\]
at $(w,z)$. Similarly, $\max_{\alpha\j} |\delta_{\alpha\j}+z\G_{\alpha\j}+\ul{\G}\G_{\alpha\j}+\bar{w}G_{\alpha'\j}| \prec\cal E$ at $(w,z)$. The rest of the proof follows exactly as the steps in the proof of Theorem \ref{theorem 4.1}.
\end{proof}

Basing on Theorem \ref{theoremA}, have the following delocalization result.
\begin{lemma} \label{lemmadelocalization}
Let $|w|\leq \delta^{-1}$ for some fixed $\delta>0$. We denote the eigenvalues and corresponding $L^2$-normalized eigenvectors of $\widetilde{H}_w$ by $\pm \sigma_1,...,\pm \sigma_N$ and $\begin{pmatrix}
	{\b v_1} \\  \pm \b w_1
\end{pmatrix},...,\begin{pmatrix}
{\b v_N} \\  \pm \b w_N
\end{pmatrix}$ respectively. Here $\b v_i,\b w_i \in \bb C^{N}$ for all $i$. Then
\[
\max_i\|\b v_i\|_\infty+\max_i \|\b w_i\|_\infty \prec N^{-1/2+\varepsilon}\,.
\]
\end{lemma}
\begin{proof}
	W.L.O.G assume $\sigma_1,...,\sigma_N \geq 0$ and $\sigma_1=\max_i{\sigma_i}$. Using the moment method, it is not hard to show that $\|H_w\|=O(1)$ with very high probability (see e.g.\,\cite[Lemma 4.3]{EKYY1} for the proof of a similar result). Thus an application of Bauer-Fike Theorem shows that with very high probability
	\[
	|\sigma_1-f|=O(1) \quad \mbox{and} \quad \max_{2\leq i\leq N} |\sigma_i|=O(1)\,.
	\]
Hence Theorem \ref{theoremA} and a spectral decomposition implies
\[
\max_{i \geq 2}\|\b v_i\|_\infty+\max_{i \geq 2} \|\b u_i\|_\infty \prec N^{-1/2+\varepsilon}\,.
\]

Let $\b 0\deq (0,0,...,0)^*\in \bb R^{N}$ and $\b e\deq N^{-1/2}(1,1,...,1)^*\in \bb R^{N}$. In addition, set $\b x\deq \begin{pmatrix}
	{\mathbf e} \\ \b 0
\end{pmatrix}$ and $\b y\deq \begin{pmatrix}
	{\mathbf 0} \\ \b e
\end{pmatrix}$. The identity $\widetilde{H}_w=H_w+f \b x\b y^*+f\b y \b x^*$ yields
\[
\sigma_1 \begin{pmatrix}
	{\b v_1} \\  \b w_1
\end{pmatrix}= H_w \begin{pmatrix}
{\b v_1} \\  \b w_1
\end{pmatrix}+f (\b e, \b w_1)\b x+f (\b e, \b v_1)\b y\,.
\]
As $\sigma_1$ is not in the spectrum of $H_w$ with very high probability, we get
\begin{equation} \label{4.98}
\begin{pmatrix}
{\b v_1} \\  \b w_1
\end{pmatrix}=f\sigma_1^{-1}(\b e, \b w_1)(I-\sigma_1^{-1}H_w)^{-1}\b x+f\sigma^{-1}_1(\b e, \b v_1)(I-\sigma_1^{-1}H_w)^{-1} \b y\,.
\end{equation}
Using an argument similar to \cite[Lemma 7.10]{EKYY1}, it can be shown that
\begin{equation} \label{4.999}
	(I-\sigma_1^{-1}H_w)^{-1}\b x=\b x+\b \varepsilon_1 \quad \mbox{and}\quad  (I-\sigma_1^{-1}H_w)^{-1}\b y=\b y+\b \varepsilon_2\,,
\end{equation}
where $\|\b \varepsilon_1\|_\infty, \|\b \varepsilon_2\|_\infty \prec N^{-1/2}\sigma_1^{-1}$. As $\sigma_1\asymp f$ with very high probability, we get from \eqref{4.98} and \eqref{4.999} that
\[
\b v_1=(\b e , \b w_1) \b e+\b \varepsilon_3 \quad \mbox{and} \quad \b w_1=(\b e , \b v_1) \b e+\b \varepsilon_4\,,
\]
where $\|\b \varepsilon_3\|_\infty, \|\b \varepsilon_4\|_\infty \prec N^{-1/2}f^{-1}$. Thus $\|\b v_1-\b e/\sqrt{2}\|_\infty+\|\b w_1-\b e/\sqrt{2}\|_\infty \prec N^{-1/2}f^{-1}$. This finishes the proof.
\end{proof}

To prove delocalization Theorem \ref{radius and delocalization} (ii), we also need the following prior estimate on the extreme eigenvalues of $A$.
\begin{lemma} \label{lemma 5.1} 
	Let $\lambda_1,\lambda_2,...,\lambda_N$ be the eigenvalues of $A$ with $|\lambda_1|=\max_i |\lambda_i|$. Then with very high probability,
	\[
	|\lambda_1-f|=O(1) \quad \mbox{and} \quad \max_{2\leq i\leq N} |\lambda_i|=O(1)\,.
	\]
\end{lemma}
\begin{proof}
Recall the definition of $H$ in \eqref{H}.	Using the moment method, it is not hard to see that $\|H\|=O(1)$ with very high probability. Thus $\|B\|=O(1)$ with very high probability. The result then follows easily from the Bauer-Fike theorem.
\end{proof}	

\begin{proof}[Proof of Theorem \ref{radius and delocalization} (ii)]
Let us use the normalization that $\|\b u\|=1$. Set $\b 0\deq (0,0,...,0)^*\in \bb R^{N}$. It is easy to see that $0$ is an eigenvalue of $\widetilde{H}_{\lambda}$, with eigenvector 
$\begin{pmatrix}
	{\mathbf 0} \\ \b u
\end{pmatrix}\in \mathbb C^{2N}$.
 As Lemma \ref{lemma 5.1} shows that $\max_{2\leq i\leq N} |\lambda_i|=O(1)$ with very high probability, we can easily deduce from Theorem \ref{theoremA} and spectral decomposition that
	\[
	\bigg\|	\begin{pmatrix}
		{\mathbf 0} \\ \b u
	\end{pmatrix}\bigg\|_\infty \prec N^{-1/2}\,.
	\]
	This completes the proof.
\end{proof}

\subsection{The spectral radius of $A$} \label{secradius}
Recall the definition of $\b S^{(2)}_\delta$ in \eqref{3.111}. The main goal of this section is to prove the following improvement of Corollary \ref{cor4.2} outside the unit disc.

\begin{proposition} \label{prop4.5}
	Fix $\delta\in (0,\xi/100)$ and $\nu\in (0,\delta/100)$. We have
	\[
	\ul{\G}-m\prec \frac{1}{N^{1+\nu}\eta}
	\]
	uniformly for all $(w,\ii\eta)\in \b S^{(2)}_\delta$
\end{proposition}

From Lemma \ref{lemma 5.1}, we know that $A$ has a nontrivial eigenvalue $\lambda_1\in \bb C$ that satisfies $|\lambda_1-f|=O(1)$ with very high probability. Moreover, Proposition \ref{prop4.5} implies that for any fixed $\delta>0$, with very high probability, $A$ has no eigenvalues in the ring $\{w: 1+N^{-1/2+\delta}\leq|w|\leq \delta^{-1} \}$. Thus we deduce the following upper bound in Theorem \ref{radius and delocalization} (i).
\begin{corollary} \label{cor4.7}
	Let $\lambda_1,\lambda_2,...,\lambda_N$ be the eigenvalues of $A$ with $|\lambda_1|=\max_i |\lambda_i|$. We have
		\[
		\max_{ 2\leq i \leq N} |\lambda_i|\leq1+O_{\prec}(N^{-1/2})\,.
		\]
\end{corollary}

 Similar to \eqref{wardd} and \eqref{4q2}, we can apply Theorem \ref{theoremA} and Lemma \ref{lemmadelocalization} to improve \eqref{ward} to
\begin{equation} \label{warddd}
	\sum_{\i}|\G_{\i\j}|^2=\frac{\im \G_{\j\j}}{\eta}\prec \frac{\im \ul{\G}}{\eta} \prec \frac{|\ul{\G}- m|+\im m}{\eta}\,,
\end{equation}
and we also have
\begin{equation} \label{4q22}
	\G_{\i\i} \prec |\ul{\G}|\leq |\ul{\G}-m|+\im m\,.
\end{equation} 

To prove Proposition \ref{prop4.5}, recall that we have already obtained its analogue for $\ul{G}$ in Theorem \ref{thmHstrong}. The main difference between $G$ and $\G$ is in their resolvent identities. More precisely, the Green function $G$ satisfies
\[
\delta_{i\j}+\ii \eta G_{i\j}-(HG)_{i\j}+wG_{i'\j}=0\,,
\]
while for the Green function $\G$, we have
\begin{equation} \label{5.20}
	\delta_{i\j}+\ii \eta \G_{i\j}-(H\G)_{i\j}+w\G_{i'\j}=\frac{f}{N}\sum_\alpha \G_{\alpha \j}\,.
\end{equation}
The next result estimates the RHS of \eqref{5.20}, which is the key in showing Proposition \ref{prop4.5}.

\begin{lemma} \label{lemma4.7}
	Fix $\delta\in (0,\xi/100)$, and let $(w,\ii\eta)\in \b S_\delta$. Denote $g\deq \ul{G}$. Suppose that $|g-m|\prec \Lambda$ for some deterministic $\Lambda \in [N^{-1},N^{-1}\eta^{-1}]$ at $(w,\ii\eta)$. Then at $(w,
	\eta)$ we have
\begin{equation} \label{5.11}
	\max_{\j}\Big|\sum_{i} \widetilde{G}_{i\j}\Big|+\max_{\j}\Big|\sum_{\alpha} \widetilde{G}_{\alpha\j}\Big| \prec \frac{\im m+\Lambda}{f\eta}
\end{equation}
as well as
\begin{equation} \label{5.12}
\begin{aligned}
\cal K&\deq	\Big|\sum_{ij} \widetilde{G}_{ij}\Big|+\Big|\sum_{i\alpha} \widetilde{G}_{i\alpha}\Big|+\Big|\sum_{\alpha i} \widetilde{G}_{\alpha i}\Big| +\Big|\sum_{\alpha \beta} \G_{\alpha\beta}\Big| \prec (\im m+\Lambda)\eta^{-2}f^{-2}+Nf^{-1}\,.
\end{aligned}
\end{equation}
\end{lemma}

\begin{proof}	
	(i) We first prove \eqref{5.11}. By Theorem \ref{theoremA}, the LHS of \eqref{5.11} is stochastically dominated by $N$. Now suppose 
	\[
	\max_{\j}\Big|\sum_{i} \widetilde{G}_{i\j}\Big|+\max_{\j}\Big|\sum_{\alpha} \widetilde{G}_{\alpha\j}\Big| \prec \psi
	\]
	for some deterministic $\psi\in [N^{-1},N]$ at $(w,\ii\eta)$. We can repeat the proof of Proposition \ref{prop5.2} (ii), using \eqref{warddd} instead of \eqref{ward}, and together with the help of \eqref{4q22}, to show that
		\begin{equation} \label{Q}
			\begin{aligned}
			\cal Q_*&\deq \max_{\j}\Big|\sum_{i} \big((H\widetilde{G})_{i\j}+\ul{\widetilde{G}}\widetilde{G}_{i\j}\big)\Big|+\max_{\j}\Big|\sum_{\alpha}\big( (H\widetilde{G})_{\alpha\j}+\ul{\widetilde{G}}\widetilde{G}_{\alpha\j}\big)\Big|\\
			&\prec (\im m+\Lambda) \eta^{-1} +(\im m+\Lambda)^2Nq^{-2}+\psi \prec (\im m+\Lambda) \eta^{-1}+\psi\,.
			\end{aligned}
		\end{equation}
Here in the last step we used Lemma \ref{lemma4.11} (v). The rest of the proof is very close to the derivation of \eqref{5.9}; we omit the details. Note that \eqref{5.11} and \eqref{Q} also implies
\begin{equation} \label{QQQ}
	\cal Q_* \prec (\im m+\Lambda)\eta^{-1}\,.
\end{equation} 

(ii) Now we prove \eqref{5.12}. Suppose at $(w,\ii\eta)$ we have $\cal K \prec \varphi$ for some deterministic $\varphi \in [1,N^2]$. We shall show \eqref{5.12} by proving that
\begin{equation} \label{5.13}
	\begin{aligned}
		\cal S_*\deq &\,\Big|\sum_{ij} \big((H\widetilde{G})_{ij}+\ul{\G}\G_{ij})\Big|+\Big|\sum_{i\alpha} \big((H\widetilde{G}\big)_{i\alpha}+\ul{\G}\G_{i \alpha }\big)\Big|+\Big|\sum_{\alpha i} \big((H\widetilde{G})_{\alpha i}+\ul{\G}\G_{\alpha i}\big)\Big|\\
		&+\Big|\sum_{\alpha \beta} \big((H\widetilde{G})_{\alpha \beta}+\ul{\G}\G_{\alpha \beta}\big)\Big|\prec (\im m+\Lambda)\eta^{-2}f^{-1}+N +\varphi\eqd  \widehat{\cal E}\,. 
	\end{aligned}
\end{equation}
Indeed, as the resolvent identity gives 
\[
\sum_{ij}\big((H\widetilde{G})_{ij}+\ul{\G}\G_{ij}
\big)=N+\ii \eta \sum_{ij} \G_{ij} +\sum_{ij}\ul{\G}\G_{ij}+w\sum_{\alpha j}\G_{\alpha j} -f\sum_{\alpha j}\G_{\alpha j} \prec \widehat{\cal E} \,,
\]
which implies 
\[
\Big|\sum_{\alpha j}\G_{\alpha j}\Big| \prec (\im m+\Lambda)\eta^{-2}f^{-2}+Nf^{-1} +f^{-1}\varphi\,.
\]
The same estimate applies to the other three terms in $\cal K$. Thus 
\begin{equation} \label{5.144}
 \cal K \prec (\im m+\Lambda)\eta^{-2}f^{-2}+Nf^{-1} +f^{-1}\varphi
\end{equation}
provided that $\cal K \prec \varphi$. Iterating \eqref{5.14} we get the desired result.

Now suppose $\cal S_* \prec \widehat{\Psi}$ for some deterministic $\widehat{\Psi}\in [1,N^3]$. Let $\cal S\deq \sum_{ij}((HG)_{ij}+\ul{\G}\G_{ij})$, and fix an even integer $n \geq 2$. We shall prove \eqref{5.13} by showing that
\begin{equation} \label{5.14}
	\bb E |\cal S|^{n} \prec \sum_{a=1}^n (\widehat{\cal E}+\widehat{\cal E}^{1/2}\widehat{\Psi}^{1/2}+N^{-\delta}\widehat{\Psi} )^n \bb E |\cal S|^{n-a}\,.
\end{equation}
Indeed, \eqref{5.14} implies that $|\sum_{ij}((HG)_{ij}+\ul{\G}\G_{ij})|\prec \widehat{\cal E}+\widehat{\cal E}^{1/2}\widehat{\Psi}^{1/2}+N^{-\delta}\widehat{\Psi}$.
Then by estimating other three terms in $\cal S_*$ in a similar fashion, we get
\begin{equation} \label{5.15}
\cal S_* \prec \widehat{\cal E}+\widehat{\cal E}^{1/2}\widehat{\Psi}^{1/2}+N^{-\delta}\widehat{\Psi}
\end{equation}
provided that $\cal S_* \prec \widehat{\Psi}$. Iterating \eqref{5.15} finitely many time we get \eqref{5.13} as desired. Moreover, as complex conjugates play no role in the subsequent analysis, we shall ignore it on LHS of \eqref{5.14} and prove 
\begin{equation} \label{5.16}
	\bb E \cal S^{n} \prec \sum_{a=1}^n (\widehat{\cal E}+\widehat{\cal E}^{1/2}\widehat{\Psi}^{1/2}+N^{-\delta}\widehat{\Psi} )^n \bb E |\cal S|^{n-a}
\end{equation}
instead. By Lemma \ref{cumulant} we get
\begin{equation}
	\begin{aligned}
		\bb E \cal S^n&=\sum_{r=1}^{\ell}\sum_{ij\alpha}\cal C_{r+1}(H_{i\alpha})\bb E\partial_{i\alpha}^r(\widetilde{G}_{\alpha j}\cal S^{n-1})+\sum_{ij} \bb E\ul{\widetilde{G}}\widetilde{G}_{ij}\cal S^{n-1}+O_\prec(N^{-10n})\\
		&\eqd \sum_{r=1}^{\ell} Z_r+\sum_{ij}\bb E \ul{\widetilde{G}}\widetilde{G}_{ij}\cal S^{n-1}+O_\prec(N^{-10n})\,.
	\end{aligned}
\end{equation}
By \eqref{diff}, we have
\begin{align}
	Z_1=&\,-\sum_{ij} \bb E\ul{\widetilde{G}}\widetilde{G}_{ij}\cal S^{n-1}-\frac{1}{N} \sum_{ij\alpha}\bb E \widetilde{G}_{\alpha i} \widetilde{G}_{\alpha j}\cal S^{n-1}+(n-1)\sum_{jl\alpha}\bb E \widetilde{G}_{\alpha j}\G_{\alpha l}\cal S^{n-2}\nonumber\\
	&\,-\frac{n-1}{N}\sum_{jl\alpha}\bb E \widetilde{G}_{\alpha j}\G_{\alpha l} \sum_{ik}((H\widetilde{G})_{ki}+\ul{\widetilde{G}}\widetilde{G}_{ki})\cal S^{n-2}-\frac{n-1}{N}\sum_{ijl\alpha}\bb E \widetilde{G}_{\alpha j}\G_{i l} \sum_{k}((H\widetilde{G})_{k\alpha}+\ul{\widetilde{G}}\widetilde{G}_{k\alpha})\cal S^{n-2}\nonumber\\
	&-\frac{2(n-1)}{N^2}\sum_{ij\alpha}\bb E \widetilde{G}_{\alpha j} ((\G^2)_{i\alpha}+(\G^2_{\alpha i}))\sum_{kl}\G_{kl}\cal S^{n-2}\nonumber\\
	=&\,-\sum_{ij}\bb E\ul{\widetilde{G}}\widetilde{G}_{ij}\cal S^{n-1}+O_{\prec}((\im m+\Lambda)^2f^{-2}\eta^{-2}) \bb E |\cal S|^{n-1}+O_{\prec}(N(\im m+\Lambda)^2f^{-2}\eta^{-2}) \bb E |\cal S|^{n-2}\nonumber\\
	&\,+O_{\prec}((\im m+\Lambda)^{2}f^{-2}\eta^{-2}\widehat{\Psi}) \bb E |\cal S|^{n-2}+O_\prec((\im m+\Lambda)f^{-1}\eta^{-1}\varphi(\im m+\Lambda)\eta^{-1})\bb E |\cal S|^{n-2}\nonumber\\
	&\,+O_{\prec}((\im m+\Lambda)^2f^{-1}\eta^{-2}\varphi)\bb E |\cal S|^{n-2}\nonumber\,,
\end{align}
where in the second step we used \eqref{5.11}, \eqref{QQQ}, and $\cal K\prec\varphi$. Thus
\begin{equation} \label{5.66}
	Z_1	+\sum_{ij}\bb E\ul{\widetilde{G}}\widetilde{G}_{ij}\cal S^{n-1}=O_\prec(\widehat{\cal E}) \bb E |\cal S|^{n-1}+O_\prec(\widehat{\cal E}^2+\widehat{\cal E}\widehat{\Psi}) \bb E |\cal S|^{n-2}\,.
\end{equation}
For $r\geq 2$, we have
\begin{equation*}
	\begin{aligned}
		Z_r&\prec \frac{1}{Nq^{r-1}}\sum_{r_1=0}^r\sum_{ij\alpha}\bb E (\partial_{i\alpha}^{r-r_1} \G_{\alpha j})(\partial^{r_1}_{i\alpha} \cal S^{n-1}) \eqd \sum_{r_1=0}^rY_{r,r_1}\,.
	\end{aligned}
\end{equation*}
Note that for $s\geq 1$, by \eqref{5.11} and \eqref{QQQ}, 
\begin{equation} \label{4.28}
	\begin{aligned}
\partial_{i\alpha}^s \cal S &\prec (\im m+\Lambda)^2f^{-1}\eta^{-2}+(\im m+\Lambda)f^{-1}\eta^{-1}+\frac{(\im m+\Lambda)}{N\eta}(\varphi+(\im m+\Lambda)^2f^{-2}\eta^{-2}) \\
&\prec (\im m+\Lambda)\widehat{\cal E}
	\end{aligned}
\end{equation}
and
\begin{equation} \label{4.299}
\sum_{\alpha j} \partial^s_{i\alpha}G_{\alpha j} \prec \bigg( \frac{\im m+\Lambda}{f\eta}\bigg)\bigg(\sqrt{\frac{N(\im +\Lambda)}{\eta}}+N(\im m+\Lambda)\bigg) \prec N^{\delta}\widehat{\cal E}\,.
\end{equation}
From \eqref{4.299} we know
\[
Z_{r,0} \prec \frac{1}{q^{r-1}}\widehat{\cal E}\bb E |\cal S|^{n-1} \prec \widehat{\cal E} \bb E|\cal S|^{n-1}\,.
\]
By \eqref{5.11} and \eqref{4.28} we have
\[
Z_{r,r_1} \prec \frac{N}{q^{r-1}} \frac{\im m+\Lambda}{f\eta}(\im m+\Lambda) \sum_{a=1}^{(n-1)\wedge r}\widehat{\cal E}^{a} \bb E |\cal S|^{n-1-a} \prec \sum_{a=1}^{n-1}\widehat{\cal E}^{a}\bb E |\cal S|^{n-1-a}\,.
\]
for all $r_1\geq 1$. Combining the above two result we have $Z_r\prec \sum_{a=0}^{n}\widehat{\cal E}^{a}\bb E |\cal S|^{n-a}$ for all $r\geq 2$. Together with \eqref{5.66}, we get \eqref{5.12} as desired. This finishes the proof.
\end{proof}

Proposition \ref{prop4.5} now follows immediately from the following analogue of Proposition \ref{prop4.1} and Lemma \ref{lemma4.4}, together with the stability analysis presented in Section \ref{sec4.3}.

\begin{lemma} \label{lemma4.8}
	Fix $\delta\in (0,\xi/100)$, and let $(w,\ii\eta)\in \b S_\delta$. Denote $\tilde{g}\deq \ul{\G}$. Suppose that $|\tilde{g}-m|\prec \Lambda$ for some deterministic $\Lambda \in [N^{-1},N^{-1}\eta^{-1}]$ at $(w,\ii\eta)$. Then at $(w,
	\eta)$ we have
	\begin{equation} \label{4.244}
	\max_{\j}	\frac{1}{N}\sum_{\i} |\G_{\i \j}|^4 \prec  \Big(\frac{\Lambda+\im m}{N\eta}\Big)^2+\frac{1}{N}
	\end{equation}
	and
	\begin{equation} \label{4.245}
\hspace{-0.08cm}	P(\tilde{g}) \prec \frac{(\Lambda+\im { m})^2}{N\eta}+\frac{(\Lambda+\im m)^{1/2}}{N^{5/2}\eta^{5/2}}+\frac{(\Lambda+\im m)^{1/2}\kappa^{3/4}}{N\eta}+\frac{\Lambda^3+(\im {m})^3+\eta+\eta^{1/3}\kappa}{q^2}+\frac{1}{N}\,.
    \end{equation}
\end{lemma}
\begin{proof}
	Observe from \eqref{5.20} that the main difference between the proofs of \eqref{4.244}, \eqref{4.245} and those of Lemma \ref{lemma4.4}, Proposition \ref{prop4.1} is the use of resolvent identity. Thanks to Lemma \ref{lemma4.7}, we have a sufficient estimate of the RHS of \eqref{5.20}, which leads to 
\begin{equation} \label{5.21}
	\delta_{i\j}+\ii \eta \G_{i\j}-(H\G)_{i\j}+w\G_{i'\j} \prec \frac{\im m+\Lambda}{N\eta}\,.
\end{equation}
Similarly, we also have
\begin{equation} \label{5.22}
	\delta_{\alpha\j}+\ii \eta \G_{\alpha\j}-(H\G)_{\alpha\j}+\bar{w}\G_{\alpha'\j} \prec \frac{\im m+\Lambda}{N\eta}\,,
\end{equation}
and
\begin{equation} \label{5.23}
	1+\ii \eta \ul{\G}-\ul{H\G}+\frac{w}{N}\sum_i \G_{i'i}=\frac{f}{N^2}\sum_{i\alpha}\G_{\alpha i}\prec \frac{\im m+\Lambda}{N^2\eta^2}+\frac{1}{N}\,,
\end{equation}
as well as
\begin{equation} \label{5.24}
	\frac{\ii \eta}{N}\sum_{i'}\G_{i'i}-\frac{1}{N}\sum_{i}(H\G)_{i'i}-\frac{1}{N}\sum_{i}(H\G)_{i'i}+\bar{w}\ul{G}\prec \frac{\im m+\Lambda}{N^2\eta^2}+\frac{1}{N}\,.
\end{equation}
Using \eqref{5.21} -- \eqref{5.24}, the proofs \eqref{4.244} and \eqref{4.245} are essentially identical to those of Lemma \ref{lemma4.4} and Proposition \ref{prop4.1} respectively; we omit the details.
\end{proof}

\section{Edge universality of $A$.} \label{sec5} 

In this section we prove Theorem \ref{main theorem}. By Girko's Hermitization \eqref{1.5}, our computation consists of the analysis of $\widetilde{G}$ in the \textit{subcritical regime} $\eta\leq N^{-3/4-\delta}$, the \textit{critical regime} $\eta\in [N^{-3/4-\delta},N^{-3/4+\delta}]$ and the \textit{supercritical regime} $\eta \geq N^{-3/4+\delta}$. 

In the subcritical regime, thanks to the optimal averaged law Corollary \ref{cor4.2} (ii), the steps are essentially identical to those in the dense case \cite{CEK3}. 

In the critical regime, we introduce the matrix flow \eqref{matrix flow} as in \cite{CEK3}. Our starting point is Lemma \ref{lemma 5.2}, which states that the results proved in Section \ref{sec4.4} for $\widetilde{G}\equiv \widetilde{G}(0)$ remain valid for all $\widetilde{G}(t)$, $t\geq 0$. Then we are able to bound $\ul{\G}(0)-\ul{\G}(\infty)$ in Lemma \ref{lemma5.3}. Note that $\G(\infty)$ is associate with the Ginibre ensemble with a rank-one perturbation, and we still need to compare it with the centered model in Lemma \ref{lemma5.4}. Combining Lemmas \ref{lemma 5.2} -- \ref{lemma5.4} settles the critical regime.

The supercritical regime turns out to be the most difficult, as the optimal estimate $\ul{\G}-m\prec (N\eta)^{-1}$ fails for large $\eta$ (see Corollary \ref{cor4.2}). The key idea here is to use integration by parts formula for the shift variable $w$ inside the Girko's Hermitization \eqref{5.28}, which avoids the study of $\widetilde{G}$ for $\eta>N^{-3/4+\delta}$. In addition, we are able to prove the $N$th sub-trace of $\widetilde{G}$ satisfies a strong rigidity estimate (Proposition \ref{prop5.1}), which is required by the RHS of \eqref{5.28}. As a result, we are able to show that the supercritical regime is negligible \eqref{6.6}.

Combining the above results yields the proof of Theorem \ref{main theorem}.

\begin{figure}
	\begin{tikzpicture}[>=stealth,every node/.style={shape=rectangle,draw,rounded corners, minimum width=2.9cm,},]
		\node[very thick,] (l2) { 
				Theorem~\ref{main theorem}}; 

		\node[very thick, minimum width=2.7cm] (l5)[right=of l2]{Proposition~\ref{prop5.5} }; 

		\node[very thick] (l7) [right=of l5]{Lemma~\ref{lemma5.3}}; 
		
			\node[very thick] (l8) [above=of l7]{Lemma~\ref{lemma5.4}}; 
			
						\node[very thick] (l9) [below=of l7]{Lemma~\ref{lemma5.6}}; 

		\node[very thick] (l10) [right=of l7]{Lemma~\ref{lemma 5.2}}; 

			\node[very thick] (l11) [right=of l9, fill=, fill=blue!30]{Proposition~\ref{prop5.1}}; 

				\node[very thick] (l12) [below=of l11]{Corollary~\ref{cor4.2}}; 

		
		\draw[->,  line width=.6mm] (l5) to[out=180,in=0] (l2);
		
		\draw[->,  line width=.6mm] (l7) to[out=180,in=0] (l5);
		
			\draw[->,  line width=.6mm] (l8) to[out=180,in=3mm] (l5);

				\draw[->,  line width=.6mm] (l9) to[out=180,in=-3mm] (l5);

		\draw[->,  line width=.6mm] (l10) to[out=180,in=0] (l7);
		
		\draw[->,  line width=.6mm] (l11) to[out=180,in=0] (l9);
		
				\draw[->,  line width=.6mm] (l12) to[out=180,in=-3mm] (l9);

	\end{tikzpicture}
	\caption{Proof of Theorem \ref{main theorem}}
	\label{fig:proofthm1.1}
\end{figure}

The last result remains to be proved is the lower bound in Theorem \ref{radius and delocalization} (i). This is a direct consequence of Theorem \ref{thmcircularlaw} -- a mesoscopic density law near the edge. To prove Theorem \ref{thmcircularlaw}, one only needs to lift \eqref{6.4} and Lemma \ref{lemma5.6} to the mesoscopic scale.

\begin{proposition} \label{prop5.1}
	Fix $\delta\in (0,\xi/100)$. We have
	\[
	\frac{1}{N}\sum_{i} \G_{i'i}+\frac{1+m^2}{w} \prec \frac{1}{N^2\eta^2}+\frac{1}{N\eta^{2/3}}
	\]
	uniformly for $(w,\ii\eta) \in \b S_{\delta}$. 
\end{proposition}

\begin{proof}
	From Corollary \ref{cor4.2}, we know that $|\ul{\G}-m|\prec N^{-1}\eta^{-1}\eqd \Lambda$. Fix an even integer $n$. Using Lemma \ref{cumulant} and a argument similar to the proof of Proposition \ref{prop4.1}, it is not hard to get the recursive estimate 
	\begin{equation} \label{6.2}
	\bb E |\ul{H\G}+\ul{\G}^2|^n \prec \sum_{a=1}^n \bigg(\frac{\im m+\Lambda}{N\eta}+\frac{(\im m+\Lambda)^{1/2}}{(N\eta)^{3/2}}+\frac{(\im m+\Lambda)^2}{q^2}+\frac{1}{N\eta^{2/3}}\bigg)^a \bb E |\ul{H\G}+\ul{\G}^2|^{n-a}\,.
	\end{equation}
	In fact, \eqref{6.2} is much easier to prove than \eqref{3.3}, as it does not require the exploration of cusp fluctuation. By \eqref{6.2} and Lemma \ref{lemma4.11}, we get 
	\begin{equation} \label{6.1}
		\ul{H\G}+\ul{\G}^2 \prec \frac{\im m+\Lambda}{N\eta}+\frac{(\im m+\Lambda)^{1/2}}{(N\eta)^{3/2}}+\frac{(\im m+\Lambda)^2}{q^2} \frac{1}{N\eta^{2/3}}\leq \frac{C}{N^2\eta^2}+\frac{C}{N\eta^{2/3}}\,.
	\end{equation}
The resolvent identity, Corollary \ref{cor4.2} and Lemma \ref{lemma4.7} imply
\begin{equation*}
	\ul{H\G}+\ul{\G}^2=1+\ii \eta\ul{\G}+\ul{\G}^2+\frac{w}{N}\sum_i \G_{i'i}-\frac{f}{N^2}\sum_{\alpha i}\G_{\alpha i}=1+m^2+\frac{w}{N}\sum_i \G_{i'i}+O_\prec\Big(\frac{1}{N^2\eta^2}+\frac{1}{N\eta^{2/3}}\Big)\,.
\end{equation*} 
Combining the above with \eqref{6.1} and $|w|^{-1}\leq 2$, we conclude the proof.
\end{proof}

\subsection{Matrix flows}
Let $W\in \bb R^{N\times N}$ denote the real Ginibre ensemble, i.e.\,$W_{ij} (1\leq i,j\leq N)$ are i.i.d.\,with $W_{ij}\overset{d}{=}\cal N(0,1/N)$. We also assume $W$ and $A$ are independent. Denote 
\[
B(t)\deq \e^{-t/2}B+\sqrt{1-\e^{t/2}}W\quad \mbox{and} \quad A(t)\deq B(t)+ f\b e \b e^*
\]
for any $t\in [0,\infty]$. It is easy to see that $A(0)=A$ and $A(\infty)=W+f\b e \b e^*$. Accordingly, for $w \in \bb C$ and $\eta>0$, we define the Hermitization of $A(t)$ and its Green function by
\begin{equation} \label{matrix flow}
\widetilde{H}_w(t)\deq
\begin{pmatrix}
	0&  A(t)-w\\
	A^*(t)-\bar{w}&0
\end{pmatrix} \quad \mbox{and} \quad \G_w(t;\ii\eta)\deq (\widetilde{H}_w(t)-\ii \eta)^{-1}
\end{equation}
respectively. In $\widetilde{G}(\cdot,\cdot)$, the first variable is real and the second variable is purely imaginary. We sometimes use the abuses of notations $\widetilde{G}(t)\equiv \widetilde{G}(t,\cdot)$ and $\widetilde{G}(\ii \eta)\equiv \widetilde{G}(\cdot,\ii\eta)$.  In addition, 
\begin{equation*}
\widetilde{H}(t)\deq \widetilde{H}_0(t)\,, \quad H(t)\deq \begin{pmatrix}
	0&  B(t)\\
	B^*(t)&0
\end{pmatrix}\,,\quad \mbox{and} \quad 	\partial_{\i\j} F\deq \frac{\partial F}{\partial H_{\i\j}(t)}
\end{equation*}
for differentiable functions $F$ of $\widetilde{H}_w(t)$. It is easy to check that
\begin{equation} \label{diff3}
\partial_{\i\j} \widetilde{G}(t)_{\k\l}=-\widetilde{G}_{\k\i}(t)\widetilde{G}_{\j\l}(t)-\widetilde{G}_{\k\j}(t)\widetilde{G}_{\i\l}(t)\,.
\end{equation}

Note that the entries of $B(t)$ satisfies $\bb E B_{ij}(t)=0$, $\var(B_{ij}(t))=N^{-1}$ and $\cal C_k(B_{ij}(t))=O_k(1/(Nq^{k-2}))$ for all fixed $k\geq 3$. As a result, everything we have proved so far for $\G$, we can repeat exactly the same proof for $\G(t)$.

\begin{lemma} \label{lemma 5.2}
Let $t \in [0,\infty]$. All results stated in Section \ref{sec4.4} concerning $\G=\G(0)$ also hold for $\G(t)$.
\end{lemma}

Next we would like to control the number of eigenvalues of $\widetilde{H}_w$ with size less than $N^{-3/4}$.

\begin{lemma} \label{lemma5.3}
For $(w,\ii\eta) \in \b S_\delta^{(1)}$, we have
\begin{equation} \label{blb}
\bb E\frac{\dd \ul{\G}(t)}{\dd t} \prec  N^{10\delta}\Big(\frac{\eta^{1/3}}{\e^{t/2}q}+\frac{1}{\e^{t/2}N\eta q}\Big) \eqd \cal E_4
\end{equation}
and
\begin{equation} \label{121212}
	\bb E\frac{\dd |\ul{\G}(t)-\ul{\G}(\infty)|^2}{\dd t} \prec  N^{10\delta}\Big(\frac{1}{\e^{t/2}N\eta^{2/3}q}+\frac{1}{\e^{t/2}N^2\eta^2 q}\Big) 
\end{equation}
for all $t\in [0,\infty]$. 
\end{lemma}

\begin{proof}
For simplicity we shall not write the parameter $t$ in $\G$. Note that
\begin{equation} \label{6.3}
\bb E\frac{\dd \ul{\G}}{\dd t} = -\frac{1}{2N}\sum_{\alpha i}\bb E \dot A_{\alpha i}(t)(\G^2)_{i\alpha }-\frac{1}{2N}\sum_{i\alpha} \bb E\dot A_{i\alpha} (t)(\G^2)_{\alpha i}\,.
\end{equation}
By Lemma \ref{cumulant}, we have
	\begin{align}
 &\quad -\frac{1}{2N}\sum_{\alpha i}\bb E \dot A_{\alpha i}(t)(\G^2)_{i\alpha }\nonumber\\
 &=\frac{\e^{-t/2}}{4N}\sum_{\alpha i}\bb E B_{\alpha i}(\G^2)_{i\alpha }-\frac{\e^{-t}}{4N\sqrt{1-\e^{-t}}}\sum_{\alpha i} \bb E W_{\alpha i}(\widetilde{G}^2)_{i\alpha}\nonumber\\
 &=\frac{\e^{-t/2}}{4N}\sum_{r=1}^{\ell}\sum_{\alpha i}\frac{1}{r!}\cal C_{r+1}(B_{\alpha i}) \e^{-r/2}\bb E \partial_{\alpha i}^r (\G^2)_{i\alpha}-\frac{\e^{-t}}{4N^2}\sum_{\alpha i} \bb E  \partial_{\alpha i} (\G^2)_{i\alpha}+O(\e^{-r/2}N^{-1})\nonumber\\
 &=\frac{\e^{-t/2}}{4N}\sum_{r=2}^{\ell}\sum_{\alpha i}\frac{1}{r!}\cal C_{r+1}(B_{\alpha i}) \e^{-r/2}\bb E \partial_{\alpha i}^r (\G^2)_{i\alpha}+O(\e^{-r/2}N^{-1})\eqd \sum_{r=2}^\ell U_r+O(\e^{-r/2}N^{-1})\label{7.4}\,.
 \end{align}

When $r=2$, we have
\begin{equation*}
	\begin{aligned}
U_2=&\,O(\e^{-t}N^{-2}q^{-1}) \sum_{\alpha i}\bb E (\G^2)_{\alpha\alpha}\G_{ii}\G_{\alpha i}+O(\e^{-t}N^{-2}q^{-1}) \sum_{\alpha i}\bb E (\G^2)_{ii}\G_{\alpha\alpha}\G_{\alpha i}\\
&+O(\e^{-t}N^{-2}q^{-1}) \sum_{\alpha i}\bb E (\G^2)_{\alpha i}\G^2_{\alpha i}+O(\e^{-t}N^{-2}q^{-1}) \sum_{\alpha i}\bb E (\G^2)_{\alpha i}\G_{\alpha \alpha}\G_{ii}\\
&+O(\e^{-t}N^{-2}q^{-1}) \sum_{i\alpha }\bb E (\G^2)_{\alpha i}\G_{\alpha \alpha}\G_{ii}\eqd\, U_{2,1}+U_{2,2}+U_{2,3}+U_{2,4}+U_{2,5}\,.
	\end{aligned}
\end{equation*}
By the resolvent identity $\bar w\G_{\alpha i}=(\G H)_{\alpha i'}+fN^{-1}\sum_j\G_{\alpha j}-\delta_{\alpha i'}-\ii \eta \G_{\alpha i'}$ and $|w|^{-1}\leq 2$, we get
	\begin{align}	
	\hspace{-0.3cm}	U_{2,1}&=O(\e^{-t}N^{-2}q^{-1}) \sum_{\alpha ij}\bb E (\G^2)_{\alpha\alpha}\G_{ii}\G_{\alpha j }H_{ji'}+O(\e^{-t}N^{-3}q^{-1}f) \sum_{\alpha 
			i j}\bb E (\G^2)_{\alpha\alpha}\G_{ii}\G_{\alpha j}\nonumber\\ 
		&=O(\e^{-t}N^{-2}q^{-1}) \sum_{\alpha ij}\bb E (\G^2)_{\alpha\alpha}\G_{ii}\G_{\alpha j }H_{ji'}+O_{\prec}(\cal E_4)\label{U21}\,.
		\end{align}
Here in the second step we used 
\[
\sum_j\G_{\alpha j} \prec \frac{\im m+1/(N\eta)}{\eta f}\,,
\] 
which is deduced from Lemmas \ref{lemma4.7} and \ref{lemma 5.2}. We expand the first term on RHS of \eqref{U21} by Lemma \ref{cumulant}, i.e.
\[
U_{2,1}=O(\e^{-t}N^{-2}q^{-1}) \sum_{k=1}^{\ell}\frac{1}{k!}\cal C_{k+1}(H_{ji'})\sum_{\alpha ij}\bb E \partial_{ji'}^k\big((\G^2)_{\alpha\alpha}\G_{ii}\G_{\alpha j })+O_{\prec}(\cal E_4)\,.
\]
We then apply \eqref{diff} and estimate the results using \eqref{warddd}, \eqref{4q22}, Lemmas \ref{lemma4.7} and \ref{lemma 5.2}. Here the off-diagonal terms $\G_{\alpha i'},\G_{\alpha j}, \G_{ij},\G_{i'j}$ can be estimated through \eqref{warddd} and Lemma \ref{lemma4.7}, and the diagonal terms of $\G$ are bounded by $G_{\i\i}\prec \eta^{1/3}+1/(N\eta)$ using \eqref{4q22}. However, as $\G_{ii'},\G_{i'i} \asymp 1$, the worst case happens when the index $i'$ in $H_{ji'}$ and the index $i$ in $\G_{ii}$ are matched. This leads to 
\begin{equation} \label{7.7}
	\begin{aligned}
U_{2,1}=&\,	O(\e^{-t}N^{-3}q^{-1}) \sum_{\alpha ij}\bb E (\G^2)_{\alpha\alpha}\G_{ij} \G_{i' i}\G_{\alpha j }+O(\e^{-t}N^{-3}q^{-1}) \sum_{\alpha ij}\bb E (\G^2)_{\alpha\alpha}\G_{ji} \G_{i i'}\G_{\alpha j }\\
&+O(\e^{-t}N^{-3}q^{-2}) \sum_{\alpha ij}\bb E (\G^2)_{\alpha\alpha}\G_{jj} \G_{i' i}\G_{ii'}\G_{\alpha j} +O_{\prec}(\cal E_4)\,.
\end{aligned}
\end{equation} 
The first three terms on RHS of \eqref{7.7} cannot be naively bounded by $O_{\prec}(\cal E_4)$.  Luckily, we can proceed by again applying the resolvent identity $w\G_{i'i}=(H\G)_{ii}+fN^{-1}\sum_{\beta} \G_{\beta i}-1-\ii \eta \G_{ii}$ and Lemma \ref{cumulant}. Similar to \eqref{U21} and \eqref{7.7}, we have
\begin{equation*} 
	\begin{aligned}
		U_{2,1}=&\,	O(\e^{-t}N^{-3}q^{-1}) \sum_{\alpha ij}\bb E (\G^2)_{\alpha\alpha}\G_{ij} \G_{\alpha j }+O(\e^{-t}N^{-3}q^{-1}) \sum_{\alpha ij}\bb E (\G^2)_{\alpha\alpha}\G_{ji} \G_{\alpha j }\\
		&+O(\e^{-t}N^{-2}q^{-2}) \sum_{\alpha j}\bb E (\G^2)_{\alpha\alpha}\G_{jj} \G_{\alpha j} +O_{\prec}(\cal E_4) \eqd U_{2,1,1}+U_{2,1,2}+U_{2,1,3}\,.
	\end{aligned}
\end{equation*}
Now for the term $U_{2,1,1}$, if we again use resolvent identity on $\G_{ij}$ and expand via Lemma \ref{cumulant}, there will no longer be other Green functions that matches the index $i$, and we can get $U_{2,1,1}\prec \cal E_4$. The same holds for $U_{2,1,2}$. As a result, we have
\begin{equation} \label{7.8}
	U_{2,1}=U_{2,1,3}+O_{\prec}(\cal E_4)=O(\e^{-t}N^{-2}q^{-2}) \sum_{\alpha j}\bb E (\G^2)_{\alpha\alpha}\G_{jj} \G_{\alpha j}+O_{\prec}(\cal E_4)\,.
\end{equation}
Note that $U_{2,1,3}$ and $U_{2,1}$ are similar in structure, yet $U_{2,1,3}$ is $q^{-1}$ times smaller than $U_{2,1}$, due to the extra factor $q^{-1}$. We can then iterate \eqref{7.8} finitely many times, and get $U_{2,1}\prec \cal E_4$ as desired. Similar arguments also work for $U_{2,2},...,U_{2,5}$. Thus $U_2\prec \cal E_4$.

When $r=3$, by \eqref{diff3}, Lemma \ref{lemma 5.2}, \eqref{warddd} and \eqref{4q22} we get
\begin{equation} \label{7.6}
	\begin{aligned}
	U_{3}&=O(\e^{-t}N^{-2}q^{-1}) \sum_{\alpha i}\bb E (\G^2)_{\alpha i}\G^3_{\alpha i}+O_{\prec}(\cal E_4)\\
	&=O(\e^{-t}N^{-2}q^{-1}) \cdot \frac{\eta^{1/3}+1/(N\eta)}{\eta}\sum_{\alpha i} \bb E|\G_{\alpha i}^3|+O_{\prec}(\cal E_4)\,.
	\end{aligned}
\end{equation}
In addition, Lemmas \ref{lemma4.8} and \ref{lemma 5.2} imply
\begin{equation} \label{7.9}
\sum_{i} \bb E |\G_{\alpha i}^3| \prec N\cdot \Big(\frac{\im m+1/(N\eta)}{N\eta}\Big)^{3/2}\,,
\end{equation}
and combining \eqref{7.6} and \eqref{7.9} we get $U_3\prec \cal E_4$. 

When $r \geq 4$, the estimates of $U_r$ are easier than that of $U_4$, due to the decay of cumulants. As a result, from \eqref{7.4} we have 
\[
-\frac{1}{2N}\sum_{\alpha i}\bb E \dot A_{\alpha i}(t)(\G^2)_{i\alpha } \prec \cal E_4\,.
\]
Due to symmetry, the second term on RHS of \eqref{6.3} can be handed in the same way; this finishes the proof of \eqref{blb}.

For the proof of \eqref{121212}, we have
	\begin{align}
&\, \bb E\frac{\dd  |\ul{\G}(t)-\ul{\G}(\infty)|^2}{\dd t} =	\bb E\frac{\dd  [(\ul{\G}(t)-\ul{\G}(\infty))(\ul{\G}^*(t)-\ul{\G}^*(\infty))]}{\dd t}\nonumber\\
		=&\, -\frac{1}{2N}\sum_{\alpha i}\bb E \dot A_{\alpha i}(t)(\G^2)_{i\alpha }(t)(\ul{\G}^*(t)-\ul{\G}^*(\infty)) -\frac{1}{2N}\sum_{i\alpha}\bb E \dot A_{i\alpha }(t)(\G^2)_{\alpha i}(t)(\ul{\G}^*(t)-\ul{\G}^*(\infty))\nonumber\\
		&\,-\frac{1}{2N}\sum_{\alpha i}\bb E \dot A_{\alpha i}(t)(\G^{*2})_{i\alpha }(t)(\ul{\G}(t)-\ul{\G}(\infty)) -\frac{1}{2N}\sum_{i\alpha}\bb E \dot A_{i\alpha }(t)(\G^{*2})_{\alpha i}(t)(\ul{\G}(t)-\ul{\G}(\infty))\label{7.13}\,.
	\end{align}
Due to symmetry, we only look at the first term on RHS of \eqref{7.13}. Similar to \eqref{7.4}, we get
\begin{equation}
\begin{aligned}
&-\frac{1}{2N}\sum_{\alpha i}\bb E \dot A_{\alpha i}(t)(\G^2)_{i\alpha }(t)(\ul{\G}^*(t)-\ul{\G}^*(\infty))\\
=&\,\frac{\e^{-t/2}}{4N}\sum_{r=2}^{\ell}\sum_{\alpha i}\frac{1}{r!}\cal C_{r+1}(B_{\alpha i}) \bb E \widetilde{\partial}_{\alpha i}^r [(\G^2)_{i\alpha}(t)(\ul{\G}^*(t)-\ul{\G}^*(\infty))]+O(\e^{-r/2}N^{-1})\,,
\end{aligned}
\end{equation}
where we abbreviate $\widetilde{\partial}_{\i\j} F=\partial F/\partial H_{\i\j}(0)$. The rest of the proof is essentially the same to that of \eqref{blb}: we first apply \eqref{diff3}, and then explore the index matching through the resolvent identity and Lemma \ref{cumulant}. We omit the details.
\end{proof}

By Lemma \ref{lemma5.3}, we see that
\begin{equation} \label{7.15}
	\bb E \ul{\G}(0)-\bb E \ul{\G}(\infty)\prec \frac{\eta^{1/3}}{q}+\frac{1}{N\eta q}
\end{equation}
as well as
\begin{equation} \label{7.16}
	\bb E |\ul{\G}(0)-\ul{\G}(\infty)|^2 \prec \frac{1}{N\eta^{2/3}q}+\frac{1}{N^2\eta^2q}\,.
\end{equation}

To interpolate between $A(\infty)$ and $W$, let $\b e_1=(1,0,...,0)^* \in \bb R^N$, and set $\widehat{W}\deq W+f \b e_1^*\b e_1$. By the invariance of $W$, see that $\widehat{W}$ and $A(\infty)$ have the same spectral distribution. For $w \in \bb C$ and $\eta>0$, we denote the Hermitization of $\widehat{W}$ and its Green function by
\begin{equation} \label{astronaut}
\widehat{H}_w\deq
\begin{pmatrix}
	0&  \widehat{W}-w\\
	\widehat{W}-\bar{w}&0
\end{pmatrix}  \quad \mbox{and} \quad \widehat{G}\deq (\widehat{H}_w-\ii \eta)^{-1}
\end{equation}
respectively. In addition, we set
\begin{equation*} 
	{H}_w^W\deq
	\begin{pmatrix}
		0&  {W}-w\\
		{W}-\bar{w}&0
	\end{pmatrix}  \quad \mbox{and} \quad {G}^W\deq ({H}^W_w-\ii \eta)^{-1}
\end{equation*}
Next we compare $ \bb E\ul{\widehat{G}}=\bb E\ul{\G}(\infty)$ and $ \bb E\ul{G}^W$.

\begin{lemma} \label{lemma5.4}
	For $(w,\ii\eta)\in \b S_{\delta}^{(1)}$, we have
	\begin{equation} \label{6.12}
		\bb E \ul{\widehat{G}}-\bb E \ul{{G}}^W\prec \frac{1}{N\eta }\Big(\frac{1}{q}+\frac{1}{N^{1-\delta}\eta}\Big)
	\end{equation}
and
\begin{equation} \label{5.19}
	\bb E |\ul{\widehat{G}}-\ul{{G}}^W|^2 \prec\frac{1}{N^2\eta^2 }\Big(\frac{1}{q}+\frac{1}{N^{1-\delta}\eta}\Big)\,.
\end{equation}
\end{lemma}

\begin{proof}
As in Corollary \ref{cor4.2} (ii), it is easy to show that 
\begin{equation} \label{1113}
\ul{\widehat{G}}-m \prec \frac{1}{N\eta} \quad \mbox{and} \quad \ul{{G}}^W-m \prec \frac{1}{N\eta}\,.
\end{equation}	
Let us denote the eigenvalues and corresponding $L^2$-normalized eigenvectors of $\widehat{H}$ by $\pm \widehat{\sigma}_{1},...,\pm \widehat{\sigma}_{N}$ and $\begin{pmatrix}
	{\widehat{\b v}_{1}} \\  \pm \widehat{\b w}_{1}
\end{pmatrix},...,\begin{pmatrix}
	{\widehat{\b v}_{N}} \\  \pm \widehat{\b w}_{N}
\end{pmatrix}$ respectively. Here $\widehat{\sigma}_{1},...,\widehat{\sigma}_{N}\geq 0$, $\widehat{\sigma}_{1}=\max_{i}\widehat{\sigma}_{i}$, and $\widehat{\b v}_{i},\widehat{\b w}_{i} \in \bb C^{N}$ for all $i$. Similar to the proof of Lemma \ref{lemmadelocalization}, we can show that
\begin{equation} \label{ev}
	\|\widehat{\b v}_{1}-\b e_1/\sqrt{2}\|_{\infty}+\|\widehat{\b w}_{1}-\b e_1/\sqrt{2}\|_{\infty} \prec N^{-1/2}f^{-1}
\end{equation}
and
\begin{equation} \label{ev2}
	\max_{i\geq 2} \|\widehat{\b v}_{i}\|_\infty+\max_{i\geq 2} \|\widehat{\b w}_{i}\|_\infty \prec N^{-1/2}\,,
\end{equation}
where $\b e_1\deq (1,0,...,0)\in \bb R^N$. By the resolvent identity, we have
\begin{equation} \label{5.222}
		\ul{\widehat{G}}-\ul{{G}}^W=-\frac{f}{2N}\sum_{\i}{G}^W_{\i 1'}\widehat{G}_{1 \i}-\frac{f}{2N}\sum_{\i}{G}^W_{1\i}\widehat{G}_{\i 1'}=-\frac{f}{2N}(\widehat{G}G^W)_{11'}-\frac{f}{2N}(\widehat{G}G^W)_{1'1}\,,
\end{equation}
and
	\begin{equation} \label{resolvent}
		\delta_{1'\i}+\ii \eta \widehat{G}_{1'\i}=({H}^W\widehat{G})_{1'\i}-\bar{w}\widehat{G}_{1\i}+f\widehat{G}_{1\i}\,,
	\end{equation}
where $H^W\deq H^W_0$. From \eqref{resolvent} and Stein's formula, we know that
	\begin{align}
	&(f-\bar{w})\bb E (\widehat{G}G^W)_{11}=\bb E\Big( G^W_{1'1}+\ii \eta (\widehat{G}G^W)_{1'1}-\sum_{i}{H}^W_{1'i}(\widehat{G}G^W)_{i1}\Big)\label{12345}\\
	=\,&\bb E\Big( G^W_{1'1}+\ii \eta (\widehat{G}G^W)_{1'1}+N^{-1}\sum_{i}\big(\widehat{G}_{i1'}(\widehat{G}G^W)_{i1}+\widehat{G}_{ii}(\widehat{G}G^W)_{1'1}+(\widehat{G}G^W)_{i1'}G^W_{i1}+(\widehat{G}G^W)_{ii}G^W_{1'1}\big)\Big)\,.\nonumber
	\end{align}
For any indices $\i,\j$, using Cauchy-Schwarz and the Ward identity, we can get 
\[
\eta (\widehat{G}G^W)_{\i\j} \prec  (\im \widehat{G}_{\i\i} \im G^W_{\j\j})^{1/2}\,.
\]
As $\sigma_1\asymp f$, we can use spectral decomposition and \eqref{1113} -- \eqref{ev2} to show that
\begin{equation}\label{oasis}
\im \widehat{G}_{\i\i} \prec f^{-1}+\im \ul{\widehat{G}} \prec f^{-1}+(N\eta)^{-1}+|m|\prec  f^{-1}+(N^{1-\delta}\eta)^{-1} \,.
\end{equation}
Together with $\im G^W_{\j\j}\prec (N^{1-\delta}\eta)^{-1}$ we get 
\begin{equation}\label{oasis2}
	\eta (\widehat{G}G^W)_{\i\j} \prec f^{-1}+(N^{1-\delta}\eta)^{-1}\,.
\end{equation}
Inserting \eqref{oasis} and \eqref{oasis2} into \eqref{12345} we get
\[
(f-\bar{w})\bb E (\widehat{G}G^W)_{11'} \prec \eta^{-1}\big(f^{-1}+(N^{1-\delta}\eta)^{-1}\big)\,.
\]
Similarly, $(f-{w})\bb E (\widehat{G}G^W)_{1'1} \prec \eta^{-1}\big(f^{-1}+(N^{1-\delta}\eta)^{-1}\big)$. Inserting the above two estimates into \eqref{5.222}, and together with $|w|\leq 2$ and $f \asymp q$, we get \eqref{6.12} as desired. The proof of \eqref{5.19} follows in a similar fashion; we omit the details.
\end{proof}

\subsection{Proof of Theorem \ref{main theorem}}
In this section we prove the following result, which obviously implies Theorem \ref{main theorem}.

\begin{proposition} \label{prop5.5}
	Fix $k \in \bb N_+$ and $w_1,...,w_k\in \bb C$ with $|w_j|=1$ for all $j\in \{1,2,...,k\}$. Let $f_1,...,f_k :\bb C \to \bb C$ be smooth and compactly supported, independent of $N$, and set
	\[
	f_{j,w_j}(w)\deq N f_j(\sqrt{N}(w-w_j))
	\]
	for all $j \in \{1,2,...,k\}$. Let $\lambda_1,...,\lambda_N$ be the eigenvalues of $A$, and $\mu_1,...,\mu_N$ be the eigenvalues of $W$. Then
	\begin{equation*}
		\begin{aligned}
	\bb E \Bigg[ \prod_{j=1}^k \bigg(\frac{1}{N}\sum_{i=1}^N&f_{j,w_j}(\lambda_i)-\frac{1}{\pi}\int_{|w|\leq 1}f_{j,w_j}(w)\dd^2 w\bigg)\\
	&-\prod_{j=1}^k \bigg(\frac{1}{N}\sum_{i=1}^Nf_{j,w_j}(\mu_i)-\frac{1}{\pi}\int_{|w|\leq 1}f_{j,w_j}(w)\dd^2 w\bigg)\Bigg]=O(N^{-c})
		\end{aligned}
	\end{equation*}
for some constant $c\equiv c(\xi,k,f_1,...,f_k)>0$. 
\end{proposition}

Our starting point is Girko's Hermitization formula \cite{Gir84}, which reads
\begin{equation} \label{Gir}
\frac{1}{N}\sum_i f_{j,w_j}(\lambda_i)=\frac{1}{2\pi N}\sum_i\int_{\bb C}  \nabla^2 f_{j,w_j}(w) \log |\lambda_i-w|\, \dd^2 w\,,
\end{equation}
and
\begin{equation} \label{Herm}
	\sum_i \log |\lambda_i-w|=\frac{\ii}{2}\int_0^{\infty} \tr \widetilde{G}_{w}(\ii \eta) \dd \eta\,.
\end{equation}
We would like to replace the integral on RHS of \eqref{Gir} by a Riemann sum. To this end, fix $\delta\in (0,\xi/100)$, and let $K$ be an $N^{-10}$-net of the ring $\{w\in \bb C: ||w|-1|\leq N^{-1/2+\delta}\}$.  Let $U\deq \cup_i \cal B_{\lambda_i} (N^{-5})$, where $\cal B_z(r)$ denotes the ball centered at $z$ with radius $r$. As $\|\nabla^2 f_{j,w_j}(w)\|_\infty=O(N^2)$, we have
	\begin{align}
		&\,\frac{1}{2\pi N}\sum_i\int_{\bb C }  \nabla^2 f_{j,w_j}(w) \log |\lambda_i-w|\, \dd^2 w-\frac{1}{2\pi N^{21}}\sum_i  \sum_{w\in K\cap U^c} \nabla^2 f_{j,w_j}(w) \log |\lambda_i-w|\nonumber\\
		=&\, \frac{1}{2\pi N}\sum_i\int_{\bb C\cap U^c}  \nabla^2 f_{j,w_j}(w) \log |\lambda_i-w|\, \dd^2 w\label{5.322}\\
		&-\frac{1}{2\pi N^{21}}\sum_i  \sum_{w\in K\cap U^c} \nabla^2 f_{j,w_j}(w) \log |\lambda_i-w|+O(N^{-6})=O(N^{-6})\nonumber\,,
	\end{align}
where in the last step we used $\log |\lambda_i-w|-\log |\lambda_i-w'|=O(N^{5}|w-w'|)$ for all $w,w'\in K\cap U^c$, and
\begin{equation} \label{count}
	|K\cap U^c\cap \supp f_{j,w_j}|=O(N^{19})\,.
\end{equation}
By \eqref{Gir}, \eqref{Herm} and \eqref{5.322} we get
\begin{equation*}
	\begin{aligned}
		\frac{1}{N}\sum_i f_{j,w_j}(\lambda_i)&=\frac{1}{2\pi N^{21}}\sum_i  \sum_{w\in K\cap U^c} \nabla^2 f_{j,w_j}(w) \log |\lambda_i-w|+O(N^{-6})\\
		&=\frac{\ii }{4\pi N^{21}} 
		\sum_{w\in K\cap U^c} \int_0^{\infty} \nabla^2 f_{j,w_j}(w) \tr \G_w(\ii \eta) \dd \eta+O(N^{-6})\,.
	\end{aligned}
\end{equation*}
As a result
\begin{equation}  \label{6.4}
\frac{1}{N}\sum_i f_{j,w_j}(\lambda_i)-\int_{|w|\leq 1} f_{j,w_j}(w) \dd^2 w=	T_{j,1}+T_{j,2}+T_{j,3}+O(N^{-6})\,,
\end{equation}
where 
\[
T_{j,1}= \frac{\ii}{2\pi N^{20}}\sum_{w\in K\cap U^c} \int_0^{\eta_*} \nabla^2 f_{j,w_j}(w) \big(\ul{\G}_w(\ii\eta) - m(w,\ii\eta)\big)  \dd \eta
\]
\[
T_{j,2}= \frac{\ii}{2\pi N^{20}}\sum_{w\in K\cap U^c}  \int_{\eta_*}^{\eta^*} \nabla^2 f_{j,w_j}(w) \big(  \ul{\G}_w(\ii\eta) - m(w,\ii\eta)\big)  \dd \eta
\]
\[
T_{j,3}= \frac{\ii}{2\pi N^{20}}\sum_{w\in K\cap U^c}  \int_{\eta^*}^{\infty} \nabla^2 f_{j,w_j}(w) \big(  \ul{\G}_w(\ii\eta) - m(w,\ii\eta)\big)  \dd \eta
\]
with $\eta_*=N^{-3/4-\delta}$, $\eta^*=N^{-3/4+\delta}$, and fixed $\delta \in (0,\xi/100)$. 

The following lemma gives prior estimates for the RHS of \eqref{6.4}.

\begin{lemma} \label{lemma5.6}
We have
	\begin{equation} \label{6.5}
		|T_{j,1}|+ |T_{j,2}| \prec 1
	\end{equation}
and 
\begin{equation} \label{6.6}
	T_{j,3} \prec N^{-\delta}
\end{equation}
uniformly in $j$.
\end{lemma}

\begin{proof}
(i) Using Corollary \ref{cor4.2} (ii) and a deterministic monotonicity argument (see e.g.\,\cite[Section 10]{BK16}), we have
\begin{equation} \label{5.344}
	\ul{\G}-m \prec \frac{1}{N\eta}
\end{equation}	
for all $|1-|w||\leq N^{-1/2+\delta}$ and $0<\eta\leq N^{-3/4+\delta}$. Together with \eqref{count} we have
\begin{equation*}
	T_{j,2} \prec \frac{1}{N^{20}}\sum_{w\in K\cap U^c} \int_{\eta_*}^{\eta^*}|\nabla^2 f_{w_j}(w) |\frac{1}{N\eta} \dd \eta \prec 1\,.
\end{equation*}
To estimate $T_{j,1}$, first note that by \eqref{6.4} and a triangle inequality, we have the deterministic bound $|T_{j,1}|=O(N^{2})$. Applying \cite[Theorem 2.9]{TV08} with $\rho=p/(1-p)$, $x\overset{d}{=}\mbox{Bern}(1-p)$ and $M=-w\sqrt{(1-p)}I$, we have
\begin{equation} \label{345}
\mathbb P\Big(\max_{w\in K}\|\widetilde{H}_w^{-1}\|\geq N^{\log N}\Big)=O (N^{-D})
\end{equation}
for any fixed $D>0$. By \eqref{5.344}, we know that on the event $\Sigma\deq \{\max_{w\in K}\|\widetilde{H}_w^{-1}\|< N^{\log N} \}$ we have
\begin{equation}
	\begin{aligned}
T_{j,1} &\prec  \frac{1}{ N^{18}}\sum_{w\in K\cap U^c}\bigg|\int_0^{N^{-\log N}} \ul{\G}_w - m(w,\ii \eta)\, \dd \eta\bigg|+1\prec N\max_{w\in K\cap U^c} \bigg|\int_0^{N^{-\log N}} \ul{\G}_w \, \dd \eta\bigg|+1\\
&\prec  \max_{w\in K\cap U^c} \sum_{|\sigma_{i,w}|\geq N^{-\log N}}\log\bigg(1+\frac{1}{\sigma_{i,w}^2N^{2\log N}}\bigg)+1\prec 1\,,
	\end{aligned}
\end{equation}
where $\sigma_{i,w}$ are the eigenvalues of $\widetilde{H}_w$. This finishes the proof of \eqref{6.5}.
	
(ii) Note that by the trivial bound $\max_{\i\j}\|\G_w(\ii \eta)_{\i\j}\|\leq \eta^{-1}$, we have
\begin{align*}
	\int_{\eta^*}^{\infty} \nabla^2 f_{j,w_j}(w) \big(  \ul{\G}_w(\ii\eta) - m(w,\ii\eta)\big)  \dd \eta-\int_{\eta^*}^{\infty} \nabla^2 f_{j,w_j}(w') \big(  \ul{\G}_{w'}(\ii\eta) - m(w',\ii\eta)\big)  \dd \eta\\ =O\big(|w-w'|\big(N^2/\eta^*+N^{5/2}\big)\big)=O(N^3|w-w'|)
\end{align*}
for all $w,w'\in K$. As a result, we can recover an integration from the Riemann sum in $T_{j,3}$, i.e.
\[
	T_{j,3}= \frac{\ii}{2\pi }\int_{\bb C}  \int_{\eta^*}^{\infty} \nabla^2 f_{w_j}(w) \big(  \ul{\G} - m\big)  \dd \eta\, \dd^2 w+O(N^{-6})
\]
 Let $u\deq - \bar{w}m /(\ii \eta+m)$, and thus $1+\ii \eta m+m^2+wu=0$. It is east to check that
\begin{equation*}
	\partial_w m=-\frac{\ii}{2}\partial_\eta u
	\quad 
	\mbox{and}
	\quad
	\partial_w \ul{\G}=\frac{1}{2N}\sum_{i}(\G^2)_{i'i}=-\frac{\ii}{2N} \sum_{i}\partial_\eta \G_{i'i}\,.
\end{equation*}
Together with integration by parts, we get
	\begin{equation} \label{5.28}
		\begin{aligned}
			T_{j,3}&=-\frac{2 \ii}{\pi }\int_{\bb C}  \int_{\eta^*}^{\infty} \partial_{\bar w} f_{w_j}(w) \big(  \partial_w\ul{\G} - \partial_w m\big)  \dd \eta\, \dd^2 w+O(N^{-6})\\
			&=-\frac{1}{\pi}\int_{\bb C}  \int_{\eta^*}^{\infty}  \partial_{\bar w} f_{w_j}(w)\Big(\frac{1}{N}\sum_{i}\partial_\eta \G_{i'i}-\partial_\eta u \Big)\dd \eta\, \dd^2w+O(N^{-6})\\
			&=\frac{1}{\pi} \int_{\bb C} \partial_{\bar w} f_{w_j}(w)\Big(\frac{1}{N}\sum_i\G_{i'i}(\ii\eta^*)-u(\ii\eta^*) \Big)\dd^2 w+O(N^{-6})\,.
		\end{aligned}
	\end{equation}
By Proposition \ref{prop5.1} and $u(\ii\eta^*)=-(1+m^2)/w+O(\eta^*)$, we have
\[
T_{j,3} \prec N^{-1/2-\delta}\int_{\bb C} |\partial_{\bar w} f_{w_j}(w)|\dd^2 w +N^{-6} \prec N^{-\delta}\,.
\]
This finishes the proof of \eqref{6.6}.
\end{proof}

\begin{proof}[Proof of Proposition \ref{prop5.5}]
Now we deduce Proposition \ref{prop5.5} from Lemmas \ref{lemma 5.2} -- \ref{lemma5.4} and \ref{lemma5.6}. A similar strategy was used in \cite{CEK3}.

\textit{Step 1.} We first show that
\begin{equation} \label{T1}
	\bb E |T_{j,1}| \prec N^{-\delta/2}\,.
\end{equation}
Let $\Sigma\deq \{\max_{w\in K}\|\widetilde{H}_w^{-1}\|< N^{\log N} \}$, and \eqref{345} shows $\bb P(\Sigma^c)=O(N^{-D})$ for any fixed $D>0$. Thus by \eqref{um} we know that for all $w\in K$,
\begin{equation*}
	\begin{aligned}
		&\,\bigg|\b 1_\Sigma\int_0^{\eta_*}  \ul{\G}_w(\ii\eta) - m(w,\ii\eta)  \dd \eta\bigg| \\
		=&\,\,O(N^{-1})\b 1_\Sigma\sum_{|\sigma_{i,w}|\leq N^{-\log N}} \log \Big(1+\frac{\eta_*^2}{\sigma_{i,w}^2}\Big)+O(N^{-1})\sum_{|\sigma_{i,w}|\geq N^{-\log N}}\log \Big(1+\frac{\eta_*^2}{\sigma_{i,w}^2}\Big)+\bigg|\int_0^{\eta_*} m(w,\ii\eta)\dd \eta \bigg|\\
		=&\, O(N^{-1})\sum_{|\sigma_{i,w}|\geq N^{-\log N}}\log \Big(1+\frac{\eta_*^2}{\sigma_{i,w}^2}\Big)+O(N^{-1-\delta})\,,
	\end{aligned}
\end{equation*}
where $\sigma_{i,w}$ denotes the eigenvalues of $\widetilde{H}_w$, and in the second step we used \eqref{um}. By \eqref{6.4}, we have the deterministic bound $|T_{j,1}|=O(N^{2})$. Hence
\begin{equation} \label{5.30}
\hspace{-0.2cm}	\bb E |T_{j,1}| =\bb E |\b 1_\Sigma \,T_{j,1}|+O(N^{-1})\prec \frac{1}{N^{19}}
\sum_{w\in K\cap U^c } \sum_{|\sigma_{i,w}|\geq N^{-\log N}}\bb E\log \Big(1+\frac{\eta_*^2}{\sigma_{i,w}^2}\Big)\dd w+O(N^{-\delta})\,.
\end{equation}
Observe that
\begin{align}
&\quad \sum_{|\sigma_i|\geq N^{-\log N}}\bb E\log \Big(1+\frac{\eta_*^2}{\sigma_i^2}\Big)\prec \bb E|\{i:|\sigma_i|\leq N^{\delta/2}\eta_*\}|+\bb E \sum_{|\sigma_i|\geq N^{\delta/2}\eta_*} \frac{\eta_*^2}{\sigma_i^2}\nonumber\\
&\prec \bb E|\{i:|\sigma_i|\leq N^{\delta/2}\eta_*\}|+\bb E \sum_{|\sigma_i|\geq N^{\delta/2}\eta_*} \frac{\eta_*^2}{\sigma_i^2+(N^{\delta/2}\eta_*)^2}\label{7.30}\\
& \prec \bb E|\{i:|\sigma_i|\leq N^{\delta/2}\eta_*\}|+N^{1-\delta/2}\eta_*\bb E\im  \ul{\G}(\ii N^{\delta/2}\eta_*) \prec \bb E|\{i:|\sigma_i|\leq N^{\delta/2}\eta_*\}|+N^{-\delta}\nonumber\,,
\end{align}
where in the last step we used Corollary \ref{cor4.2} (ii). By \eqref{7.15} and \eqref{6.12}, we get
\begin{equation*}
	\bb E|\{i:|\sigma_i|\leq N^{\delta/2}\eta_*\}| \leq N\eta_*\bb E \im \ul{\G}(\ii N^{\delta/2}\eta_*) (0)=N\eta_*\bb E \im \ul{\widehat{G}}(\ii N^{\delta/2}\eta_*) (0)+O(N^{-\delta})=O(N^{-\delta/2})\,,
\end{equation*}
where in the last step we used \cite[(28)]{CEK3}. Combining the above with \eqref{5.30} and \eqref{7.30}, we get \eqref{T1} as desired.

\textit{Step 2.} Similar as in \eqref{6.4} let us define
\[
\widehat{T}_{j,1}= \frac{\ii}{2\pi N^{21}}\sum_{w\in K\cap U^c_*} \int_0^{\eta_*} \nabla^2 f_{j,w_j}(w) \big(\ul{\widehat{G}}(0) - m\big)  \dd \eta
\]
\[
\widehat{T}_{j,2}= \frac{\ii}{2\pi N^{21}}\sum_{w\in K\cap U^c_*}  \int_{\eta_*}^{\eta^*} \nabla^2 f_{j,w_j}(w) \big(\ul{\widehat{G}}(0) - m\big)  \dd \eta
\]
\[
\widehat{T}_{j,3}= \frac{\ii}{2\pi N^{21}}\sum_{w\in K\cap U^c_*}  \int_{\eta^*}^{\infty} \nabla^2 f_{j,w_j}(w) \big(\ul{\widehat{G}}(0) - m\big)  \dd \eta\,,
\]
where $U_*\deq \cup_i \cal B_{\lambda^W_i} (N^{-5})$, and $\lambda_i^W$ are the eigenvalues of $W$. It is conventional to check that Lemma \ref{lemma5.6} and \eqref{T1} remain valid for $T_{j,1},...,T_{j,3}$ replaced by $\widehat{T}_{j,1},...,\widehat{T}_{j,3}$. By Lemma \ref{lemma5.6} and \eqref{T1}, we get
		\begin{align}
		\hspace{-0.1cm}	&\bb E \bigg[ \prod_{j=1}^k \bigg(\frac{1}{N}\sum_{i=1}^Nf_{j,w_j}(\lambda_i)-\frac{1}{\pi}\int_{|w|\leq 1}f_{j,w_j}(w)\dd^2 w\bigg)-\prod_{j=1}^k \bigg(\frac{1}{N}\sum_{i=1}^Nf_{j,w_j}(\mu_i)-\frac{1}{\pi}\int_{|w|\leq 1}f_{j,w_j}(w)\dd^2 w\bigg)\bigg]\nonumber\\
		&=\bb E \Big[\prod_{j=1}^k T_{j,2}-\prod_{j=1}^k \widehat{T}_{j,2}\Big]+O_{\prec}(N^{-\delta/2})\prec \sum_{j=1}^k\bb E |T_{j,2}-\widehat{T}_{j,2}| +N^{-\delta/2}\label{5.32}\,.
	\end{align}
As $\|\nabla^2 f_{j,w_j}\|_\infty=O(N^2)$, and $|U|,|U_*|=O( N^{10})$, for each $j$,
\begin{equation} \label{5.33}
	\begin{aligned}
	&\bb E |T_{j,2}-\widehat{T}_{j,2}| \prec  N \max_{w \in K}\int_{\eta_*}^{\eta^*}  \,\bb E |\ul{\G}(0)-\widehat{G}(0)| \dd \eta +N^{-1}\\
	\prec&\,N \max_{w \in K}  \int_{\eta_*}^{\eta^*}\bb E |\ul{\G}(0)-{\G}(\infty)| \dd \eta +N \max_{w \in K} \int_{\eta_*}^{\eta^*}  \bb E |\ul{\widehat{G}}(1)-\widehat{G}(0)| \dd \eta+N^{-1} \\
	\prec&\, N \max_{w \in K}  \int_{\eta_*}^{\eta^*} \bigg(\frac{1}{N\eta^{2/3}q}+\frac{1}{N^2\eta^2q}+\frac{1}{N^{3-\delta}\eta^3}\bigg)^{1/2} \dd \eta +N^{-1}\prec N^{-\delta/2}\,,
	\end{aligned}
\end{equation}
where in the third step we used \eqref{7.16} and \eqref{5.19}. Combining \eqref{5.32} and \eqref{5.33} we conclude the proof of Proposition \ref{prop5.5}.
\end{proof}

\subsection{Proof of Theorem \ref{radius and delocalization}}
The upper bound in \eqref{rigidity} was proved in Corollary \ref{cor4.7}, and Theorem \ref{radius and delocalization} (ii) was proved at the end of Section \ref{sec4.2}. Now we only need to prove the lower bound in \eqref{rigidity}. The statement follows directly from the following local circular law near the spectral edge.

\begin{theorem} \label{thmcircularlaw}
Fix $\delta \in(0,\xi/100)$, and let $f : \bb C \to \bb C $ be smooth and compactly supported, independent of $N$. Fix $a\in (1/2-\delta/2,1/2]$ and let $w_* \in \bb C$ satisfy $1-N^{-1/2+\delta/2}\leq |w_*|\leq \delta^{-1}$. Let $\lambda_1,\lambda_2,...,\lambda_N$ be the eigenvalues of $A$ and $f_{w_*}(w)\deq N^{2a}f(N^a(w-w_*))$. Then
\[
\frac{1}{N}\sum_i f_{w_*}(\lambda_i)-\frac{1}{\pi}\int_{|w|\leq 1} f_{w_*}(z) \dd^2w \prec N^{2a-1}\,. 
\]
\begin{proof}
The steps are essentially identical to those of Lemma \ref{lemma5.6}. We omit the details.
\end{proof}
\end{theorem}

\subsubsection*{Conflict of Interest}

Not Applicable.

	{\small
	
	\bibliography{bibliography} 
	
	\bibliographystyle{amsplain}
}

\end{document}